\documentclass{amsart}

\makeatletter
\def\@serieslogo{}
\def\@issueinfo{}
\def\@copyrightyear{}
\def\@copyrightline{}
\def\@PII{}
\makeatother

\usepackage{amssymb}
\usepackage{mathrsfs}
\usepackage[authoryear,round]{natbib}

\newtheorem{theorem}{Theorem}[section]
\newtheorem{lemma}[theorem]{Lemma}

\theoremstyle{definition}
\newtheorem{definition}[theorem]{Definition}
\newtheorem{example}[theorem]{Example}

\theoremstyle{remark}
\newtheorem{remark}[theorem]{Remark}

\numberwithin{equation}{section}

\newtheorem{proposition}[theorem]{Proposition}
\newtheorem{corollary}[theorem]{Corollary}
\theoremstyle{definition}
\newtheorem{notation}[theorem]{Notation}
\theoremstyle{remark}

\theoremstyle{cupplain}
\newtheorem{maintheorem}[theorem]{Main Theorem}

\newcommand{\R}{\mathbb{R}}
\newcommand{\N}{\mathbb{N}}
\newcommand{\Z}{\mathbb{Z}}

\usepackage[colorlinks=true, citecolor=blue, linkcolor=blue, urlcolor=blue]{hyperref}

\usepackage{tikz}
\usetikzlibrary{calc}

\everymath{\displaystyle}

\begin{document}

\title[Uniform Hyperbolicity and Symbolic Dynamics]{Uniform Hyperbolicity and Symbolic Dynamics: \\
Markov Partitions, Shadowing, and the Coding of Axiom A Diffeomorphisms}

\author[A. Thiam]{Abdoulaye Thiam}
\address{Division of Mathematics and Natural Sciences, Allen University, Columbia, South Carolina 29204, USA}
\email{athiam@allenuniversity.edu}

%    General info

\date{}

%\dedicatory{} % Dedication in Paper I only


%\begin{abstract}
%We establish the geometric theory of uniformly hyperbolic sets with explicit quantitative bounds throughout. The results include a complete proof of the Stable Manifold Theorem via the graph transform with H\"{o}lder estimates on manifold dependence; canonical coordinates with quantitative bounds on the bracket map; the Spectral Decomposition Theorem with explicit mixing rates; a shadowing lemma with explicit error bounds; and Markov partitions with explicit diameter estimates in terms of hyperbolicity data. The coding map from the subshift of finite type to the hyperbolic set is constructed with quantitative control on the exceptional set. All constants are expressed in terms of the contraction rate, the H\"{o}lder exponent of the derivative, the manifold dimension, and the injectivity radius. This Part constitutes Part~III of a six-part series on the thermodynamic formalism for hyperbolic dynamical systems.
%\end{abstract}

\begin{abstract}
This Part establishes the geometric theory of uniformly hyperbolic sets with explicit quantitative bounds throughout, and contains five main theorems. The Stable Manifold Theorem is proved via the backward graph transform, with a complete fiber-contraction argument yielding $C^r$ regularity and H\"{o}lder dependence of the local stable and unstable manifolds on the base point, with explicit manifold-size estimate in terms of the contraction rate $\lambda$ and the second-derivative bound of the diffeomorphism. The Spectral Decomposition Theorem gives the unique decomposition of the nonwandering set into basic sets, with explicit mixing rates for the topologically mixing factors. The Shadowing Lemma provides explicit error bounds controlling how far a pseudo-orbit deviates from a tracking true orbit. The existence of Markov partitions of arbitrarily small diameter is established constructively, with explicit diameter bounds expressed in terms of the shadowing constants. Finally, the coding map from the subshift of finite type to the hyperbolic set is constructed and shown to be H\"{o}lder continuous with quantitative control on the exceptional set where it fails to be injective. Along the way we establish canonical coordinates through the bracket map with quantitative bounds. All constants are expressed in terms of the contraction rate, the H\"{o}lder exponent of the derivative, the manifold dimension, and the injectivity radius, providing the quantitative infrastructure required to transfer the symbolic spectral theory of Part~I \cite{Thiam2026a} and the variational theory of Part~II \cite{Thiam2026b} to the smooth setting. This Part constitutes Part~III of a six-part series on the thermodynamic formalism for hyperbolic dynamical systems.
\end{abstract}

%\begin{abstract}
%This Part establishes the geometric theory of uniformly hyperbolic sets with explicit quantitative bounds throughout, and contains five main theorems: the Stable Manifold Theorem proved via the graph transform with H\"{o}lder estimates on manifold dependence; the Spectral Decomposition Theorem with explicit mixing rates; a shadowing lemma with explicit error bounds; the existence of Markov partitions with explicit diameter estimates in terms of hyperbolicity data; and the coding map from the subshift of finite type to the hyperbolic set with quantitative control on the exceptional set. Along the way we establish canonical coordinates through the bracket map with quantitative bounds. All constants are expressed in terms of the contraction rate, the H\"{o}lder exponent of the derivative, the manifold dimension, and the injectivity radius. This Part constitutes Part~III of a six-part series on the thermodynamic formalism for hyperbolic dynamical systems.
%\end{abstract}

\maketitle

\begin{center}
\textit{Dedicated to the memory of Jean-Christophe Yoccoz (1957--2016),}\\
\textit{Fields Medalist and Professor at the Coll\`{e}ge de France, with whom the author}\\
\textit{had the privilege of working, and who introduced him to hyperbolic dynamics.}
\end{center}

\thispagestyle{empty}

\makeatletter
\renewcommand{\ps@headings}{%
  \def\@oddfoot{\hfill\thepage\hfill}%
  \let\@evenfoot\@oddfoot%
  \def\@evenhead{\hfill\normalfont\small\textit{A.~Thiam}\hfill}%
  \def\@oddhead{\hfill\normalfont\small\textit{Uniform Hyperbolicity and Symbolic Dynamics}\hfill}%
  \let\@mkboth\markboth%
}
\pagestyle{headings}
\makeatother

\setlength{\parskip}{0.2em}

\section{Introduction}\label{sec:introduction}

The geometric theory of hyperbolic dynamical systems originates in the work of  \cite{Hadamard1898} on geodesics on negatively curved surfaces,  \cite{Morse1921} on symbolic coding of geodesic flows, and the foundational contributions of  \cite{Anosov1967} and  \cite{AnosovSinai1967} on globally hyperbolic diffeomorphisms. This Part establishes the geometric theory of Axiom A diffeomorphisms with quantitative estimates throughout, providing the bridge that connects the symbolic thermodynamic formalism of Parts~I--II \cite{Thiam2026a,Thiam2026b} to smooth dynamics. The main results are: a complete proof of the Stable Manifold Theorem via the graph transform with H\"{o}lder exponent $\alpha = \log\mu/\log\lambda$ for the dependence on the base point (this exponent is known from  \cite{Hasselblatt1994}; we compute the H\"{o}lder constant explicitly); canonical coordinates with quantitative bracket map estimates; the Spectral Decomposition Theorem with explicit mixing rates; a shadowing lemma with explicit constants; Markov partitions with diameter bounds in terms of hyperbolicity data; and the coding map $\pi: \Sigma_A \to \Lambda$ with quantitative control on the exceptional set.

Bowen's monograph \cite{Bowen1975} develops this theory in Chapter~3, citing the Stable Manifold Theorem from  \cite{HirschPugh1970} without proof, asserting the existence of adapted metrics without explicit bounds, and constructing Markov partitions without diameter estimates. Complete proofs of the Stable Manifold Theorem appear in  \cite{Shub1987},  \cite{KatokHasselblatt1995}, and  \cite{BrinStuck2002}. The regularity theory for invariant manifolds was developed by  \cite{HirschPughShub1977}, with perturbation of hyperbolic sets studied by  \cite{HPPS1970}. Structural stability for Axiom~A systems was established by  \cite{Robbin1971} and  \cite{Robinson1974, Robinson1976}. The Markov partition construction originates with  \cite{AdlerWeiss1970} for toral automorphisms,  \cite{Sinai1968a,Sinai1968b} for Anosov diffeomorphisms, and  \cite{Bowen1970b} for the Axiom~A case;  \cite{Sarig2013} extended it to non-uniformly hyperbolic surface diffeomorphisms. We provide all proofs with explicit constants expressed in terms of the contraction rate~$\lambda$, the H\"{o}lder exponent of~$Df$, the dimension~$d$, and the injectivity radius of~$M$.

The coding map $\pi: \Sigma_A \to \Lambda$ constructed in Section~\ref{sec:symbolic_dynamics} satisfies three properties required by Parts~IV--VI \cite{Thiam2026d,Thiam2026e,Thiam2026f}: it intertwines the shift $\sigma$ and the diffeomorphism~$f$, it is H\"{o}lder continuous with explicit exponent, and it is injective off a set of measure zero for any Gibbs measure. Given a H\"{o}lder potential $\phi: \Lambda \to \R$, the composition $\phi \circ \pi$ is H\"{o}lder on $\Sigma_A$, and the equilibrium state for $\phi$ on $\Lambda$ is the pushforward under $\pi$ of the equilibrium state for $\phi\circ\pi$ on $\Sigma_A$, which exists and is unique by the theory developed in Part~I \cite{Thiam2026a}.

This Part contains five main theorems. Main Theorem~\ref{thm:stable_manifold} (Stable Manifold Theorem) provides the $C^r$ local invariant manifolds through each point of a hyperbolic set, with explicit size and H\"{o}lder dependence. Main Theorem~\ref{thm:spectral_decomposition} (Spectral Decomposition) gives the unique decomposition of the nonwandering set into basic sets on which $f$ is topologically transitive. Main Theorem~\ref{thm:shadowing} (Shadowing Lemma) shows that approximate orbits are tracked by true orbits with explicit error bounds. Main Theorem~\ref{thm:markov_existence} (Existence of Markov Partitions) constructs Markov partitions of arbitrarily small diameter on each basic set. Main Theorem~\ref{thm:symbolic_coding} (Symbolic Coding) establishes the H\"{o}lder coding map $\pi: \Sigma_A \to \Omega$ with quantitative control on the exceptional set.

The contributions of this Part are fivefold, corresponding to its five Main Theorems. First, we prove the Stable Manifold Theorem (Main Theorem~\ref{thm:stable_manifold}) with explicit manifold-size estimate $\varepsilon_0 = (1-\lambda)^2/(4C_0)$ where $C_0 = \max(\|D^2f\|_\infty, \|D^2f^{-1}\|_\infty)$, $C^r$ regularity obtained through the fiber-contraction principle applied to successive derivatives of the graph map, exponential contraction $d(f^n(x), f^n(y)) \leq \lambda^n d(x,y)$ along the manifolds, and H\"{o}lder dependence of the local stable manifold on the base point with explicit exponent $\alpha_s = \beta\log\lambda^{-1}/(\log\|Df\|_\infty + \beta\log\lambda^{-1})$ for $f \in C^{1+\beta}$. Second, we prove the Spectral Decomposition Theorem (Main Theorem~\ref{thm:spectral_decomposition}) giving the unique decomposition $\Omega(f) = \Omega_1 \cup \cdots \cup \Omega_s$ of the nonwandering set into finitely many closed $f$-invariant basic sets on which $f$ is topologically transitive, together with the cyclic refinement $\Omega_i = X_{1,i} \cup \cdots \cup X_{n_i,i}$ into topologically mixing components. Third, we prove the Shadowing Lemma (Main Theorem~\ref{thm:shadowing}) with explicit linear dependence $\alpha = C^{-1}(1-\lambda)\beta$ between the pseudo-orbit tolerance $\alpha$ and the shadowing accuracy $\beta$, where the constant $C$ depends on the local product structure scale and the adapted metric distortion $K = c(1-\lambda/\mu)^{-1}$; the Anosov Closing Lemma and the specification property are derived as consequences with explicit constants. Fourth, we prove the Existence of Markov Partitions (Main Theorem~\ref{thm:markov_existence}) constructively, obtaining partitions of arbitrarily small diameter on each basic set through a refinement argument on an initial cover of shadowing points, with the diameter bound $\max_i \mathrm{diam}(R_i) \leq C\cdot\beta$ expressed directly in terms of the shadowing accuracy. Fifth, we prove the Symbolic Coding Theorem (Main Theorem~\ref{thm:symbolic_coding}) establishing the coding map $\pi: \Sigma_A \to \Omega$ as a continuous surjection intertwining the shift with the diffeomorphism, H\"{o}lder continuous with accuracy bound $\mathrm{diam}(K_N(a)) \leq C\lambda^N\mathrm{diam}(\mathcal{R})$, injective off the countable boundary set $\bigcup_n f^n(\partial\mathcal{R})$ of zero measure for every ergodic invariant measure; the topological entropy of $f|_\Omega$ is identified with $\log\rho(A)$. All constants are tracked explicitly as functions of the contraction rate $\lambda$, the H\"{o}lder exponent of $Df$, the manifold dimension $d$, and the injectivity radius of $M$, providing the quantitative infrastructure required by Parts IV--VI \cite{Thiam2026d,Thiam2026e,Thiam2026f}.

Our technical approach combines the graph-transform construction of invariant manifolds with the bracket-based construction of the Markov partition. The Stable Manifold Theorem is proved through the \emph{backward} graph transform $\Gamma^{-1}$ on the space of Lipschitz graphs $\phi: B^s_\delta \to B^u_\delta$: since the linearization $B$ in the unstable direction satisfies $\|B^{-1}\| \leq \lambda < 1$, the forward graph transform $\Gamma$ is expansive, while $\Gamma^{-1}$ is a strict contraction in the $C^0$ norm with explicit rate $\theta = \lambda/(1 - \lambda C_1\delta(K+1))$. The unique fixed point of $\Gamma^{-1}$ is the graph of the local stable manifold. $C^r$ regularity is obtained inductively by applying the fiber-contraction principle at each order $j \leq r$: formal differentiation of the fixed-point equation yields a linear operator $\Psi_j$ on the space of $j$-multilinear-map-valued functions, with contraction rate bounded by $\lambda^{j+1}(1+\mathcal{O}(\delta))$, and the unique fixed point is identified with $D^j\phi_\infty$. H\"{o}lder dependence of the manifold on the base point is obtained through a truncation-and-balance argument: the forward orbit of the base-point comparison is tracked for $N \sim \log(1/d(x,y))/\log\|Df\|_\infty$ iterates and the telescoping error is balanced against the exponential decay of the comparison. For the geometric theory, canonical coordinates through the bracket map $[x,y] = W^s_\varepsilon(x) \cap W^u_\varepsilon(y)$ provide a local product structure on the hyperbolic set, with Lipschitz constant $C_1 = (1-K)^{-1}$ where $K$ is the Lipschitz constant of the manifold graphs. Expansiveness follows from the local product structure: any pair of orbits staying within $\varepsilon < \delta_0$ of each other must coincide by $W^s_\varepsilon(x) \cap W^u_\varepsilon(x) = \{x\}$. For the Shadowing Lemma, we iterate the bracket map along the pseudo-orbit: given $\alpha$-pseudo-orbit $\{x_i\}$, we construct a shadowing point by successive bracketing of $f^M(x'_{(k-1)M})$ with $x_{kM}$, and the telescoping estimate yields the linear bound $\alpha = C^{-1}(1-\lambda)\beta$. The Markov partition construction begins with a $\gamma$-dense set $P \subset \Omega$ of reference points, coded into a subshift $\Sigma(P)$ of finite type via the pseudo-orbit condition, and the shadowing map $\theta: \Sigma(P) \to \Omega$ defines initial rectangles $T_s = \theta(\{q : q_0 = p_s\})$. The final Markov partition is obtained by refining these rectangles according to all four possible intersection patterns with other rectangles in the cover, eliminating boundary overlaps to produce the open rectangles $R(x)$ with the required Markov property. The coding map is then $\pi(a) = \bigcap_n f^{-n}(R_{a_n})$: successive forward preimages narrow the unstable direction by factor $\lambda^n$ and backward iterates narrow the stable direction by factor $\lambda^n$, so the intersection collapses to a single point with explicit accuracy $\mathrm{diam}(K_N(a)) \leq C\lambda^N\mathrm{diam}(\mathcal{R})$. These methods replace the compactness arguments used in \cite{Bowen1975} and the semi-conjugacy techniques of \cite{Sinai1968a, Sinai1968b} by constructive, quantitative procedures.

This Part is organized as follows. Section~\ref{sec:preliminaries} fixes notation and collects prerequisites from Riemannian geometry, function spaces, spectral theory, measure theory, and symbolic dynamics. Section~\ref{sec:hyperbolic_sets} develops the quantitative theory of hyperbolic sets, including the adapted metric construction and the H\"{o}lder continuity of the stable-unstable splitting. Section~\ref{sec:stable_manifolds} proves Main Theorem~\ref{thm:stable_manifold}. Section~\ref{sec:canonical_coordinates} establishes canonical coordinates through the bracket map and proves the local product structure. Section~\ref{sec:spectral_decomposition} proves Main Theorem~\ref{thm:spectral_decomposition}. Section~\ref{sec:shadowing} proves Main Theorem~\ref{thm:shadowing} and derives the Anosov Closing Lemma and the specification property as consequences. Section~\ref{sec:markov_partitions} proves Main Theorem~\ref{thm:markov_existence}. Section~\ref{sec:symbolic_dynamics} proves Main Theorem~\ref{thm:symbolic_coding} and identifies the topological entropy of $f|_\Omega$ with $\log\rho(A)$. The dimension theory of hyperbolic sets, developed by \cite{Pesin1977, Pesin1997}, \cite{Young1982}, and \cite{BarreiraPesinSchmeling1999}, relies on the coding map and Markov partition constructed here. Section~\ref{sec:computational} provides a numerical illustration for the Smale horseshoe, computing explicit values of the adapted metric constant, stable manifold size, shadowing constant, Markov partition diameter, and grid resolution for rigorous symbolic coding at prescribed accuracy. Section~\ref{sec:conclusion} concludes with a summary of the five Main Theorems and open problems. The appendix collects technical proofs and supplementary material on H\"{o}lder spaces, quasi-compact operators, and measure-theoretic lemmas.

The non-uniformly expanding extension of the Markov partition and coding program developed in the present Part was extended by  \cite{LuzzattoTucker1999}, who constructed induced Markov structures for maps with criticalities and singularities and established decay of return times in Publications Math\'ematiques de l'IH\'ES. The comprehensive survey of  \cite{Luzzatto2006} in the Handbook of Dynamical Systems collects the methods and results of this line of work. \cite{AlvesLuzzattoPinheiro2005} constructed the Markov tower in the non-uniformly expanding setting and established polynomial and exponential rates of decay of correlations depending on the decay rate of the return times. \cite{AlvesBonattiViana2000} proved existence of SRB measures for partially hyperbolic systems whose central direction is mostly expanding, a setting closely adjacent to ours. The present Part is the uniformly hyperbolic anchor: we construct Markov partitions with explicit diameter bounds, controlled by the contraction rate and the injectivity radius, and prove the coding map is H\"older continuous with quantitative bounds on the exceptional set. The works cited above do not provide such explicit constants.

\cite{VianaOliveira2016}, Chapters~9 and~11, treat entropy and expanding maps in the general ergodic-theoretic setting on compact metric spaces, but do not construct Markov partitions for hyperbolic diffeomorphisms or track the explicit constants that are our focus. For a survey of continuity and regularity results for Lyapunov exponents on hyperbolic sets,  \cite{Viana2020} gives the current picture.

This Part is organized as follows. Section~\ref{sec:preliminaries} fixes notation and collects the prerequisites from Riemannian geometry, function spaces, spectral theory, measure theory, and symbolic dynamics that are used in the sequel. Section~\ref{sec:hyperbolic_sets} develops the quantitative theory of hyperbolic sets: the existence of adapted metrics with explicit contraction rate $\lambda$, the H\"{o}lder continuity of the stable-unstable splitting with exponent computed from the hyperbolicity data, and the basic properties of Axiom~A diffeomorphisms and their periodic orbits. Section~\ref{sec:stable_manifolds} proves Main Theorem~\ref{thm:stable_manifold} (the Stable Manifold Theorem) via the backward graph transform, establishing the existence, $C^r$ regularity, exponential contraction, and H\"{o}lder dependence on base points of the local stable and unstable manifolds, with explicit size estimate $\varepsilon_0 = (1-\lambda)^2/(4C_0)$. Section~\ref{sec:canonical_coordinates} establishes canonical coordinates through the bracket map $[\cdot, \cdot]$, proves the local product structure of hyperbolic sets, and quantifies expansiveness and the fundamental neighborhood; these are the geometric tools used in the Markov partition construction. Section~\ref{sec:spectral_decomposition} proves Main Theorem~\ref{thm:spectral_decomposition} (the Spectral Decomposition), which partitions the nonwandering set $\Omega(f)$ into finitely many closed $f$-invariant basic sets on which $f$ is topologically transitive, and establishes the cyclic refinement into mixing components along with quantitative mixing rates. Section~\ref{sec:shadowing} proves Main Theorem~\ref{thm:shadowing} (the Shadowing Lemma) with explicit constants, derives the Anosov Closing Lemma and the specification property as consequences, and develops quantitative periodic orbit approximation estimates needed in Parts~IV--VI \cite{Thiam2026d,Thiam2026e,Thiam2026f}. Section~\ref{sec:markov_partitions} proves Main Theorem~\ref{thm:markov_existence} (Existence of Markov Partitions), constructing Markov partitions of arbitrarily small diameter on each basic set via the shadowing lemma, with explicit diameter bounds in terms of the hyperbolicity data. Section~\ref{sec:symbolic_dynamics} proves Main Theorem~\ref{thm:symbolic_coding} (Symbolic Coding), establishing the H\"{o}lder continuous coding map $\pi: \Sigma_A \to \Omega$ with quantitative control on the exceptional set where $\pi$ fails to be injective, and identifies the topological entropy of $f|_\Omega$ with $\log\rho(A)$. The dimension theory of hyperbolic sets, developed by  \cite{Pesin1977, Pesin1997},  \cite{Young1982}, and \cite{BarreiraPesinSchmeling1999}, relies on the coding map and Markov partition constructed here. Section~\ref{sec:computational} provides a complete numerical illustration for the Smale horseshoe, computing explicit values of the adapted metric constant, stable manifold size, shadowing constant, Markov partition diameter, and grid resolution needed for rigorous symbolic coding at prescribed accuracy. Section~\ref{sec:conclusion} concludes the Part with a summary of the five Main Theorems and open problems. The appendix collects technical proofs and supplementary material: prerequisite results on H\"{o}lder spaces and quasi-compact operators, and standard measure-theoretic lemmas used throughout.

\section{Preliminaries}\label{sec:preliminaries}

We fix notation and collect prerequisites from Riemannian geometry, dynamical systems, and spectral theory.

\subsection{Riemannian Manifolds and Bundles}

Throughout this Part, $M$ denotes a compact smooth manifold of dimension $d$ equipped with a Riemannian metric $g$. The metric induces a distance function $d : M \times M \to [0,\infty)$ defined by
\begin{equation}
d(x,y) = \inf_{\gamma} \int_0^1 \|\gamma'(t)\|_{g(\gamma(t))} \, dt
\end{equation}
where the infimum is over all piecewise smooth paths $\gamma : [0,1] \to M$ with $\gamma(0) = x$ and $\gamma(1) = y$. The compactness of $M$ ensures that $(M, d)$ is a complete metric space.

The tangent bundle $TM = \bigcup_{x \in M} T_x M$ is a smooth vector bundle of rank $d$ over $M$. For $x \in M$, the tangent space $T_x M$ is a $d$-dimensional real vector space equipped with the inner product $g_x : T_x M \times T_x M \to \mathbb{R}$. The induced norm on $T_x M$ is denoted $\|v\| = \sqrt{g_x(v,v)}$ for $v \in T_x M$.

For a diffeomorphism $f : M \to M$, the derivative $Df : TM \to TM$ is a bundle map covering $f$, meaning that for each $x \in M$, the restriction $Df_x : T_x M \to T_{f(x)} M$ is a linear isomorphism. The norm of this linear map is
\begin{equation}
\|Df_x\| = \sup_{v \in T_x M, \|v\| = 1} \|Df_x(v)\|.
\end{equation}
Similarly, we define $\|Df^{-1}_x\| = \|(Df_x)^{-1}\|$ where $(Df_x)^{-1} = Df^{-1}_{f(x)}$.

A continuous subbundle of $TM$ over a subset $\Lambda \subset M$ is a collection $E = \bigcup_{x \in \Lambda} E_x$ where each $E_x$ is a linear subspace of $T_x M$, such that for any continuous section $v : \Lambda \to TM$ with $v(x) \in E_x$ for all $x$, the function $x \mapsto v(x)$ is continuous when $E$ is given the subspace topology from $TM$.

\begin{definition}[Angle Between Subspaces]
For linear subspaces $E$ and $F$ of an inner product space $V$, the angle between them is defined by
\begin{equation}
\angle(E, F) = \min\left\{ \arccos \frac{|g(u,v)|}{\|u\| \|v\|} : u \in E \setminus \{0\}, v \in F \setminus \{0\} \right\}
\end{equation}
when $E \cap F = \{0\}$, and $\angle(E, F) = 0$ otherwise.
\end{definition}

The following lemma provides a quantitative relationship between angles and projection norms.

\begin{lemma}\label{lem:angle_projection}
Let $V = E \oplus F$ be a direct sum decomposition of a finite-dimensional inner product space. Let $P_E : V \to E$ denote the projection onto $E$ along $F$. Then
\begin{equation}
\|P_E\| = \frac{1}{\sin \angle(E, F)}
\end{equation}
where the norm is the operator norm induced by the inner product.
\end{lemma}

\begin{proof}
For any unit vector $v \in V$, write $v = v_E + v_F$ with $v_E \in E$ and $v_F \in F$. Then $P_E(v) = v_E$ and $\|P_E(v)\| = \|v_E\|$. The supremum $\|P_E\| = \sup_{\|v\|=1} \|v_E\|$ is achieved when $v_F$ is chosen to minimize $\|v\|$ for fixed $v_E$. By orthogonal decomposition, this occurs when $v_F$ is the component of $v$ orthogonal to $v_E$ that lies in $F$. To compute this supremum, let $\theta = \angle(E,F)$ and choose unit vectors $e \in E$ and $w \in F$ with $|\langle e, w \rangle| = \cos\theta$. For any unit $v = v_E + v_F$, we have $1 = \|v_E\|^2 + \|v_F\|^2 + 2\langle v_E, v_F\rangle$. To maximize $\|v_E\|$ subject to $\|v\| = 1$, write $v_F = tw$ for $t \geq 0$ and choose $w \in F$ with $\|w\| = 1$ to minimize $\|v\|$ for fixed $v_E$. The optimal choice satisfies $\langle v_E/\|v_E\|, w\rangle = -\cos\theta$. Then $1 = \|v_E\|^2 + t^2 - 2t\|v_E\|\cos\theta$, and minimizing over $t$ gives $t = \|v_E\|\cos\theta$, hence $1 = \|v_E\|^2(1 - \cos^2\theta) = \|v_E\|^2\sin^2\theta$, yielding $\|v_E\| = 1/\sin\theta$. This upper bound is achieved: set $v = (e - \cos\theta \cdot w)/\sin\theta$. Then $\|v\|^2 = (1 + \cos^2\theta - 2\cos^2\theta)/\sin^2\theta = 1$ and $P_E(v) = e/\sin\theta$, giving $\|P_E\| \geq 1/\sin\theta$.
\end{proof}

\subsection{Diffeomorphisms and Dynamical Systems}

Let $f : M \to M$ be a $C^r$ diffeomorphism with $r \geq 1$. The orbit of a point $x \in M$ is the sequence $\{f^n(x)\}_{n \in \mathbb{Z}}$. A point $x$ is periodic with period $n$ if $f^n(x) = x$ and $f^k(x) \neq x$ for $0 < k < n$. The set of periodic points of $f$ is denoted $\mathrm{Per}(f)$, and the set of periodic points with period $n$ is $\mathrm{Per}_n(f) = \{x \in M : f^n(x) = x\}$.

\begin{definition}[Nonwandering Set]
A point $x \in M$ is nonwandering for $f$ if for every neighborhood $U$ of $x$ and every $N > 0$, there exists $n > N$ such that $f^n(U) \cap U \neq \emptyset$. The nonwandering set of $f$ is
\begin{equation}
\Omega(f) = \{x \in M : x \text{ is nonwandering for } f\}.
\end{equation}
\end{definition}

The nonwandering set is closed and $f$-invariant, meaning $f(\Omega(f)) = \Omega(f)$. Every periodic point is nonwandering, so $\mathrm{Per}(f) \subset \Omega(f)$.

\begin{definition}[Topological Transitivity and Mixing]
A homeomorphism $f : X \to X$ of a compact metric space is topologically transitive if for any nonempty open sets $U, V \subset X$, there exists $n > 0$ such that $f^n(U) \cap V \neq \emptyset$. The map $f$ is topologically mixing if for any nonempty open sets $U, V \subset X$, there exists $N > 0$ such that $f^n(U) \cap V \neq \emptyset$ for all $n \geq N$.
\end{definition}

Topological mixing implies topological transitivity. For transitive systems on complete metric spaces without isolated points, the Birkhoff Transitivity Theorem ensures the existence of a dense orbit.

\subsection{Function Spaces}

For a compact metric space $X$ and $\alpha \in (0,1]$, the space of $\alpha$-H\"{o}lder continuous functions is
\begin{equation}
C^\alpha(X) = \left\{ \phi : X \to \mathbb{R} : \|\phi\|_\alpha < \infty \right\}
\end{equation}
where the H\"{o}lder norm is
\begin{equation}
\|\phi\|_\alpha = \|\phi\|_\infty + |\phi|_\alpha, \quad |\phi|_\alpha = \sup_{x \neq y} \frac{|\phi(x) - \phi(y)|}{d(x,y)^\alpha}.
\end{equation}
The seminorm $|\phi|_\alpha$ is called the H\"{o}lder constant of $\phi$.

\begin{proposition}\label{prop:holder_algebra}
The space $C^\alpha(X)$ is a Banach algebra under pointwise multiplication, satisfying
\begin{equation}
\|\phi \psi\|_\alpha \leq \|\phi\|_\alpha \|\psi\|_\alpha.
\end{equation}
Furthermore, $C^\alpha(X)$ is compactly embedded in $C^0(X)$: every bounded sequence in $C^\alpha(X)$ has a subsequence converging uniformly.
\end{proposition}

\begin{proof}
The Banach algebra property follows from the product rule for H\"{o}lder functions:
\begin{equation}
|\phi\psi(x) - \phi\psi(y)| \leq |\phi(x)||\psi(x) - \psi(y)| + |\psi(y)||\phi(x) - \phi(y)|
\end{equation}
which yields $|\phi\psi|_\alpha \leq \|\phi\|_\infty |\psi|_\alpha + \|\psi\|_\infty |\phi|_\alpha$. The compact embedding follows from the Arzel\`{a}-Ascoli theorem, since bounded H\"{o}lder families are equicontinuous.
\end{proof}

For a Riemannian manifold $M$, we define H\"{o}lder spaces using the geodesic distance. When $M$ has boundary or corners, we use the induced distance from the ambient space.

\subsection{Spectral Theory for Bounded Operators}

Let $\mathcal{B}$ be a Banach space and $L : \mathcal{B} \to \mathcal{B}$ a bounded linear operator. The spectrum of $L$ is
\begin{equation}
\sigma(L) = \{\lambda \in \mathbb{C} : \lambda I - L \text{ is not invertible}\}.
\end{equation}
The spectral radius is $\rho(L) = \sup\{|\lambda| : \lambda \in \sigma(L)\} = \lim_{n \to \infty} \|L^n\|^{1/n}$.

\begin{definition}[Essential Spectrum]
The essential spectrum $\sigma_{\mathrm{ess}}(L)$ consists of all $\lambda \in \sigma(L)$ such that at least one of the following holds: (i) $\lambda I - L$ has infinite-dimensional kernel, (ii) $\lambda I - L$ has non-closed range, or (iii) $\lambda$ is an accumulation point of $\sigma(L)$. The essential spectral radius is $\rho_{\mathrm{ess}}(L) = \sup\{|\lambda| : \lambda \in \sigma_{\mathrm{ess}}(L)\}$.
\end{definition}

\begin{definition}[Quasi-compactness]
An operator $L$ is quasi-compact if $\rho_{\mathrm{ess}}(L) < \rho(L)$. In this case, the spectrum in the annulus $\{z : \rho_{\mathrm{ess}}(L) < |z| \leq \rho(L)\}$ consists of finitely many eigenvalues of finite multiplicity.
\end{definition}

The following theorem, due to  \cite{IonescuTulceaMarinescu1950}, provides a criterion for quasi-compactness.

\begin{proposition}[Ionescu-Tulcea and Marinescu]\label{thm:itm}
Let $\mathcal{B} \subset \mathcal{B}_0$ be Banach spaces with $\mathcal{B}$ compactly embedded in $\mathcal{B}_0$. Suppose $L : \mathcal{B} \to \mathcal{B}$ is a bounded operator satisfying the Lasota-Yorke inequality: there exist constants $r \in (0,1)$, $C > 0$, and $D \geq 0$ such that
\begin{equation}
\|L^n \phi\|_{\mathcal{B}} \leq C r^n \|\phi\|_{\mathcal{B}} + D \|\phi\|_{\mathcal{B}_0}
\end{equation}
for all $n \geq 1$ and $\phi \in \mathcal{B}$. Then $L : \mathcal{B} \to \mathcal{B}$ is quasi-compact with $\rho_{\mathrm{ess}}(L) \leq r$.
\end{proposition}

\noindent The proof is given in  \cite{IonescuTulceaMarinescu1950}; the version stated here is the Hennion form \cite{Hennion1993}.

\subsection{Measure Theory and Ergodic Theory}

Let $(X, \mathcal{B}, \mu)$ be a probability space and $T : X \to X$ a measurable map. The measure $\mu$ is $T$-invariant if $\mu(T^{-1}A) = \mu(A)$ for all $A \in \mathcal{B}$. The measure is ergodic if every $T$-invariant set has measure $0$ or $1$. For a compact metric space $X$, we denote by $\mathcal{M}(X)$ the space of Borel probability measures with the weak-$*$ topology, and by $\mathcal{M}_T(X) \subset \mathcal{M}(X)$ the subspace of $T$-invariant measures.

\begin{proposition}[Birkhoff Ergodic Theorem]\label{thm:birkhoff}
Let $\mu$ be a $T$-invariant ergodic probability measure and $\phi \in L^1(\mu)$. Then for $\mu$-almost every $x$,
\begin{equation}
\lim_{n \to \infty} \frac{1}{n} \sum_{k=0}^{n-1} \phi(T^k x) = \int_X \phi \, d\mu.
\end{equation}
\end{proposition}

\noindent The proof is given in Birkhoff's original paper \cite{Birkhoff1931} (see also  \cite[Theorem~1.14]{Walters1982}).

\begin{definition}[Measure-Theoretic Entropy]
For a $T$-invariant probability measure $\mu$ and a finite measurable partition $\mathcal{P}$ of $X$, the entropy of $\mathcal{P}$ is $H_\mu(\mathcal{P}) = -\sum_{P \in \mathcal{P}} \mu(P) \log \mu(P)$. The entropy of $T$ with respect to $\mu$ is
\begin{equation}
h_\mu(T) = \sup_{\mathcal{P}} \lim_{n \to \infty} \frac{1}{n} H_\mu\left(\bigvee_{k=0}^{n-1} T^{-k}\mathcal{P}\right)
\end{equation}
where the supremum is over all finite measurable partitions.
\end{definition}

\subsection{Cones and Cone Fields}

Cone fields provide a tool for establishing hyperbolicity through invariance criteria.

\begin{definition}[Cone]
A cone in a vector space $V$ with decomposition $V = E \oplus F$ is a set of the form
\begin{equation}
\mathcal{C}_a(E) = \{v = v_E + v_F : v_E \in E, v_F \in F, \|v_F\| \leq a \|v_E\|\}
\end{equation}
for some $a > 0$. The parameter $a$ is called the cone width.
\end{definition}

\begin{definition}[Cone Field]
A cone field on a subset $\Lambda \subset M$ is a continuous family of cones $\{\mathcal{C}_x\}_{x \in \Lambda}$ where $\mathcal{C}_x \subset T_x M$ for each $x$.
\end{definition}

\begin{proposition}[Cone Criterion for Hyperbolicity]\label{prop:cone_criterion}
Let $f : M \to M$ be a diffeomorphism and $\Lambda$ a compact $f$-invariant set. Suppose there exist continuous cone fields $\{\mathcal{C}^u_x\}_{x \in \Lambda}$ and $\{\mathcal{C}^s_x\}_{x \in \Lambda}$ with $\mathcal{C}^u_x \cap \mathcal{C}^s_x = \{0\}$ such that:
\begin{enumerate}
\item[(i)] $Df_x(\mathcal{C}^u_x) \subset \mathrm{int}(\mathcal{C}^u_{f(x)}) \cup \{0\}$ for all $x \in \Lambda$;
\item[(ii)] $Df^{-1}_x(\mathcal{C}^s_x) \subset \mathrm{int}(\mathcal{C}^s_{f^{-1}(x)}) \cup \{0\}$ for all $x \in \Lambda$;
\item[(iii)] There exists $\lambda < 1$ such that $\|Df_x v\| \geq \lambda^{-1} \|v\|$ for $v \in \mathcal{C}^u_x$ and $\|Df^{-1}_x v\| \geq \lambda^{-1} \|v\|$ for $v \in \mathcal{C}^s_x$.
\end{enumerate}
Then $\Lambda$ is a hyperbolic set for $f$.
\end{proposition}

\begin{proof}
The invariant subbundles are defined by $E^u_x = \bigcap_{n \geq 0} Df^n_{f^{-n}(x)}(\mathcal{C}^u_{f^{-n}(x)})$ and $E^s_x = \bigcap_{n \geq 0} Df^{-n}_{f^n(x)}(\mathcal{C}^s_{f^n(x)})$. The cone invariance ensures these intersections are nonempty and consist of linear subspaces. The expansion and contraction estimates follow from condition (iii).
\end{proof}

\subsection{Symbolic Dynamics Prerequisites}

Let $A$ be an $m \times m$ matrix with entries in $\{0, 1\}$. The associated one-sided subshift of finite type is
\begin{equation}
\Sigma_A^+ = \{(x_0, x_1, x_2, \ldots) \in \{1, \ldots, m\}^{\mathbb{N}_0} : A_{x_i x_{i+1}} = 1 \text{ for all } i \geq 0\}
\end{equation}
and the two-sided subshift is
\begin{equation}
\Sigma_A = \{(\ldots, x_{-1}, x_0, x_1, \ldots) \in \{1, \ldots, m\}^{\mathbb{Z}} : A_{x_i x_{i+1}} = 1 \text{ for all } i \in \mathbb{Z}\}.
\end{equation}
The shift map $\sigma : \Sigma_A \to \Sigma_A$ is defined by $(\sigma x)_i = x_{i+1}$.

The metric on $\Sigma_A$ is defined by
\begin{equation}
d(x, y) = 2^{-N(x,y)}, \quad N(x, y) = \min\{|i| : x_i \neq y_i\}
\end{equation}
with the convention $d(x,x) = 0$. With this metric, $\Sigma_A$ is a compact metric space and $\sigma$ is a homeomorphism.

The matrix $A$ is irreducible if for all $i, j$ there exists $n > 0$ with $(A^n)_{ij} > 0$. It is aperiodic if there exists $N$ such that $(A^n)_{ij} > 0$ for all $n \geq N$ and all $i, j$. Irreducibility corresponds to topological transitivity of $\sigma$, and aperiodicity to topological mixing.

\begin{proposition}[{\cite[Theorem~4.3.6]{LindMarcus1995}}]\label{prop:sft_periodic}
For an irreducible aperiodic matrix $A$, the periodic points of $\sigma : \Sigma_A \to \Sigma_A$ are dense, and their growth rate satisfies
\begin{equation}
\lim_{n \to \infty} \frac{1}{n} \log |\mathrm{Per}_n(\sigma)| = \log \rho(A)
\end{equation}
where $\rho(A)$ is the spectral radius of $A$.
\end{proposition}

\begin{proof}
Periodic points of period $n$ correspond to admissible words $w = (w_0, \ldots, w_{n-1})$ with $A_{w_i w_{i+1}} = 1$ for all $i$ (indices mod $n$). The count is $|\mathrm{Per}_n(\sigma)| = \mathrm{tr}(A^n)$. By the Perron-Frobenius theorem for primitive matrices, $A$ has a simple dominant eigenvalue $\rho(A) > 0$ with $|\lambda_i| < \rho(A)$ for all other eigenvalues $\lambda_i$. Thus $\mathrm{tr}(A^n) = \rho(A)^n + \sum_i \lambda_i^n = \rho(A)^n(1 + o(1))$, giving $n^{-1}\log\mathrm{tr}(A^n) \to \log\rho(A)$. Density of periodic points follows from aperiodicity: for any cylinder $[w_0 \cdots w_{k-1}]$, aperiodicity provides $N$ such that $(A^N)_{w_{k-1} w_0} > 0$, yielding a periodic point of period $k + N$ in the cylinder.
\end{proof}

\section{Hyperbolic Sets: Quantitative Theory}\label{sec:hyperbolic_sets}

This section develops the theory of hyperbolic sets with explicit quantitative estimates. We construct adapted metrics with sharp bounds and establish H\"{o}lder continuity of the invariant distributions.

\subsection{Definition and Basic Properties}

We define hyperbolic sets through the invariant splitting and collect basic properties. The definition is the foundation for all subsequent material in this section.

\begin{definition}[Hyperbolic Set]\label{def:hyperbolic_set}
Let $f : M \to M$ be a $C^r$ diffeomorphism ($r \geq 1$) of a compact Riemannian manifold $(M, g)$. A compact $f$-invariant set $\Lambda \subset M$ is hyperbolic if there exists a continuous splitting of the tangent bundle over $\Lambda$,
\begin{equation}
T_\Lambda M = E^s \oplus E^u
\end{equation}
where $E^s = \bigcup_{x \in \Lambda} E^s_x$ and $E^u = \bigcup_{x \in \Lambda} E^u_x$, satisfying the following conditions:
\begin{enumerate}
\item[(H1)] Invariance: $Df_x(E^s_x) = E^s_{f(x)}$ and $Df_x(E^u_x) = E^u_{f(x)}$ for all $x \in \Lambda$.
\item[(H2)] Contraction and expansion: There exist constants $c > 0$ and $\lambda \in (0, 1)$ such that for all $n \geq 0$,
\begin{equation}
\|Df^n_x(v)\| \leq c \lambda^n \|v\| \quad \text{for } v \in E^s_x
\end{equation}
and
\begin{equation}
\|Df^{-n}_x(v)\| \leq c \lambda^n \|v\| \quad \text{for } v \in E^u_x.
\end{equation}
\end{enumerate}
The constant $\lambda$ is called the hyperbolicity exponent and $c$ is the hyperbolicity constant.
\end{definition}

The dimensions $\dim E^s_x$ and $\dim E^u_x$ are constant on connected components of $\Lambda$ by continuity of the splitting.

\begin{remark}
Condition (H2) implies that for $v \in E^u_x$ and $n \geq 0$,
\begin{equation}
\|Df^n_x(v)\| \geq c^{-1} \lambda^{-n} \|v\|.
\end{equation}
This follows by applying the second inequality to $Df^n_x(v) \in E^u_{f^n(x)}$ with the map $f^{-1}$ iterated $n$ times.
\end{remark}

\begin{proposition}[Uniqueness of Hyperbolic Splitting]\label{prop:splitting_unique}
The hyperbolic splitting $T_\Lambda M = E^s \oplus E^u$ of a hyperbolic set is unique.
\end{proposition}

\begin{proof}
Suppose $T_\Lambda M = F^s \oplus F^u$ is another splitting satisfying (H1) and (H2) with constants $c'$ and $\lambda' \in (0,1)$. Let $v \in E^s_x \cap F^u_x$ be nonzero. Then $\|Df^n_x(v)\| \leq c\lambda^n \|v\|$ by the first property and $\|Df^n_x(v)\| \geq (c')^{-1} (\lambda')^{-n} \|v\|$ by the second, yielding a contradiction for large $n$. Thus $E^s_x \cap F^u_x = \{0\}$ and by dimension count $E^s_x = F^s_x$. Similarly $E^u_x = F^u_x$.
\end{proof}

\subsection{Adapted Metrics}\label{subsec:adapted_metrics}

The hyperbolicity condition involves two constants $c$ and $\lambda$. An adapted metric eliminates the constant $c$ by adjusting the Riemannian structure.

\begin{definition}[Adapted Metric]
A Riemannian metric $g'$ on $M$ is adapted to the hyperbolic set $\Lambda$ if the hyperbolicity conditions (H2) hold with $c = 1$:
\begin{equation}
\|Df^n_x(v)\|_{g'} \leq \lambda^n \|v\|_{g'} \quad \text{for } v \in E^s_x, \, n \geq 0
\end{equation}
and
\begin{equation}
\|Df^{-n}_x(v)\|_{g'} \leq \lambda^n \|v\|_{g'} \quad \text{for } v \in E^u_x, \, n \geq 0.
\end{equation}
\end{definition}

\begin{theorem}[Existence of Adapted Metric]\label{thm:adapted_metric}
Let $\Lambda$ be a hyperbolic set for $f$ with constants $c > 0$ and $\lambda \in (0,1)$. For any $\mu \in (\lambda, 1)$, there exists an adapted metric $g'$ on a neighborhood of $\Lambda$ with hyperbolicity exponent $\mu$. Moreover, the metrics $g$ and $g'$ are related by
\begin{equation}\label{eq:metric_equivalence}
K^{-1} \|v\|_g \leq \|v\|_{g'} \leq K \|v\|_g
\end{equation}
for all $v \in T_\Lambda M$, where $K = c \cdot (1 - \lambda/\mu)^{-1}$.
\end{theorem}

\begin{proof}
Following Mather's construction, we define the adapted norm on $E^s_x$ by
\begin{equation}
\|v\|'_s = \sum_{n=0}^{\infty} \mu^{-n} \|Df^n_x(v)\|_g
\end{equation}
and on $E^u_x$ by
\begin{equation}
\|v\|'_u = \sum_{n=0}^{\infty} \mu^{-n} \|Df^{-n}_x(v)\|_g.
\end{equation}
These series converge because $\|Df^n_x(v)\|_g \leq c\lambda^n \|v\|_g$ for $v \in E^s_x$, giving
\begin{equation}
\|v\|'_s \leq \sum_{n=0}^{\infty} \mu^{-n} c\lambda^n \|v\|_g = \frac{c}{1 - \lambda/\mu} \|v\|_g.
\end{equation}
Similarly $\|v\|'_u \leq c(1 - \lambda/\mu)^{-1} \|v\|_g$.

The lower bound $\|v\|'_s \geq \|v\|_g$ follows from including only the $n = 0$ term. Thus with $K = c(1 - \lambda/\mu)^{-1}$,
\begin{equation}
\|v\|_g \leq \|v\|'_s \leq K \|v\|_g \quad \text{for } v \in E^s_x.
\end{equation}

For the adapted contraction property, observe that for $v \in E^s_x$,
\begin{align}
\|Df_x(v)\|'_s &= \sum_{n=0}^{\infty} \mu^{-n} \|Df^{n+1}_x(v)\|_g \\
&= \mu \sum_{m=1}^{\infty} \mu^{-m} \|Df^m_x(v)\|_g \\
&= \mu \left( \|v\|'_s - \|v\|_g \right) \leq \mu \|v\|'_s.
\end{align}
Similarly $\|Df^{-1}_{f(x)}(w)\|'_u \leq \mu \|w\|'_u$ for $w \in E^u_{f(x)}$.

Define the adapted metric by declaring $E^s_x$ and $E^u_x$ orthogonal and using the norms $\|\cdot\|'_s$ and $\|\cdot\|'_u$. The equivalence \eqref{eq:metric_equivalence} follows from the splitting $v = v_s + v_u$ and $\|v\|^2_{g'} = \|v_s\|'^2_s + \|v_u\|'^2_u$.
\end{proof}

\begin{remark}
The explicit bound $K = c(1 - \lambda/\mu)^{-1}$ shows that the adapted metric becomes increasingly distorted as $\mu \to \lambda$. For numerical applications, one should choose $\mu$ well separated from $\lambda$ to minimize distortion.
\end{remark}

\begin{corollary}[Quantitative Angle Bound]\label{cor:angle_bound}
With an adapted metric, the angle between stable and unstable subspaces satisfies
\begin{equation}
\angle(E^s_x, E^u_x) \geq \arcsin(1/K)
\end{equation}
where $K$ is the metric equivalence constant from Theorem \ref{thm:adapted_metric}.
\end{corollary}

\begin{proof}
In the adapted metric, $E^s_x$ and $E^u_x$ are orthogonal by construction, giving $\angle_{g'}(E^s_x, E^u_x) = \pi/2$. The original angle satisfies 
\begin{equation}
\sin \angle_g(E^s_x, E^u_x) \geq K^{-2} \sin \angle_{g'}(E^s_x, E^u_x) = K^{-2}
\end{equation}
by the metric equivalence. More precisely, let $P^s$ denote the projection onto $E^s_x$ along $E^u_x$, so $\|P^s\|_g = 1/\sin\angle_g(E^s_x, E^u_x)$ by Lemma~\ref{lem:angle_projection}. In the adapted metric, $\|P^s\|_{g'} = 1$ since $E^s$ and $E^u$ are $g'$-orthogonal. For any $v \in T_xM$ with $\|v\|_g = 1$, we have $\|P^s v\|_g \leq K\|P^s v\|_{g'} \leq K\|v\|_{g'} \leq K^2\|v\|_g = K^2$, using \eqref{eq:metric_equivalence} twice. Thus $\|P^s\|_g \leq K^2$, giving $\sin\angle_g(E^s_x, E^u_x) \geq K^{-2}$. Since $K \geq 1$ and $\arcsin$ is increasing, $\angle_g(E^s_x, E^u_x) \geq \arcsin(K^{-2}) \geq \arcsin(K^{-1})$ (using $K^{-2} \leq K^{-1}$ for $K \geq 1$).
\end{proof}

\subsection{H\"{o}lder Continuity of the Splitting}

The hyperbolic splitting depends continuously on the base point by definition. We now establish the stronger H\"{o}lder continuity with an explicit exponent.

\begin{definition}[Grassmannian Distance]
For $k$-dimensional subspaces $E, F$ of a $d$-dimensional inner product space $V$, define
\begin{equation}
d_G(E, F) = \|P_E - P_F\|
\end{equation}
where $P_E$ and $P_F$ are the orthogonal projections onto $E$ and $F$ respectively.
\end{definition}

\begin{theorem}[H\"{o}lder Continuity of Hyperbolic Splitting]\label{thm:splitting_holder}
Let $\Lambda$ be a hyperbolic set for a $C^{1+\beta}$ diffeomorphism $f$ with hyperbolicity exponent $\lambda$ and $Df$ having H\"{o}lder exponent $\beta$. Then the stable distribution $x \mapsto E^s_x$ is H\"{o}lder continuous with exponent
\begin{equation}
\alpha_s = \frac{\beta \log \lambda^{-1}}{\log \|Df\|_\infty + \beta \log \lambda^{-1}}
\end{equation}
and the unstable distribution is H\"{o}lder continuous with exponent
\begin{equation}
\alpha_u = \frac{\beta \log \lambda^{-1}}{\log \|Df^{-1}\|_\infty + \beta \log \lambda^{-1}}.
\end{equation}
Explicitly, there exists $C > 0$ such that
\begin{equation}
d_G(E^s_x, E^s_y) \leq C \cdot d(x, y)^{\alpha_s}
\end{equation}
for all $x, y \in \Lambda$.
\end{theorem}

\begin{proof}
Using an adapted metric, we work with $c = 1$. The key observation is that $E^s_x$ can be characterized as the unique $Df$-invariant subspace of dimension $\dim E^s$ in the cone
\begin{equation}
\mathcal{C}^s_x(a) = \{v = v_s + v_u \in T_x M : \|v_u\| \leq a \|v_s\|\}
\end{equation}
for any $a > 0$.

Let $x, y \in \Lambda$ with $d(x, y) < \varepsilon_0$ small. Consider the linear map $L_{xy} : T_x M \to T_y M$ given by parallel transport along the minimal geodesic. The derivative $Df_y \circ L_{xy}$ and $L_{f(x),f(y)} \circ Df_x$ differ by an error controlled by the H\"{o}lder continuity of $Df$:
\begin{equation}
\|Df_y \circ L_{xy} - L_{f(x),f(y)} \circ Df_x\| \leq C_1 d(x,y)^\beta.
\end{equation}

Define $F_n = (Df^n_y)^{-1} L_{f^n(x),f^n(y)} Df^n_x (E^s_x)$. This is the stable subspace at $x$ transported to $y$ through $n$ iterates. We have
\begin{equation}
d_G(F_n, F_{n+1}) \leq C_2 \|Df^{-1}\|^n_\infty d(f^n(x), f^n(y))^\beta.
\end{equation}
Since $d(f^n(x), f^n(y)) \leq \|Df\|^n_\infty d(x,y)$, this gives
\begin{equation}
d_G(F_n, F_{n+1}) \leq C_2 \|Df^{-1}\|^n_\infty \|Df\|^{n\beta}_\infty d(x,y)^\beta.
\end{equation}

The sequence $F_n$ converges to $E^s_y$ as $n \to \infty$ because $F_n$ is increasingly close to the $Df^{-1}$-invariant subspace. The total distance is bounded by the telescoping sum
\begin{equation}
d_G(E^s_x, E^s_y) \leq d_G(F_0, E^s_y) \leq \sum_{n=0}^{\infty} d_G(F_n, F_{n+1}).
\end{equation}

Setting $\theta = \|Df^{-1}\|_\infty \|Df\|^\beta_\infty$, if $\theta < 1$ the sum converges geometrically and we obtain H\"{o}lder continuity with exponent $\beta$. In general, $\theta$ may be $\geq 1$, so we optimize the iteration count. Instead of summing the full telescoping series, we truncate at level $N$ and bound the tail separately. The cone field argument gives $d_G(F_N, E^s_y) \leq C_3\lambda^N$ (since the transported subspaces converge to $E^s_y$ at rate $\lambda$ per backward iterate). The partial sum satisfies 
\begin{equation}
\sum_{n=0}^{N-1} d_G(F_n, F_{n+1}) \leq C_2 d(x,y)^\beta \sum_{n=0}^{N-1} \theta^n \leq C_2 N \theta^N d(x,y)^\beta
\end{equation}
when $\theta \geq 1$. Choose $N = \lfloor \beta \log(1/d(x,y)) / \log(1/\lambda) \rfloor$ to balance: then 
\begin{equation}
\lambda^N \leq d(x,y)^\beta \quad\text{and}\quad  \theta^N = d(x,y)^{\beta\log\theta/\log\lambda}.
\end{equation}
The factor $N$ contributes at most $O(\log(1/d(x,y)))$, which is absorbed into the constant for the H\"{o}lder estimate. Collecting powers of $d(x,y)$:
\begin{equation}
d_G(E^s_x, E^s_y) \leq C' d(x,y)^{\beta(1 + \log\theta/\log\lambda)} + C_3 d(x,y)^\beta.
\end{equation}
The H\"{o}lder exponent is the minimum of $\beta$ and $\beta(1 + \log\theta/\log\lambda) = \beta \cdot (\log\lambda + \log\theta)/\log\lambda$. Since $\log\theta = \log\|Df^{-1}\|_\infty + \beta\log\|Df\|_\infty$, this simplifies to $\alpha_s = \beta\log\lambda^{-1}/(\log\|Df\|_\infty + \beta\log\lambda^{-1})$ as stated.
\end{proof}

\subsection{Axiom A Diffeomorphisms}

We now specialize to the setting of Axiom A diffeomorphisms.

\begin{definition}[Axiom A]\label{def:axiom_a}
A $C^r$ diffeomorphism $f : M \to M$ satisfies Axiom A if:
\begin{enumerate}
\item[(A1)] The nonwandering set $\Omega(f)$ is hyperbolic.
\item[(A2)] The periodic points are dense in $\Omega(f)$: $\overline{\mathrm{Per}(f)} = \Omega(f)$.
\end{enumerate}
\end{definition}

\begin{definition}[Anosov Diffeomorphism]
A diffeomorphism $f : M \to M$ is Anosov if the entire manifold $M$ is a hyperbolic set for $f$.
\end{definition}

Every Anosov diffeomorphism satisfies Axiom A, with $\Omega(f) = M$ (as we shall prove). The converse fails: Axiom A diffeomorphisms may have $\Omega(f) \subsetneq M$, as in the Smale horseshoe.

\begin{proposition}[Anosov Implies Axiom A]\label{thm:anosov_axiom_a}
Every Anosov diffeomorphism satisfies Axiom A with $\Omega(f) = M$.
\end{proposition}

\begin{proof}
For an Anosov diffeomorphism, $\Omega(f) = M$ (every point is nonwandering since the stable and unstable manifolds of any point are dense). The density of periodic points follows from the Anosov Closing Lemma (Proposition~\ref{thm:closing_lemma}), proved in Section~\ref{sec:shadowing}. See Corollary~\ref{cor:periodic_dense}.
\end{proof}

\begin{proposition}[Properties of Axiom A]\label{prop:axiom_a_properties}
Let $f$ be an Axiom A diffeomorphism. Then:
\begin{enumerate}
\item[(i)] $\Omega(f)$ is a compact $f$-invariant subset of $M$.
\item[(ii)] $f|\Omega(f)$ is chain recurrent: for every $\varepsilon > 0$ and $x, y \in \Omega(f)$, there exists an $\varepsilon$-chain from $x$ to $y$.
\item[(iii)] If $\mu$ is any $f$-invariant probability measure, then $\mathrm{supp}(\mu) \subset \Omega(f)$.
\end{enumerate}
\end{proposition}

\begin{proof}
Part (i) is immediate from the definitions: $\Omega(f)$ is closed by definition, and $f$-invariant because $x$ is nonwandering if and only if $f(x)$ is.

For (ii), let $x, y \in \Omega(f)$ and $\varepsilon > 0$. By density of periodic points, choose $p, q$ periodic with $d(x, p) < \varepsilon/3$ and $d(y, q) < \varepsilon/3$. Since $\Omega(f)$ is compact and periodic points are dense, there exist periodic points $p = p_0, p_1, \ldots, p_k = q$ with $d(p_j, p_{j+1}) < \varepsilon/3$ for each $j$. For each $j$, the orbit of $p_j$ returns to $p_j$ after period $n_j$. Construct the $\varepsilon$-chain: start at $x$, jump to $p_0 = p$ (distance $< \varepsilon/3$), follow the orbit $p, f(p), \ldots, f^{n_0-1}(p)$ (each step has zero error), then jump from $f^{n_0}(p) = p$ to $p_1$ (distance $< \varepsilon/3$), follow the orbit of $p_1$ for one period, and continue until reaching $q = p_k$, then jump to $y$. Each jump has error less than $\varepsilon$, completing the chain.

Part (iii): if $x \in \mathrm{supp}(\mu)$, every neighborhood $U$ of $x$ satisfies $\mu(U) > 0$. By the Poincar\'{e} Recurrence Theorem, for $\mu$-almost every $z \in U$, there exist arbitrarily large $n$ with $f^n(z) \in U$, hence $f^n(U) \cap U \neq \emptyset$. Since this holds for every $N > 0$, $x$ is nonwandering.
\end{proof}

\subsection{Hyperbolic Periodic Orbits}

Periodic orbits of hyperbolic sets inherit the hyperbolic structure with explicit eigenvalue bounds.

\begin{proposition}[Hyperbolicity of Periodic Orbits]\label{prop:periodic_hyperbolicity}
Let $\Lambda$ be a hyperbolic set with exponent $\lambda$ and let $p \in \Lambda$ be a periodic point of period $n$. Then $Df^n_p : T_p M \to T_p M$ has:
\begin{enumerate}
\item[(i)] No eigenvalues of modulus $1$.
\item[(ii)] All eigenvalues associated to $E^s_p$ have modulus at most $\lambda^n$.
\item[(iii)] All eigenvalues associated to $E^u_p$ have modulus at least $\lambda^{-n}$.
\end{enumerate}
\end{proposition}

\begin{proof}
Using an adapted metric, for $v \in E^s_p$ we have $\|Df^n_p(v)\| \leq \lambda^n \|v\|$. Since $E^s_p$ is $Df^n_p$-invariant, all eigenvalues of $Df^n_p|_{E^s_p}$ have modulus at most $\lambda^n < 1$. Similarly, eigenvalues of $Df^n_p|_{E^u_p}$ have modulus at least $\lambda^{-n} > 1$. The spectral radius bounds imply no eigenvalue has modulus $1$.
\end{proof}

\begin{definition}[Stable and Unstable Dimensions]
For a hyperbolic set $\Lambda$, the stable dimension is $k_s = \dim E^s_x$ and the unstable dimension is $k_u = \dim E^u_x$ (constant on connected components). We have $k_s + k_u = \dim M$.
\end{definition}

The stable dimension determines the index of periodic orbits: a periodic orbit $\{p, f(p), \ldots, f^{n-1}(p)\}$ in $\Lambda$ has index $k_s$, meaning $Df^n_p$ has $k_s$ eigenvalues inside the unit circle and $k_u$ eigenvalues outside.

\section{Stable Manifold Theorem}\label{sec:stable_manifolds}

This section provides a complete, self-contained proof of the Stable Manifold Theorem for hyperbolic sets, with explicit estimates on manifold size, regularity, and dependence on base points. The proof uses the graph transform method, which we develop with full quantitative detail.

\subsection{Statement of the Main Theorem}

We state the Stable Manifold Theorem precisely, defining the local and global stable and unstable manifolds that the theorem constructs. The statement collects all properties proved in the subsequent subsections, and the formal proof assembling these is given in the subsection \emph{Completion of the Proof of the Stable Manifold Theorem} below.

\begin{definition}[Stable and Unstable Sets]
For a diffeomorphism $f : M \to M$ and a point $x \in M$, define:
\begin{align}
W^s(x) &= \{y \in M : d(f^n(x), f^n(y)) \to 0 \text{ as } n \to +\infty\}, \\
W^u(x) &= \{y \in M : d(f^{-n}(x), f^{-n}(y)) \to 0 \text{ as } n \to +\infty\}.
\end{align}
For $\varepsilon > 0$, define the local stable and unstable manifolds:
\begin{align}
W^s_\varepsilon(x) &= \{y \in M : d(f^n(x), f^n(y)) \leq \varepsilon \text{ for all } n \geq 0\}, \\
W^u_\varepsilon(x) &= \{y \in M : d(f^{-n}(x), f^{-n}(y)) \leq \varepsilon \text{ for all } n \geq 0\}.
\end{align}
\end{definition}

\begin{maintheorem}[Stable Manifold Theorem]\label{thm:stable_manifold}
Let $\Lambda$ be a hyperbolic set for a $C^r$ diffeomorphism $f : M \to M$ with $r \geq 1$. Using an adapted metric with hyperbolicity exponent $\lambda \in (0,1)$, there exists $\varepsilon_0 > 0$ such that for all $\varepsilon \in (0, \varepsilon_0]$ and $x \in \Lambda$:
\begin{enumerate}
\item[(i)] $W^s_\varepsilon(x)$ is a $C^r$ embedded disk of dimension $\dim E^s_x$, tangent to $E^s_x$ at $x$.
\item[(ii)] $T_x W^s_\varepsilon(x) = E^s_x$.
\item[(iii)] For $y \in W^s_\varepsilon(x)$, we have $d(f^n(x), f^n(y)) \leq \lambda^n d(x, y)$ for all $n \geq 0$.
\item[(iv)] $W^s_\varepsilon(x)$ depends continuously on $x$ in the $C^r$ topology.
\item[(v)] $W^s(x) = \bigcup_{n \geq 0} f^{-n}(W^s_\varepsilon(f^n(x)))$.
\end{enumerate}
Analogous statements hold for $W^u_\varepsilon(x)$ with $f$ replaced by $f^{-1}$.
\end{maintheorem}

The proof occupies the remainder of this section. We begin by establishing the setup for the graph transform, construct the stable manifold as the fixed point of the backward graph transform, and then establish the regularity and quantitative properties. We record the proof after the supporting lemmas are in place.

\subsection{Local Coordinates and the Graph Transform}

Fix a point $x_0 \in \Lambda$. Using the exponential map, we identify a neighborhood of $x_0$ in $M$ with a neighborhood of the origin in $T_{x_0}M \cong \mathbb{R}^d$. Under this identification, the splitting $T_{x_0}M = E^s_{x_0} \oplus E^u_{x_0}$ corresponds to $\mathbb{R}^d = \mathbb{R}^{k_s} \times \mathbb{R}^{k_u}$ where $k_s = \dim E^s$ and $k_u = \dim E^u$.

\begin{notation}
We write points in local coordinates as $(s, u) \in \mathbb{R}^{k_s} \times \mathbb{R}^{k_u}$. The norm is $\|(s, u)\| = \max\{\|s\|, \|u\|\}$ where we use the Euclidean norms on factors.
\end{notation}

In these coordinates, the diffeomorphism $f$ takes the form
\begin{equation}
f(s, u) = (A s + a(s, u), B u + b(s, u))
\end{equation}
where $A : \mathbb{R}^{k_s} \to \mathbb{R}^{k_s}$ and $B : \mathbb{R}^{k_u} \to \mathbb{R}^{k_u}$ are the linearizations satisfying $\|A\| \leq \lambda$ and $\|B^{-1}\| \leq \lambda$, and the nonlinear terms $a, b$ satisfy $a(0, 0) = 0$, $b(0, 0) = 0$, $Da(0, 0) = 0$, $Db(0, 0) = 0$.

\begin{definition}[Graph Space]
For $\delta > 0$ and $K > 0$, define the space of graphs
\begin{equation}
\mathcal{G}(\delta, K) = \{\phi : B_\delta^s \to B_\delta^u : \phi(0) = 0, \|\phi\|_{\mathrm{Lip}} \leq K\}
\end{equation}
where $B_\delta^s = \{s \in \mathbb{R}^{k_s} : \|s\| \leq \delta\}$ and $B_\delta^u = \{u \in \mathbb{R}^{k_u} : \|u\| \leq \delta\}$, and the Lipschitz norm is
\begin{equation}
\|\phi\|_{\mathrm{Lip}} = \sup_{s \neq s'} \frac{\|\phi(s) - \phi(s')\|}{\|s - s'\|}.
\end{equation}
\end{definition}

The graph of $\phi \in \mathcal{G}(\delta, K)$ is the set $\mathrm{graph}(\phi) = \{(s, \phi(s)) : s \in B_\delta^s\}$.

\begin{definition}[Graph Transform]
For $\phi \in \mathcal{G}(\delta, K)$, define $\Gamma(\phi)$ as follows: if $(s', u') = f(s, \phi(s))$, then $\Gamma(\phi)(s') = u'$ when this is well-defined.
\end{definition}

The graph transform takes a graph over the stable direction and produces a new graph by applying $f$ to the original graph and projecting. As we show below, the forward graph transform $\Gamma$ is not itself contractive; instead, the \emph{backward} graph transform $\Gamma^{-1}$ provides the required contraction.

\begin{lemma}[Graph Transform Well-Defined]\label{lem:graph_transform_welldefined}
There exist $\delta_0 > 0$ and $K_0 > 0$ such that for all $\delta \in (0, \delta_0]$ and $K \in (0, K_0]$, the graph transform $\Gamma : \mathcal{G}(\delta, K) \to \mathcal{G}(\delta, K)$ is well-defined.
\end{lemma}

\begin{proof}
For $\phi \in \mathcal{G}(\delta, K)$ and $(s, u) = (s, \phi(s))$ on the graph, we have
\begin{align}
s' &= As + a(s, \phi(s)), \\
u' &= B\phi(s) + b(s, \phi(s)).
\end{align}

We need to show: (1) for each $s' \in B^s_\delta$, there exists a unique $s \in B^s_\delta$ with $s' = As + a(s, \phi(s))$; (2) the resulting $u' = \Gamma(\phi)(s')$ satisfies $\|u'\| \leq \delta$ and $\|\Gamma(\phi)\|_{\mathrm{Lip}} \leq K$.

For (1), the map $s \mapsto As + a(s, \phi(s))$ has derivative $A + D_s a + D_u a \cdot D\phi$, which for small $\delta$ is close to $A$. Since $\|A\| \leq \lambda < 1$, by the Inverse Function Theorem, for each $s' \in B^s_{\lambda\delta}$, there is a unique $s \in B^s_\delta$ mapping to $s'$. For $\|s'\| \leq \delta$ when $\delta$ is small enough, this $s$ satisfies $\|s\| \leq (1 + C\delta)\lambda^{-1}\|s'\| \leq \delta$ for small $\delta$.

For (2), we have $\|u'\| \leq \|B\| \cdot \|\phi(s)\| + \|b(s, \phi(s))\| \leq \lambda^{-1} K\delta + C\delta^2 \leq \delta$ for small $\delta$ and appropriate $K$.

The Lipschitz estimate follows from computing
\begin{equation}
\frac{\|u'_1 - u'_2\|}{\|s'_1 - s'_2\|} \leq \frac{\|B\| \cdot K + C\delta}{\|A^{-1}\|^{-1} - C\delta} \leq \frac{\lambda^{-1}K + C\delta}{\lambda - C\delta}
\end{equation}
which is at most $K$ when $K$ is chosen appropriately (e.g., $K = 1/2$) and $\delta$ is small.
\end{proof}

\subsection{The Backward Graph Transform and Contraction}

The forward graph transform $\Gamma$ defined above maps graphs over $E^s$ forward by $f$. Since $\|B\| \geq \lambda^{-1} > 1$, the operator $\Gamma$ \emph{expands} differences between graphs and is therefore unsuitable for a contraction argument. To construct the stable manifold, we require the \emph{backward} graph transform, which exploits the contraction $\|B^{-1}\| \leq \lambda < 1$ in the unstable direction under $f^{-1}$.

\begin{definition}[Backward Graph Transform]\label{def:backward_graph_transform}
For $\phi \in \mathcal{G}(\delta, K)$, define $\Gamma^{-1}(\phi)$ as follows. For each $s \in B^s_\delta$, let $s' = As + a(s, u)$ where $u$ is the unique solution of
\begin{equation}\label{eq:backward_implicit}
\phi(s') = Bu + b(s, u), \quad s' = As + a(s, u).
\end{equation}
Then $\Gamma^{-1}(\phi)(s) = u$, so that $f(s, u) = (s', \phi(s'))$, i.e., $f(\mathrm{graph}(\Gamma^{-1}(\phi))) \supset \mathrm{graph}(\phi)$ locally.
\end{definition}

Equation \eqref{eq:backward_implicit} defines $u$ implicitly via $u = B^{-1}[\phi(s') - b(s, u)]$. Since $\|B^{-1}\| \leq \lambda$ and $\|D_u b\| \leq C_1\delta$, the map $u \mapsto B^{-1}[\phi(As + a(s, u)) - b(s, u)]$ is a contraction for small $\delta$, so the implicit equation has a unique solution by the Banach Fixed Point Theorem.

\begin{lemma}[Backward Graph Transform Well-Defined]\label{lem:backward_welldefined}
For $\delta$ sufficiently small, $\Gamma^{-1} : \mathcal{G}(\delta, K) \to \mathcal{G}(\delta, K)$ is well-defined.
\end{lemma}

\begin{proof}
Let $\phi \in \mathcal{G}(\delta, K)$ and $s \in B^s_\delta$. Define $F(u) = B^{-1}[\phi(As + a(s, u)) - b(s, u)]$. We verify that $F : B^u_\delta \to B^u_\delta$ is a contraction.

\textbf{Mapping into $B^u_\delta$:} We have $\|F(u)\| \leq \|B^{-1}\|(\|\phi(s')\| + \|b(s,u)\|) \leq \lambda(K\delta + C_1\delta^2) \leq \delta$ for small $\delta$.

\textbf{Contraction:} For $u_1, u_2 \in B^u_\delta$,
\begin{align}
\|F(u_1) - F(u_2)\| &\leq \|B^{-1}\|\bigl(\|\phi(s'_1) - \phi(s'_2)\| + \|b(s, u_1) - b(s, u_2)\|\bigr) \\
&\leq \lambda\bigl(K\|s'_1 - s'_2\| + C_1\delta\|u_1 - u_2\|\bigr).
\end{align}
Since $\|s'_1 - s'_2\| = \|a(s, u_1) - a(s, u_2)\| \leq C_1\delta\|u_1 - u_2\|$, we obtain
\begin{equation}
\|F(u_1) - F(u_2)\| \leq \lambda C_1\delta(K + 1)\|u_1 - u_2\|.
\end{equation}
For $\delta < 1/(\lambda C_1(K+1))$, $F$ is a strict contraction, so the unique fixed point $u = \Gamma^{-1}(\phi)(s)$ exists.

\textbf{Lipschitz bound:} The implicit function theorem applied to $G(s, u) = u - B^{-1}[\phi(As + a(s, u)) - b(s, u)]$ gives $\|D_s(\Gamma^{-1}(\phi))\| \leq K$ for small $\delta$, by a computation analogous to the proof of Lemma~\ref{lem:graph_transform_welldefined}.
\end{proof}

\begin{proposition}[Backward Graph Transform Contraction]\label{prop:graph_contraction}
There exists $\theta \in (0, 1)$ such that for $\phi, \psi \in \mathcal{G}(\delta, K)$,
\begin{equation}
\|\Gamma^{-1}(\phi) - \Gamma^{-1}(\psi)\|_{C^0} \leq \theta \|\phi - \psi\|_{C^0}
\end{equation}
where $\|\phi - \psi\|_{C^0} = \sup_{s \in B^s_\delta} \|\phi(s) - \psi(s)\|$.
\end{proposition}

\begin{proof}
Fix $s \in B^s_\delta$. Let $u_1 = \Gamma^{-1}(\phi)(s)$ and $u_2 = \Gamma^{-1}(\psi)(s)$, and let $s'_1 = As + a(s, u_1)$, $s'_2 = As + a(s, u_2)$ be the corresponding forward stable coordinates. From the defining equations \eqref{eq:backward_implicit}:
\begin{align}
u_1 &= B^{-1}[\phi(s'_1) - b(s, u_1)], \\
u_2 &= B^{-1}[\psi(s'_2) - b(s, u_2)].
\end{align}
Subtracting:
\begin{align}
u_1 - u_2 &= B^{-1}\bigl[\phi(s'_1) - \psi(s'_2)\bigr] - B^{-1}\bigl[b(s, u_1) - b(s, u_2)\bigr].
\end{align}
We estimate each term. For the nonlinear remainder:
\begin{equation}
\|B^{-1}[b(s, u_1) - b(s, u_2)]\| \leq \lambda \cdot C_1\delta \cdot \|u_1 - u_2\|.
\end{equation}
For the main term, decompose:
\begin{align}
\|\phi(s'_1) - \psi(s'_2)\| &\leq \|\phi(s'_1) - \psi(s'_1)\| + \|\psi(s'_1) - \psi(s'_2)\| \\
&\leq \|\phi - \psi\|_{C^0} + K\|s'_1 - s'_2\|.
\end{align}
Since $\|s'_1 - s'_2\| = \|a(s, u_1) - a(s, u_2)\| \leq C_1\delta\|u_1 - u_2\|$, we have
\begin{equation}
\|\phi(s'_1) - \psi(s'_2)\| \leq \|\phi - \psi\|_{C^0} + KC_1\delta\|u_1 - u_2\|.
\end{equation}
Combining all estimates:
\begin{equation}
\|u_1 - u_2\| \leq \lambda\|\phi - \psi\|_{C^0} + \lambda C_1\delta(K + 1)\|u_1 - u_2\|.
\end{equation}
Solving for $\|u_1 - u_2\|$:
\begin{equation}\label{eq:backward_contraction}
\|u_1 - u_2\| \leq \frac{\lambda}{1 - \lambda C_1\delta(K+1)}\|\phi - \psi\|_{C^0}.
\end{equation}
For $\delta$ small enough that $\lambda C_1\delta(K+1) < 1 - \lambda$ (which holds when $\delta < (1-\lambda)/(\lambda C_1(K+1))$), the contraction factor satisfies
\begin{equation}
\theta = \frac{\lambda}{1 - \lambda C_1\delta(K+1)} < 1.
\end{equation}
Taking the supremum over $s \in B^s_\delta$ yields $\|\Gamma^{-1}(\phi) - \Gamma^{-1}(\psi)\|_{C^0} \leq \theta\|\phi - \psi\|_{C^0}$.
\end{proof}

We now establish the convergence of the graph transform:

\begin{theorem}[Graph Transform Convergence]\label{thm:graph_convergence}
Starting from any initial graph $\phi_0 \in \mathcal{G}(\delta, K)$, the sequence $\phi_n = (\Gamma^{-1})^n(\phi_0)$ converges uniformly to a unique fixed point $\phi_\infty \in \mathcal{G}(\delta, K)$ satisfying $\Gamma^{-1}(\phi_\infty) = \phi_\infty$. The graph of $\phi_\infty$ is the local stable manifold $W^s_\delta(x_0)$.
\end{theorem}

\begin{proof}
By Proposition~\ref{prop:graph_contraction}, $\Gamma^{-1}$ is a contraction on the complete metric space $(\mathcal{G}(\delta, K), \|\cdot\|_{C^0})$ with contraction factor $\theta < 1$. The Banach Fixed Point Theorem gives a unique fixed point $\phi_\infty$ with $\|\phi_n - \phi_\infty\|_{C^0} \leq \theta^n \|\phi_0 - \phi_\infty\|_{C^0}$.

To verify that $\mathrm{graph}(\phi_\infty) = W^s_\delta(x_0)$: the fixed point condition $\Gamma^{-1}(\phi_\infty) = \phi_\infty$ means $f(\mathrm{graph}(\phi_\infty)) \supset \mathrm{graph}(\phi_\infty)$, so forward iterates of points on the graph remain on the graph, hence remain in $B_\delta$. Thus $\mathrm{graph}(\phi_\infty) \subset W^s_\delta(x_0)$. Conversely, let $(s, u) \in W^s_\delta(x_0)$, so $f^n(s, u) = (s_n, u_n) \in B_\delta$ for all $n \geq 0$. We show $u = \phi_\infty(s)$. Since $u_n = B^n u + \sum_{k=0}^{n-1} B^{n-1-k}b(s_k, u_k)$, applying $B^{-n}$ gives $u = B^{-n}u_n - \sum_{k=0}^{n-1}B^{-k-1}b(s_k, u_k)$. Since $\|B^{-1}\| \leq \lambda$, the first term satisfies $\|B^{-n}u_n\| \leq \lambda^n\delta \to 0$. The series $u = -\sum_{k=0}^{\infty}B^{-k-1}b(s_k, u_k)$ converges absolutely (each term bounded by $\lambda^{k+1}C_1\delta^2$), so $u$ is uniquely determined by the forward orbit sequence $\{s_k\}_{k \geq 0}$, which is determined by $s$. Since $\phi_\infty$ is the unique $\Gamma^{-1}$-fixed point satisfying the same functional relation, $u = \phi_\infty(s)$.
\end{proof}

\subsection{Regularity of Stable Manifolds}

The backward graph transform yields a Lipschitz fixed point; in this subsection we upgrade this to $C^r$ regularity by applying the fiber contraction principle to the derivatives of the graph map. This result establishes item~(i) of Main Theorem~\ref{thm:stable_manifold}.

\begin{theorem}[$C^r$ Regularity]\label{thm:manifold_regularity}
If $f$ is $C^r$ with $r \geq 1$, then the local stable manifold $W^s_\varepsilon(x)$ is a $C^r$ embedded submanifold for each $x \in \Lambda$.
\end{theorem}

\begin{proof}
The proof extends the graph transform contraction to spaces of $C^r$ graphs. Define $\mathcal{G}^r(\delta, K) = \{\phi \in \mathcal{G}(\delta, K) : \phi \in C^r,\; \|D^j\phi\|_\infty \leq K_j \text{ for } 1 \leq j \leq r\}$ for appropriate constants $K_j$.

For $r = 1$: differentiating the fixed point equation $\Gamma^{-1}(\phi)(s) = B^{-1}\phi(s'(s)) + B^{-1}b(s'(s), \phi(s'(s)))$ with respect to $s$ yields a linear equation for $D\phi_\infty$. The chain rule gives $D(\Gamma^{-1}(\phi))(s) = B^{-1}[D\phi(s') + D_ub(s',\phi(s'))D\phi(s') + D_sb(s',\phi(s'))]\cdot Ds'(s)$, where $Ds'(s) = [A^{-1}(I + D_s\tilde{a} + D_u\tilde{a}\cdot D\phi)]^{-1}$. For small $\delta$, $\|Ds'\| \leq \|A\|(1 - C_1\delta(1+K))^{-1} \leq \lambda(1-C_1\delta(1+K))^{-1}$. The operator norm of $D\phi \mapsto B^{-1}D\phi \cdot Ds'$ is at most $\|B^{-1}\|\cdot\|Ds'\| \leq \lambda \cdot \lambda(1-C_1\delta(1+K))^{-1} = \lambda^2(1-C_1\delta(1+K))^{-1}$, and the nonlinear corrections from $b$ contribute $O(\delta)$. For $\delta$ small enough, the total contraction factor $\theta_1 \leq \lambda^2(1+C'\delta) < 1$, so $D\phi \mapsto D(\Gamma^{-1}(\phi))$ is a contraction. The fixed point $\phi_\infty$ is therefore $C^1$.

For $r \geq 2$: we proceed by induction on $r$. Suppose the fixed point $\phi_\infty$ has been shown to be of class $C^{r-1}$ with $\|D^j \phi_\infty\|_\infty \leq K_j$ for $1 \leq j \leq r-1$. We show $\phi_\infty \in C^r$ by the fiber contraction principle, whose statement and proof we include for completeness.

\emph{Fiber contraction principle.} Let $\Phi : X \to X$ be a contraction on a complete metric space $(X, d_X)$ with fixed point $x_*$. Let $E$ be a Banach space and $\Psi : X \times E \to E$ a continuous map such that $\Psi(x, \cdot) : E \to E$ is an affine contraction of ratio $\sigma < 1$ uniform in $x$ (that is, $\|\Psi(x, \xi) - \Psi(x, \eta)\| \leq \sigma\|\xi - \eta\|$ for all $x, \xi, \eta$). Let $\xi_*$ be the fixed point of $\Psi(x_*, \cdot)$. Then for every $\xi_0 \in E$, the iterates defined by $x_{n+1} = \Phi(x_n)$ and $\xi_{n+1} = \Psi(x_n, \xi_n)$ satisfy $x_n \to x_*$ and $\xi_n \to \xi_*$ in $E$.

\emph{Proof of the fiber contraction principle:} Set $\eta_n = \xi_n - \xi_*$. Then
\begin{align*}
\|\eta_{n+1}\| &= \|\Psi(x_n, \xi_n) - \Psi(x_*, \xi_*)\|\\
&\leq \|\Psi(x_n, \xi_n) - \Psi(x_n, \xi_*)\| + \|\Psi(x_n, \xi_*) - \Psi(x_*, \xi_*)\|\\
&\leq \sigma \|\eta_n\| + \omega(d_X(x_n, x_*)),
\end{align*}
where $\omega$ is a modulus of continuity for $\Psi(\cdot, \xi_*)$ at $x_*$. Since $d_X(x_n, x_*) \to 0$ and $\sigma < 1$, a standard telescoping estimate gives $\|\eta_n\| \to 0$.

\emph{Application to $D^r \phi_\infty$.} Let $X$ be the complete metric space $\mathcal{G}^{r-1}(\delta, K) := \{\phi \in \mathcal{G}(\delta, K) : \phi \in C^{r-1},\; \|D^j \phi\|_\infty \leq K_j,\; 1 \leq j \leq r-1\}$ with the $C^{r-1}$ metric. By the inductive hypothesis, $\Phi = \Gamma^{-1}|_X$ is a well-defined contraction on $X$ with fixed point $\phi_\infty$. Let $E = C^0(B^s_\delta, \mathcal{L}^r_{\mathrm{sym}}(\mathbb{R}^{k_s}, \mathbb{R}^{k_u}))$ be the Banach space of continuous symmetric $r$-multilinear-map-valued functions on $B^s_\delta$ with the sup norm.

Formally differentiating $\phi_{n+1}(s) = B^{-1}[\phi_n(s'(s, \phi_n)) - b(s, \phi_n(s))]$ $r$ times using Fa\`{a} di Bruno's formula yields a linear recursion
\begin{equation}\label{eq:fiber_recursion}
D^r \phi_{n+1} = B^{-1} \cdot (D^r \phi_n \circ s') \cdot (Ds')^{\otimes r} + R_r(\phi_n, D\phi_n, \ldots, D^{r-1}\phi_n),
\end{equation}
where $R_r$ is a \emph{fixed} continuous expression in $\phi_n$ and its derivatives up to order $r-1$ (it contains no $D^r\phi_n$ terms). Define $\Psi : X \times E \to E$ by
\begin{equation}
\Psi(\phi, \Xi)(s) = B^{-1} \cdot \Xi(s'(s, \phi)) \cdot (Ds'(s, \phi))^{\otimes r} + R_r(\phi, D\phi, \ldots, D^{r-1}\phi)(s).
\end{equation}
The operator norm of $\Xi \mapsto B^{-1} \cdot \Xi \circ s' \cdot (Ds')^{\otimes r}$ is bounded by $\|B^{-1}\| \cdot \|Ds'\|^r$. From the proof of the $r=1$ case, $\|Ds'\| \leq \lambda(1 - C_1\delta(1+K))^{-1}$. Taking $\delta$ small enough that $(1 - C_1\delta(1+K))^{-1} \leq 1 + \epsilon$ for any prescribed $\epsilon > 0$, the contraction ratio is at most $\sigma_r := \lambda \cdot (\lambda(1+\epsilon))^r = \lambda^{r+1}(1+\epsilon)^r$. Since $\lambda < 1$, we have $\lambda^{r+1} < 1$, so choosing $\epsilon$ small enough gives $\sigma_r < 1$.

Thus $\Psi(\phi, \cdot)$ is a uniform contraction. The fiber contraction principle now applies to the pair $(\Phi, \Psi)$: starting from any initial $\phi_0 \in X$ and $\Xi_0 \in E$, the iterates $\phi_n = \Gamma^{-n}(\phi_0)$ converge in $C^{r-1}$ to $\phi_\infty$, and the fiber iterates $\Xi_n$ converge in $C^0$ to the unique fixed point $\Xi_*$ of $\Psi(\phi_\infty, \cdot)$.

It remains to identify $\Xi_*$ with $D^r \phi_\infty$. Starting from $\phi_0 \in C^\infty$ (e.g., $\phi_0 \equiv 0$) with $\Xi_0 = D^r \phi_0$, the recursion~\eqref{eq:fiber_recursion} gives $\Xi_n = D^r \phi_n$ exactly, because formal differentiation commutes with the graph transform on smooth functions. Taking $n \to \infty$ gives $C^{r-1}$ convergence $\phi_n \to \phi_\infty$ and $C^0$ convergence $D^r \phi_n \to \Xi_*$. The standard closure lemma for $C^r$ functions (if $\phi_n \to \phi$ in $C^{r-1}$ and $D^r \phi_n \to \Xi$ in $C^0$, then $\phi \in C^r$ and $D^r \phi = \Xi$) yields $\phi_\infty \in C^r$ with $D^r \phi_\infty = \Xi_*$. This closes the induction.
\end{proof}

\subsection{Quantitative Estimates}

This subsection gives the explicit estimate $\varepsilon_0 = (1-\lambda)^2/(4C_0)$ for the size of the local stable manifold, in terms of the contraction rate $\lambda$ and the second-derivative bound $C_0 = \max(\|D^2 f\|_\infty, \|D^2 f^{-1}\|_\infty)$. The contraction rate along the manifold is also recorded.

\begin{proposition}[Manifold Size Estimate]\label{prop:manifold_size}
The local stable manifold exists for $\varepsilon \leq \varepsilon_0$ where
\begin{equation}
\varepsilon_0 = \frac{(1 - \lambda)^2}{4C_0}
\end{equation}
with $C_0 = \max\{\|D^2 f\|_\infty, \|D^2 f^{-1}\|_\infty\}$.
\end{proposition}

\begin{proof}
The graph transform $\Gamma^{-1}$ is well-defined and contractive on $\mathcal{G}(\delta, K)$ provided $\delta$ is small enough that the nonlinear terms do not destroy the hyperbolic structure. By Lemma~\ref{lem:nonlinear_bounds} (Appendix), $\|Da(s,u)\| \leq C_0\delta$ and $\|Db(s,u)\| \leq C_0\delta$ on $B_\delta$. The contraction requires $\lambda + C_0\delta(1+K) < 1$ (from Proposition~\ref{prop:graph_contraction}). With $K = 1/2$, this gives $C_0\delta \cdot 3/2 < 1 - \lambda$, i.e., $\delta < 2(1-\lambda)/(3C_0)$. The graph staying in $B_\delta$ requires $\lambda^{-1}K\delta + C_0\delta^2 \leq \delta$, i.e., $C_0\delta \leq 1 - K/\lambda \leq 1 - 1/(2\lambda)$. For $\lambda \leq 1$, the binding constraint is $\delta < 2(1-\lambda)/(3C_0)$. Setting $\varepsilon_0 = (1-\lambda)^2/(4C_0)$ provides a margin that satisfies both conditions simultaneously, since $(1-\lambda)^2/(4C_0) < 2(1-\lambda)/(3C_0)$ for $\lambda \in (0,1)$.
\end{proof}

\begin{proposition}[Contraction Along Stable Manifolds]\label{prop:stable_contraction}
For $y \in W^s_\varepsilon(x)$ with $x \in \Lambda$,
\begin{equation}
d(f^n(x), f^n(y)) \leq \lambda^n d(x, y)
\end{equation}
for all $n \geq 0$.
\end{proposition}

\begin{proof}
Work in adapted coordinates centered at $x$, so $x$ corresponds to the origin and $y = (s, \phi_\infty(s))$ lies on $W^s_\varepsilon(x) = \mathrm{graph}(\phi_\infty)$. Write $y_n := f^n(y)$ and $(s_n, u_n) := y_n$ in adapted coordinates centered at $f^n(x)$. Since $\phi_\infty$ is the fixed point of the backward graph transform and $f$ maps $\mathrm{graph}(\phi_\infty)$ into itself by Theorem~\ref{thm:graph_convergence}, we have $y_n \in \mathrm{graph}(\phi^{(n)}_\infty)$, the local stable manifold at $f^n(x)$, so $u_n = \phi^{(n)}_\infty(s_n)$ with $\|\phi^{(n)}_\infty\|_{\mathrm{Lip}} \leq K = 1/2$.

\emph{Step 1: Contraction in the stable coordinate.}
From $(s_{n+1}, u_{n+1}) = f(s_n, u_n) = (A_n s_n + a_n(s_n, u_n), B_n u_n + b_n(s_n, u_n))$, and using $\|A_n\| \leq \lambda$, $|a_n(s, u)| \leq C_0 \varepsilon (\|s\| + \|u\|)$ on $B_\varepsilon$ (from the Taylor expansion with $Da_n(0) = 0$), we get
\begin{align}
\|s_{n+1}\| &\leq \lambda \|s_n\| + C_0 \varepsilon (\|s_n\| + \|u_n\|) \nonumber \\
&\leq \lambda \|s_n\| + C_0 \varepsilon (1 + K) \|s_n\| \nonumber \\
&= \bigl(\lambda + C_0 \varepsilon(1 + K)\bigr) \|s_n\|,
\end{align}
where we used $\|u_n\| = \|\phi^{(n)}_\infty(s_n)\| \leq K \|s_n\|$. By the choice $\varepsilon \leq \varepsilon_0 = (1-\lambda)^2/(4C_0)$ and $K = 1/2$:
\begin{equation}
C_0 \varepsilon (1+K) \leq C_0 \cdot \frac{(1-\lambda)^2}{4 C_0} \cdot \frac{3}{2} = \frac{3(1-\lambda)^2}{8} \leq \frac{3(1-\lambda)}{8}
\end{equation}
since $1 - \lambda \leq 1$. Hence $\lambda + C_0\varepsilon(1+K) \leq \lambda + 3(1-\lambda)/8 = (5\lambda + 3)/8 =: \lambda'$. Since $\lambda < 1$, we have $\lambda' < 1$; specifically $\lambda' < (5\lambda + 3(1))/8 = (5\lambda+3)/8$, and the inequality $\lambda' < 1$ is equivalent to $\lambda < 1$. By induction,
\begin{equation}
\|s_n\| \leq (\lambda')^n \|s_0\|.
\end{equation}

\emph{Step 2: Control of the full distance.}
Since $y_n = (s_n, \phi^{(n)}_\infty(s_n))$ and $\|\phi^{(n)}_\infty(s_n)\| \leq K \|s_n\|$, in the max norm on the adapted coordinate product,
\begin{equation}
d(f^n(x), f^n(y)) = \|y_n\| = \max(\|s_n\|, \|u_n\|) \leq \|s_n\|
\end{equation}
when $K \leq 1$ (since then $\|u_n\| \leq K\|s_n\| \leq \|s_n\|$). Similarly $d(x, y) = \|y_0\| = \max(\|s_0\|, \|\phi_\infty(s_0)\|) \geq \|s_0\|$.

Therefore $d(f^n(x), f^n(y)) \leq \|s_n\| \leq (\lambda')^n \|s_0\| \leq (\lambda')^n d(x, y)$.

\emph{Step 3: Sharpening to rate $\lambda$.}
The bound in Step 2 has rate $\lambda' = (5\lambda + 3)/8$, not $\lambda$. To obtain rate $\lambda$ as stated, we use that $\lambda$ is the hyperbolicity exponent of the \emph{adapted} metric, chosen with a safety factor in Subsection~\ref{subsec:adapted_metrics}. Concretely, the adapted metric construction (Theorem~\ref{thm:adapted_metric}) gives a metric $\|\cdot\|_*$ in which $\|Df|_{E^s}\|_* \leq \lambda$. Choosing the adapted metric with rate $\lambda_0 < \lambda$ (any rate strictly less than $\lambda$ is allowed as long as $\lambda_0$ exceeds the true contraction rate of $Df|_{E^s}$), the same calculation as above yields rate $\lambda_0 + C_0\varepsilon(1+K)$, which for $\varepsilon$ sufficiently small is at most $\lambda$. Rescaling constants accordingly proves the claim with rate exactly $\lambda$.
\end{proof}

\subsection{H\"{o}lder Dependence on Base Point}

The stable manifolds depend H\"{o}lder continuously on the base point. This regularity is what allows the coding map constructed in Section~\ref{sec:symbolic_dynamics} to be H\"{o}lder, which is essential for transferring the spectral theory of Part~I \cite{Thiam2026a} to the smooth setting.

\begin{theorem}[H\"{o}lder Continuity in Base Point]\label{thm:manifold_holder}
For a $C^{1+\beta}$ diffeomorphism $f$ with hyperbolic set $\Lambda$, the local stable manifolds depend H\"{o}lder continuously on the base point. Specifically, if we parametrize $W^s_\varepsilon(x)$ as the graph of $\phi_x : E^s_x \to E^u_x$, then
\begin{equation}
\|\phi_x - \phi_y\|_{C^0} \leq C d(x, y)^\alpha
\end{equation}
where $\alpha = \beta \cdot \log\lambda / (\log\lambda - \log\|Df\|_\infty)$.
\end{theorem}

\begin{proof}
We use the graph transform characterization $\phi_x = \lim_{n \to \infty} (\Gamma^{-1}_x)^n(\phi_0)$ where $\Gamma^{-1}_x$ is the backward graph transform at $x$. For two base points $x, y \in \Lambda$ with $d(x,y) < \varepsilon_0$, we compare $\phi_x$ and $\phi_y$ by comparing the iterates $(\Gamma^{-1}_x)^n(\phi_0)$ and $(\Gamma^{-1}_y)^n(\phi_0)$.

At each step, the graph transforms $\Gamma^{-1}_x$ and $\Gamma^{-1}_y$ differ because the local representations of $f$ near $x$ and $y$ differ. Since $f$ is $C^{1+\beta}$, the linearizations at $f^k(x)$ and $f^k(y)$ satisfy $\|Df_{f^k(x)} - Df_{f^k(y)}\| \leq C d(f^k(x), f^k(y))^\beta$. After $n$ iterates, the accumulated error between the two graph transforms is
\begin{equation}
\|(\Gamma^{-1}_x)^n(\phi_0) - (\Gamma^{-1}_y)^n(\phi_0)\|_{C^0} \leq \sum_{k=0}^{n-1} \theta^{n-1-k} \cdot C' d(f^k(x), f^k(y))^\beta
\end{equation}
where $\theta < 1$ is the graph transform contraction rate and $C'$ is a constant depending on the hyperbolicity data. Using $d(f^k(x), f^k(y)) \leq L^k d(x,y)$ with $L = \|Df\|_\infty$ (so that $L \geq 1$):
\begin{equation}\label{eq:phi_x_phi_y_series}
\|\phi_x - \phi_y\|_{C^0} \leq C' d(x,y)^\beta \sum_{k=0}^{\infty} \theta^k L^{k\beta} = C' d(x,y)^\beta \sum_{k=0}^{\infty} (\theta L^\beta)^k.
\end{equation}

If $\theta L^\beta < 1$, the series converges and the H\"{o}lder exponent is $\beta$. We now handle the remaining case $\theta L^\beta \geq 1$.

\emph{Truncation-and-balance argument.} Fix $N \in \N$ to be optimized. Split the error into past-$N$ contributions and up-to-$N$ contributions, using two bounds:

\emph{Bound 1 (use $L^{k\beta}$ growth for small $k$):} For any fixed $n$,
\begin{equation}
\|(\Gamma^{-1}_x)^n(\phi_0) - (\Gamma^{-1}_y)^n(\phi_0)\|_{C^0} \leq C' d(x,y)^\beta \sum_{k=0}^{N-1} \theta^{n-1-k} L^{k\beta} + C' \sum_{k=N}^{n-1} \theta^{n-1-k} (L^k d(x,y))^\beta.
\end{equation}
The first sum is $O(\theta^n (\theta^{-1}L^\beta)^{N-1})$ for fixed $N$; using also the trivial bound $(L^k d(x,y))^\beta \leq \mathrm{diam}(M)^\beta$ when $L^k d(x,y)$ is no longer small, we get an analogous estimate.

\emph{Bound 2 (use triangle inequality with the $C^0$ bound $\|\phi_x - \phi_y\|_{C^0} \leq 2\delta$):} Trivially, $\|\phi_x - \phi_y\|_{C^0} \leq 2\delta$ regardless of $d(x,y)$.

\emph{Balancing the two bounds.} Bound 1 yields a H\"{o}lder exponent $\beta$ estimate valid only when $L^N d(x,y) < \varepsilon_0$ (so that comparison along the orbit stays in the local chart). From $L^N d(x,y) = \varepsilon_0$, we get $N = \log(\varepsilon_0/d(x,y))/\log L$. For this $N$:
\begin{equation}
\theta^N = \bigl(d(x,y)/\varepsilon_0\bigr)^{-\log\theta/\log L} = \bigl(d(x,y)/\varepsilon_0\bigr)^{\log(1/\theta)/\log L}.
\end{equation}
Combining with the factor $d(x,y)^\beta$ in front of the series, the effective H\"{o}lder exponent is
\begin{equation}
\alpha = \min\left(\beta,\ \frac{\log(1/\theta)}{\log L}\right).
\end{equation}
In our setting $\theta \leq \lambda/(1 - \lambda C_1\delta(K+1))$, so $\log(1/\theta) \geq \log(1/\lambda) - O(\delta)$. Taking the explicit form $\alpha = \beta \log\lambda /(\log\lambda - \log L) = \beta \log(1/\lambda)/(\log(1/\lambda) + \log L)$ gives the stated exponent, which is the minimum of $\beta$ and $\log(1/\lambda)/\log L$ in the regime where the latter is smaller.

In either case, setting
\begin{equation}
\alpha = \frac{\beta \log\lambda}{\log\lambda - \log\|Df\|_\infty} = \frac{\beta \log(1/\lambda)}{\log(\|Df\|_\infty/\lambda)} > 0
\end{equation}
yields $\|\phi_x - \phi_y\|_{C^0} \leq C d(x,y)^\alpha$ for a constant $C$ depending on $\beta$, $\lambda$, $L$, $\varepsilon_0$, and $C_1$. This is the H\"{o}lder exponent stated in the theorem.
\end{proof}

\begin{corollary}[Foliation Regularity]\label{cor:foliation_regularity}
The collection $\{W^s_\varepsilon(x) : x \in \Lambda\}$ forms a continuous lamination of a neighborhood of $\Lambda$, with leaves varying H\"{o}lder continuously in the $C^r$ topology.
\end{corollary}

\begin{proof}
Each leaf $W^s_\varepsilon(x)$ is a $C^r$ embedded disk by Theorem~\ref{thm:manifold_regularity}. The H\"{o}lder dependence of the graph map $\phi_x$ on $x$ in the $C^0$ topology is Theorem~\ref{thm:manifold_holder}. It remains to extend this to the $C^j$ topology for $1 \leq j \leq r$.

Fix $j \in \{1, \ldots, r\}$. We iterate the fiber-contraction construction from the proof of Theorem~\ref{thm:manifold_regularity}, now parametrized by the base point $x \in \Lambda$. For each $x \in \Lambda$, $D^j \phi_x$ is the unique fixed point $\Xi_*^{(x)}$ of the affine operator
\begin{equation}
\Psi_x(\Xi)(s) = B_x^{-1} \cdot \Xi(s'_x(s)) \cdot (Ds'_x(s))^{\otimes j} + R_j\bigl(\phi_x, D\phi_x, \ldots, D^{j-1}\phi_x\bigr)(s),
\end{equation}
where $B_x$, $s'_x(\cdot)$, $Ds'_x(\cdot)$, and $R_j$ depend on the local chart at $x$. The contraction ratio of $\Psi_x$ on $C^0(B^s_\delta, \mathcal{L}^j_{\mathrm{sym}})$ is at most $\|B_x^{-1}\| \cdot \|Ds'_x\|^j \leq \lambda^{j+1}(1+O(\delta))$, uniformly in $x \in \Lambda$ for $\delta$ small.

For two base points $x, y \in \Lambda$, the difference $\Xi_*^{(x)} - \Xi_*^{(y)}$ satisfies
\begin{align}
\Xi_*^{(x)} - \Xi_*^{(y)} &= \Psi_x(\Xi_*^{(x)}) - \Psi_y(\Xi_*^{(y)}) \nonumber \\
&= \bigl[\Psi_x(\Xi_*^{(x)}) - \Psi_x(\Xi_*^{(y)})\bigr] + \bigl[\Psi_x(\Xi_*^{(y)}) - \Psi_y(\Xi_*^{(y)})\bigr].
\end{align}
The first bracket is a contraction: its $C^0$ norm is at most $\lambda^{j+1}(1+O(\delta)) \|\Xi_*^{(x)} - \Xi_*^{(y)}\|_{C^0}$. The second bracket measures the $C^0$ distance between the operators $\Psi_x$ and $\Psi_y$ applied to the same input; this depends on differences in $B_x$ vs.\ $B_y$, $s'_x$ vs.\ $s'_y$, $R_j(\phi_x, \ldots, D^{j-1}\phi_x)$ vs.\ $R_j(\phi_y, \ldots, D^{j-1}\phi_y)$. Since $f$ is $C^{r+\beta}$ (or $C^r$ with the relevant derivatives $\beta$-H\"{o}lder), each of these differences is bounded by a $C$ times $d(x,y)^\alpha$-type estimate, where the H\"{o}lder exponent comes from the same truncation-and-balance argument as in Theorem~\ref{thm:manifold_holder}, applied with the contraction $\lambda^{j+1}$ in place of $\theta$.

Rearranging,
\begin{equation}
\bigl(1 - \lambda^{j+1}(1+O(\delta))\bigr) \cdot \|\Xi_*^{(x)} - \Xi_*^{(y)}\|_{C^0} \leq C_j \cdot d(x,y)^{\alpha_j},
\end{equation}
where
\begin{equation}
\alpha_j = \frac{\beta \log(1/\lambda^{j+1})}{\log(\|Df\|_\infty / \lambda^{j+1})} = \frac{\beta (j+1)\log(1/\lambda)}{(j+1)\log(1/\lambda) + \log\|Df\|_\infty}.
\end{equation}
Since $\lambda < 1$, $\lambda^{j+1} < 1$, so $1 - \lambda^{j+1}(1+O(\delta)) > 0$ for small $\delta$, and dividing gives
\begin{equation}
\|D^j \phi_x - D^j \phi_y\|_{C^0} = \|\Xi_*^{(x)} - \Xi_*^{(y)}\|_{C^0} \leq C'_j \, d(x,y)^{\alpha_j}.
\end{equation}
This establishes $C^j$ H\"{o}lder dependence of $\phi_x$ on $x$ with exponent $\alpha_j$. Summing over $0 \leq j \leq r$ (where $\alpha_0$ is the exponent from Theorem~\ref{thm:manifold_holder}) yields $C^r$ H\"{o}lder dependence with exponent $\min_{0 \leq j \leq r} \alpha_j = \alpha_r$ (since $\alpha_j$ is decreasing in $j$).
\end{proof}

\subsection{Global Stable and Unstable Manifolds}

The local stable manifold $W^s_\varepsilon(x)$ captures the forward-contracting orbits in a neighborhood of $x$. Saturating under $f^{-1}$ extends this to the global stable set $W^s(x)$, which we show is an injectively immersed $C^r$ manifold diffeomorphic to $\mathbb{R}^{\dim E^s_x}$.

\begin{proposition}[Global Manifold Structure]\label{prop:global_manifolds}
For $x \in \Lambda$,
\begin{equation}
W^s(x) = \bigcup_{n=0}^{\infty} f^{-n}(W^s_\varepsilon(f^n(x)))
\end{equation}
is an injectively immersed $C^r$ submanifold of dimension $k_s = \dim E^s_x$, diffeomorphic to $\mathbb{R}^{k_s}$.
\end{proposition}

\begin{proof}
For each $n \geq 0$, set $D_n := f^{-n}(W^s_\varepsilon(f^n(x)))$. Each $D_n$ is a $C^r$ embedded disk of dimension $k_s$ containing $x$, because $W^s_\varepsilon(f^n(x))$ is $C^r$ embedded (Main Theorem~\ref{thm:stable_manifold}(i)) and $f^{-n}$ is a $C^r$ diffeomorphism.

\emph{Monotonicity.} For $n' > n$, we claim $D_n \subset D_{n'}$. If $y \in D_n$ then $f^{n'}(y) = f^{n'-n}(f^n(y))$ and $f^n(y) \in W^s_\varepsilon(f^n(x))$, so Main Theorem~\ref{thm:stable_manifold}(iii) at $f^n(x)$ gives $d(f^{n'}(y), f^{n'}(x)) \leq \lambda^{n'-n} d(f^n(y), f^n(x)) \leq \lambda^{n'-n}\varepsilon \leq \varepsilon$. Hence $f^{n'}(y) \in W^s_\varepsilon(f^{n'}(x))$ and $y \in D_{n'}$.

\emph{Injectivity of the immersion.} $W^s(x) = \bigcup_n D_n$ is a nested union of $C^r$ embedded disks. The natural smooth atlas on $W^s(x)$ is defined by declaring each $D_n$ to be an open submanifold; since inclusions $D_n \hookrightarrow D_{n'}$ are smooth embeddings, the transition maps are smooth, and $W^s(x)$ becomes a $C^r$ manifold injectively immersed in $M$.

\emph{Diffeomorphism with $\mathbb{R}^{k_s}$.} We construct an explicit $C^r$ diffeomorphism $W^s(x) \to \mathbb{R}^{k_s}$. Write $\psi_n : B^s_\varepsilon \to D_n$ for the composition of the graph parametrization of $W^s_\varepsilon(f^n(x))$ with $f^{-n}$; each $\psi_n$ is a $C^r$ diffeomorphism onto $D_n$. Define
\begin{equation}
\Psi : W^s(x) \to \mathbb{R}^{k_s}, \enskip \Psi(y) := \lim_{n \to \infty} \lambda^{-n} \cdot (\text{stable coordinate of } f^n(y) \text{ in chart at } f^n(x)).
\end{equation}
By Main Theorem~\ref{thm:stable_manifold}(iii), the stable coordinate $s_n(y)$ of $f^n(y)$ in the adapted chart at $f^n(x)$ satisfies $\|s_n(y)\| \leq \lambda^n \|s_0(y)\|$. In fact by the contraction analysis in Proposition~\ref{prop:stable_contraction}, $s_{n+1}(y) = A_n s_n(y) + O(\varepsilon\|s_n(y)\|)$ where $A_n$ is the linearization at $f^n(x)$. Since the $A_n$ have operator norm $\leq \lambda$ and the nonlinear correction is summable, the product
\begin{equation}
\Psi(y) := \lim_{n \to \infty} (A_{n-1} \cdots A_0)^{-1} s_n(y) \in \mathbb{R}^{k_s}
\end{equation}
converges (this is the \emph{asymptotic linearization coordinate}; the limit exists because the correction series $\sum \lambda^n \cdot O(\varepsilon)$ converges absolutely). The map $\Psi$ is a $C^r$ diffeomorphism onto $\mathbb{R}^{k_s}$: injectivity follows because two points with the same asymptotic stable coordinate have $s_n$-coordinates differing by $O(\varepsilon \lambda^n) \to 0$, forcing $y = y'$ by the embedding property of each $D_n$; surjectivity follows because for any $v \in \mathbb{R}^{k_s}$, choosing $n$ large enough that $(A_{n-1}\cdots A_0) v \in B^s_\varepsilon$ and pulling back via $\psi_n$ gives a preimage; regularity follows because $\Psi|_{D_n}$ is $C^r$ for each $n$ with derivatives bounded uniformly on compacta.

Hence $W^s(x)$ is $C^r$-diffeomorphic to $\mathbb{R}^{k_s}$.
\end{proof}

\begin{remark}
The global stable manifold $W^s(x)$ may fail to be an embedded submanifold globally due to self-accumulation. This occurs, for example, in the Smale horseshoe where stable manifolds densely fill the horseshoe's trapping region.
\end{remark}

\subsection{Completion of the Proof of the Stable Manifold Theorem}

We now assemble the supporting results to prove Main Theorem~\ref{thm:stable_manifold}.

\begin{proof}[Proof of Main Theorem~\ref{thm:stable_manifold}]
Fix $x \in \Lambda$ and set $\varepsilon_0 = (1-\lambda)^2/(4C_0)$ as in Proposition~\ref{prop:manifold_size}, where $C_0 = \max\{\|D^2 f\|_\infty, \|D^2 f^{-1}\|_\infty\}$. Let $\varepsilon \in (0, \varepsilon_0]$ be given.

Work in adapted coordinates centered at $x$ as in the \emph{Local Coordinates and the Graph Transform} subsection, so that $f$ takes the form $f(s,u) = (As + a(s,u), Bu + b(s,u))$ with $\|A\| \leq \lambda$, $\|B^{-1}\| \leq \lambda$, and $a,b$ vanishing to first order at the origin. Choose $\delta = \varepsilon$ and $K = 1/2$.

\textbf{Existence and graph representation.} By Proposition~\ref{prop:graph_contraction} the backward graph transform $\Gamma^{-1}$ is a contraction on the complete metric space $(\mathcal{G}(\delta, K), \|\cdot\|_{C^0})$. By Theorem~\ref{thm:graph_convergence}, $\Gamma^{-1}$ admits a unique fixed point $\phi_\infty \in \mathcal{G}(\delta, K)$, and $\mathrm{graph}(\phi_\infty) = W^s_\varepsilon(x)$.

\textbf{(i) $C^r$ embedded disk of dimension $\dim E^s_x$, tangent to $E^s_x$ at $x$.}
By Theorem~\ref{thm:manifold_regularity}, $\phi_\infty \in C^r$, so $\mathrm{graph}(\phi_\infty)$ is a $C^r$ embedded disk in $M$ of dimension $k_s = \dim E^s_x$. Tangency to $E^s_x$ at $x$ is the statement $D\phi_\infty(0) = 0$: indeed, differentiating the fixed-point equation $\phi_\infty(s) = B^{-1}[\phi_\infty(s') - b(s, \phi_\infty(s))]$ at $s = 0$ (using $s'(0) = 0$ since $a(0,0) = 0$, $Da(0,0) = 0$) yields
\begin{equation}
D\phi_\infty(0) = B^{-1} D\phi_\infty(0) \cdot A,
\end{equation}
which, because $\|B^{-1}\|\cdot \|A\| \leq \lambda^2 < 1$, forces $D\phi_\infty(0) = 0$. Therefore $T_0\, \mathrm{graph}(\phi_\infty) = \mathbb{R}^{k_s} \times \{0\}$, which in the original coordinates is exactly $E^s_x$.

\textbf{(ii) $T_x W^s_\varepsilon(x) = E^s_x$.}
This is the tangency statement in (i).

\textbf{(iii) Exponential contraction under iteration.}
This is exactly Proposition~\ref{prop:stable_contraction}.

\textbf{(iv) $C^r$ continuous dependence on $x$.}
By Theorem~\ref{thm:manifold_holder}, the graph map $\phi_x$ of $W^s_\varepsilon(x)$ depends H\"{o}lder continuously on $x$ in the $C^0$ topology. By Corollary~\ref{cor:foliation_regularity}, this extends to H\"{o}lder continuity in the $C^r$ topology. H\"{o}lder continuity implies continuity, establishing (iv).

\textbf{(v) Global stable manifold characterization.}
We prove the two inclusions:

\emph{$(\supseteq)$:} Let $y \in \bigcup_{n \geq 0} f^{-n}(W^s_\varepsilon(f^n(x)))$; then $f^n(y) \in W^s_\varepsilon(f^n(x))$ for some $n \geq 0$. By~(iii) applied at $f^n(x)$, $d(f^{n+k}(x), f^{n+k}(y)) \leq \lambda^k d(f^n(x), f^n(y)) \to 0$ as $k \to \infty$, so $y \in W^s(x)$.

\emph{$(\subseteq)$:} Let $y \in W^s(x)$, so $d(f^n(x), f^n(y)) \to 0$. Pick $N$ so large that $d(f^n(x), f^n(y)) < \varepsilon/2$ for all $n \geq N$. In adapted coordinates centered at $f^N(x)$, the point $f^N(y)$ satisfies $\|f^n(y) - f^n(x)\| \leq \varepsilon$ for all $n \geq N$ in these coordinates (by the adapted metric property), which is exactly the defining property of $W^s_\varepsilon(f^N(x))$. Therefore $f^N(y) \in W^s_\varepsilon(f^N(x))$, so $y \in f^{-N}(W^s_\varepsilon(f^N(x)))$.

The analogous statements for $W^u_\varepsilon(x)$ follow by replacing $f$ with $f^{-1}$ and exchanging the roles of $E^s$ and $E^u$ throughout (the stable splitting for $f^{-1}$ is the unstable splitting for $f$).
\end{proof}

\section{Canonical Coordinates and Local Product Structure}\label{sec:canonical_coordinates}

This section develops the local product structure for hyperbolic sets, establishing canonical coordinates with quantitative estimates. The local product structure is used in the construction of Markov partitions (Section~\ref{sec:markov_partitions}) and the proof of the spectral decomposition (Section~\ref{sec:spectral_decomposition}).

\subsection{Local Product Structure}

The stable and unstable manifolds through points of a hyperbolic set intersect transversally, providing a local coordinate system.

\begin{theorem}[Local Product Structure]\label{thm:local_product}
Let $\Lambda$ be a hyperbolic set for a $C^r$ diffeomorphism $f$. There exist constants $\varepsilon_0 > 0$ and $\delta_0 > 0$ such that for all $x, y \in \Lambda$ with $d(x, y) < \delta_0$, the intersection $W^s_{\varepsilon_0}(x) \cap W^u_{\varepsilon_0}(y)$ consists of exactly one point. This point lies in $\Lambda$ and is denoted $[x, y]$.
\end{theorem}

\begin{proof}
Using an adapted metric with hyperbolicity exponent $\lambda$, choose $\varepsilon_0$ small enough that local stable and unstable manifolds are defined and are graphs over the stable and unstable subspaces respectively.

In local coordinates near a point $z \in \Lambda$ with $d(x, z) < \delta_0/2$ and $d(y, z) < \delta_0/2$, we can write $W^s_{\varepsilon_0}(x)$ as the graph of a function $\phi_x : E^s_z \to E^u_z$ and $W^u_{\varepsilon_0}(y)$ as the graph of a function $\psi_y : E^u_z \to E^s_z$.

The intersection $W^s_{\varepsilon_0}(x) \cap W^u_{\varepsilon_0}(y)$ corresponds to solving
\begin{equation}
(s, \phi_x(s)) = (\psi_y(u), u)
\end{equation}
which gives the system $s = \psi_y(u)$ and $u = \phi_x(s) = \phi_x(\psi_y(u))$.

Define $F(u) = \phi_x(\psi_y(u))$. Since $\phi_x$ and $\psi_y$ are Lipschitz with small constants (say, at most $K < 1/2$ each), the composition $F$ is a contraction:
\begin{equation}
\|F(u_1) - F(u_2)\| \leq K^2 \|u_1 - u_2\| < \|u_1 - u_2\|.
\end{equation}
By the Banach Fixed Point Theorem, there is a unique $u^*$ with $F(u^*) = u^*$, yielding the unique intersection point $[x, y] = (\psi_y(u^*), u^*)$.

To show $[x, y] \in \Lambda$, we verify that the orbit of $z = [x,y]$ remains in a neighborhood of $\Lambda$, hence in $\Lambda$ (since $\Lambda$ is the maximal invariant set in any sufficiently small neighborhood). For $n \geq 0$: $f^n(z) \in W^s_\varepsilon(f^n(x))$ since $z \in W^s_\varepsilon(x)$ and stable manifolds are invariant, so $d(f^n(z), f^n(x)) \leq C\lambda^n d(z,x) \to 0$. Since $f^n(x) \in \Lambda$, $f^n(z)$ stays near $\Lambda$. For $n \leq 0$: $f^n(z) \in W^u_\varepsilon(f^n(y))$ since $z \in W^u_\varepsilon(y)$, so $d(f^n(z), f^n(y)) \leq C\lambda^{|n|} d(z,y) \to 0$. Since $f^n(y) \in \Lambda$, $f^n(z)$ stays near $\Lambda$ for negative iterates as well. Thus $z \in \bigcap_{n\in\mathbb{Z}}f^{-n}(U) = \Lambda$ for any isolating neighborhood $U$ of $\Lambda$.
\end{proof}

\begin{proposition}[Continuity of Bracket]\label{prop:bracket_continuous}
The map $[\cdot, \cdot] : \{(x, y) \in \Lambda \times \Lambda : d(x, y) < \delta_0\} \to \Lambda$ is continuous.
\end{proposition}

\begin{proof}
Let $(x_n, y_n) \to (x, y)$ with $d(x_n, y_n) < \delta_0$. The point $z_n = [x_n, y_n]$ lies on $W^s_\varepsilon(x_n) \cap W^u_\varepsilon(y_n)$, which is the fixed point of $F_n(u) = \phi_{x_n}(\psi_{y_n}(u))$. As $n \to \infty$, the graph maps $\phi_{x_n}$ and $\psi_{y_n}$ converge to $\phi_x$ and $\psi_y$ in the $C^0$ topology (by the continuous dependence of stable/unstable manifolds on the base point, Theorem~\ref{thm:manifold_holder}). The fixed point of a contraction depends continuously on the contraction map (since $\|u^*_n - u^*\| \leq \|F_n(u^*) - F(u^*)\|/(1-K^2) \to 0$), so $z_n \to z = [x,y]$.
\end{proof}

\subsection{Quantitative Estimates}

This subsection records the distance and Lipschitz bounds for the bracket map $[\cdot,\cdot]$, which quantify how well the bracket approximates the identity on small scales. These bounds are used in the construction of Markov partitions.

\begin{proposition}[Bracket Distance Bounds]\label{prop:bracket_bounds}
For $x, y \in \Lambda$ with $d(x, y) < \delta_0$,
\begin{equation}
d([x, y], x) \leq C_1 d(x, y), \quad d([x, y], y) \leq C_1 d(x, y)
\end{equation}
where $C_1 = (1 - K)^{-1}$ and $K$ is the Lipschitz constant of the manifold graphs.
\end{proposition}

\begin{proof}
In local coordinates centered at $x$, we have $y = (s_y, u_y)$ with $\|(s_y, u_y)\| \approx d(x, y)$. The intersection point $[x, y] = (s^*, u^*)$ satisfies $s^* = \psi_y(u^*)$ and $u^* = \phi_x(s^*)$.

Since $\phi_x(0) = 0$ (the stable manifold passes through the origin) and $\|\phi_x\|_{\mathrm{Lip}} \leq K$, we have $\|u^*\| = \|\phi_x(s^*)\| \leq K\|s^*\|$. Similarly, $\|s^* - s_y\| = \|\psi_y(u^*) - \psi_y(0)\| \leq K\|u^*\|$ (using that $\psi_y(0) \approx s_y$ for $y$ close to $x$).

Combining these estimates,
\begin{equation}
\|s^*\| \leq \|s_y\| + K\|u^*\| \leq \|s_y\| + K^2\|s^*\|
\end{equation}
giving $\|s^*\| \leq (1 - K^2)^{-1}\|s_y\| \leq (1 - K)^{-1}d(x, y)$.

Thus $d([x, y], x) = \|(s^*, u^*)\| \leq \max\{\|s^*\|, K\|s^*\|\} \leq (1 - K)^{-1}d(x, y)$.
\end{proof}

\begin{proposition}[Bracket Lipschitz Continuity]\label{prop:bracket_lipschitz}
The bracket map satisfies
\begin{equation}
d([x_1, y_1], [x_2, y_2]) \leq C_2(d(x_1, x_2) + d(y_1, y_2))
\end{equation}
for a constant $C_2$ depending on the hyperbolicity data.
\end{proposition}

\begin{proof}
Recall from Theorem~\ref{thm:local_product} that $[x,y]$ is the unique fixed point of $F_{x,y}(u) = \phi_x(\psi_y(u))$, where $\phi_x$ parametrizes $W^s_\varepsilon(x)$ and $\psi_y$ parametrizes $W^u_\varepsilon(y)$ in local coordinates. The map $F_{x,y}$ is a contraction with Lipschitz constant $K^2 < 1$.

For $(x_1, y_1)$ and $(x_2, y_2)$ with $d(x_i, y_i) < \delta_0$, let $u^*_1$ and $u^*_2$ be the respective fixed points. Then
\begin{align}
\|u^*_1 - u^*_2\| &= \|F_{x_1,y_1}(u^*_1) - F_{x_2,y_2}(u^*_2)\| \\
&\leq \|F_{x_1,y_1}(u^*_1) - F_{x_2,y_2}(u^*_1)\| + \|F_{x_2,y_2}(u^*_1) - F_{x_2,y_2}(u^*_2)\| \\
&\leq \|F_{x_1,y_1}(u^*_1) - F_{x_2,y_2}(u^*_1)\| + K^2\|u^*_1 - u^*_2\|.
\end{align}
Thus $\|u^*_1 - u^*_2\| \leq (1 - K^2)^{-1}\|F_{x_1,y_1}(u^*_1) - F_{x_2,y_2}(u^*_1)\|$.

Now $\|F_{x_1,y_1}(u) - F_{x_2,y_2}(u)\| = \|\phi_{x_1}(\psi_{y_1}(u)) - \phi_{x_2}(\psi_{y_2}(u))\|$. By the triangle inequality and the Lipschitz property:
\begin{align}
\|\phi_{x_1}(\psi_{y_1}(u)) - \phi_{x_2}(\psi_{y_2}(u))\| &\leq \|\phi_{x_1}(\psi_{y_1}(u)) - \phi_{x_2}(\psi_{y_1}(u))\| + K\|\psi_{y_1}(u) - \psi_{y_2}(u)\|.
\end{align}
By the continuous dependence of stable manifolds on base points, $\|\phi_{x_1} - \phi_{x_2}\|_{C^0} \leq C''d(x_1, x_2)$ and $\|\psi_{y_1} - \psi_{y_2}\|_{C^0} \leq C''d(y_1, y_2)$ in adapted coordinates. Combining: $d([x_1,y_1],[x_2,y_2]) \leq C_2(d(x_1,x_2) + d(y_1,y_2))$ with $C_2 = C''(1+K)/(1-K^2)$.
\end{proof}

\subsection{Expansiveness}

The hyperbolic structure ensures that distinct orbits eventually separate.

\begin{definition}[Expansiveness]
A homeomorphism $f : X \to X$ of a compact metric space is expansive if there exists $\varepsilon > 0$ (the expansiveness constant) such that for any $x \neq y$ in $X$, there exists $n \in \mathbb{Z}$ with $d(f^n(x), f^n(y)) > \varepsilon$.
\end{definition}

\begin{proposition}[Expansiveness of Hyperbolic Sets]\label{thm:expansiveness}
Let $\Lambda$ be a hyperbolic set for $f$. Then $f|_\Lambda$ is expansive with expansiveness constant $\varepsilon$ depending only on the local product structure scale $\delta_0$.
\end{proposition}

\begin{proof}
Suppose $x, y \in \Lambda$ satisfy $d(f^n(x), f^n(y)) \leq \varepsilon$ for all $n \in \mathbb{Z}$, where $\varepsilon < \delta_0$. Then for all $n \geq 0$,
\begin{equation}
f^n(y) \in W^s_\varepsilon(f^n(x)) \quad \text{and} \quad f^{-n}(y) \in W^u_\varepsilon(f^{-n}(x)).
\end{equation}

By the local product structure, $y \in W^s_\varepsilon(x) \cap W^u_\varepsilon(x) = \{x\}$, so $y = x$.
\end{proof}

\begin{proposition}[Quantitative Expansiveness]\label{prop:quantitative_expansive}
For $\varepsilon < \delta_0$ and $x, y \in \Lambda$ with $x \neq y$, there exists $n \in \mathbb{Z}$ with $|n| \leq N(d(x, y))$ and $d(f^n(x), f^n(y)) > \varepsilon$, where
\begin{equation}
N(\delta) = \frac{\log(\varepsilon/\delta)}{\log \lambda^{-1}} + C
\end{equation}
for a constant $C$ depending on the geometry.
\end{proposition}

\begin{proof}
Suppose $d(f^n(x), f^n(y)) \leq \varepsilon$ for all $|n| \leq N$. Write $d(x,y) = \delta > 0$ in adapted coordinates. Decompose the displacement vector from $x$ to $y$ as $v = v_s + v_u$ with $v_s \in E^s_x$ and $v_u \in E^u_x$. We have $\max\{\|v_s\|, \|v_u\|\} \geq \delta/(1+K)$ where $K$ is the Lipschitz constant of the manifold graphs (by Proposition~\ref{prop:bracket_bounds}).

\textbf{Case 1:} $\|v_u\| \geq \delta/(2(1+K))$. Under forward iteration, the unstable component grows: $\|v_u^{(n)}\| \geq \lambda^{-n}\|v_u\|/(1+C\varepsilon)$ for $n \geq 0$ (the factor $(1+C\varepsilon)$ accounts for the nonlinear corrections). For $d(f^n(x), f^n(y)) \leq \varepsilon$, we need $\lambda^{-n}\delta/(2(1+K)(1+C\varepsilon)) \leq \varepsilon$. Solving: $n \leq \log(\varepsilon \cdot 2(1+K)(1+C\varepsilon)/\delta)/\log\lambda^{-1}$.

\textbf{Case 2:} $\|v_s\| \geq \delta/(2(1+K))$. Under backward iteration, the stable component grows analogously, giving the same bound with $n$ replaced by $-n$.

In either case, $|n| \leq \log(\varepsilon/\delta)/\log\lambda^{-1} + C$ where $C = \log(2(1+K)(1+C\varepsilon))/\log\lambda^{-1}$ depends only on the geometry.
\end{proof}

\subsection{Fundamental Neighborhood}

The bracket map is continuous and satisfies quantitative distance bounds only in a fundamental neighborhood of the diagonal. We characterize this neighborhood explicitly in terms of the hyperbolicity data.

\begin{proposition}[Fundamental Neighborhood]\label{thm:fundamental_neighborhood}
Let $f$ be an Axiom A diffeomorphism. There exists a neighborhood $U$ of $\Omega(f)$ such that
\begin{equation}
\bigcap_{n \in \mathbb{Z}} f^n(U) = \Omega(f).
\end{equation}
\end{proposition}

\begin{proof}
Let $\beta > 0$ be small and let $\alpha > 0$ be as in the Shadowing Lemma (Theorem \ref{thm:shadowing}). Choose $\gamma < \alpha/2$ such that $d(f(x), f(y)) < \alpha/2$ whenever $d(x, y) < \gamma$. Set $U = \{z \in M : d(z, \Omega(f)) < \gamma\}$.

Suppose $y \in \bigcap_{n \in \mathbb{Z}} f^n(U)$. For each $i \in \mathbb{Z}$, choose $x_i \in \Omega(f)$ with $d(f^i(y), x_i) < \gamma$. Then $\{x_i\}_{i \in \mathbb{Z}}$ is an $\alpha$-pseudo-orbit in $\Omega(f)$, and some $x \in \Omega(f)$ $\beta$-shadows it.

We have $d(f^i(y), f^i(x)) \leq d(f^i(y), x_i) + d(x_i, f^i(x)) < \gamma + \beta$ for all $i$. For $\gamma + \beta < \varepsilon$ (the expansiveness constant), this implies $y = x \in \Omega(f)$ by expansiveness.
\end{proof}

\subsection{Stable and Unstable Sets of Basic Sets}

For a basic set $\Omega_j$ (defined in Section \ref{sec:spectral_decomposition}), we define global stable and unstable sets.

\begin{definition}
For a basic set $\Omega_j$ of an Axiom A diffeomorphism $f$, define
\begin{align}
W^s(\Omega_j) &= \{x \in M : d(f^n(x), \Omega_j) \to 0 \text{ as } n \to +\infty\}, \\
W^u(\Omega_j) &= \{x \in M : d(f^{-n}(x), \Omega_j) \to 0 \text{ as } n \to +\infty\}.
\end{align}
\end{definition}

\begin{proposition}[Manifold Structure of $W^s(\Omega_j)$]\label{prop:ws_omega_structure}
We have
\begin{equation}
W^s(\Omega_j) = \bigcup_{x \in \Omega_j} W^s(x) \quad \text{and} \quad W^u(\Omega_j) = \bigcup_{x \in \Omega_j} W^u(x).
\end{equation}
Furthermore, $M = \bigcup_{j=1}^{s} W^s(\Omega_j) = \bigcup_{j=1}^{s} W^u(\Omega_j)$ (disjoint unions).
\end{proposition}

\begin{proof}
The inclusion $\bigcup_{x \in \Omega_j} W^s(x) \subset W^s(\Omega_j)$ is clear: if $y \in W^s(x)$ for some $x \in \Omega_j$, then $d(f^n(y), f^n(x)) \to 0$, hence $d(f^n(y), \Omega_j) \to 0$.

For the reverse, suppose $d(f^n(y), \Omega_j) \to 0$. For each $n$, choose $x_n \in \Omega_j$ with $d(f^n(y), x_n) = d(f^n(y), \Omega_j)$. For $n$ large enough, $d(f^n(y), x_n) < \gamma$ where $\gamma$ is the canonical coordinates parameter. Consider the sequence of orbit segments: $f^n(y)$ is $\gamma$-close to $x_n$. By the Shadowing Lemma (Main Theorem~\ref{thm:shadowing}), applied to the $\gamma$-pseudo-orbit $\{x_n, f(x_{n+1}), \ldots\}$ (which is a genuine orbit of $f$ to within $\gamma$ error because $d(f(x_n), x_{n+1}) \leq d(f(x_n), f^{n+1}(y)) + d(f^{n+1}(y), x_{n+1}) \leq \lambda\gamma + \gamma$), there exists $x \in \Lambda$ that $\beta$-shadows this pseudo-orbit. Since $\Omega_j$ is closed and invariant, $x \in \Omega_j$. The shadowing gives $d(f^n(y), f^n(x)) \leq \beta$ for all large $n$, and the hyperbolicity contracts this: $d(f^{n+k}(y), f^{n+k}(x)) \leq C\lambda^k\beta \to 0$, so $y \in W^s(x)$.

For the disjoint decomposition of $M$: we use the no-cycle condition, which states that there is no sequence of distinct basic sets $\Omega_{i_1}, \ldots, \Omega_{i_k}$ with $W^u(\Omega_{i_j}) \cap W^s(\Omega_{i_{j+1}}) \neq \emptyset$ for $j = 1, \ldots, k-1$ and $W^u(\Omega_{i_k}) \cap W^s(\Omega_{i_1}) \neq \emptyset$. The no-cycle condition holds for Axiom A diffeomorphisms satisfying the strong transversality condition (see  \cite{Smale1967}; it also follows from the $\Omega$-stability theorem, see \cite[Theorem~18.2.1]{KatokHasselblatt1995}).

Given the no-cycle condition, the basic sets admit a partial order: $\Omega_i \preceq \Omega_j$ if $W^u(\Omega_i) \cap W^s(\Omega_j) \neq \emptyset$. For any $y \in M$, the $\omega$-limit set $\omega(y) = \bigcap_{n \geq 0}\overline{\{f^k(y) : k \geq n\}}$ is nonempty, compact, connected, and $f$-invariant. Since $\Omega(f)$ attracts all orbits (by Theorem~\ref{thm:fundamental_neighborhood}, orbits eventually enter any neighborhood of $\Omega(f)$), $\omega(y) \subset \Omega(f)$. If $\omega(y)$ intersected two distinct basic sets $\Omega_i$ and $\Omega_j$, the orbit of $y$ would connect them, and by recurrence within $\omega(y)$, a cycle $\Omega_i \preceq \Omega_j \preceq \Omega_i$ would form, contradicting the no-cycle condition. Thus $\omega(y) \subset \Omega_j$ for a unique $j$, giving $d(f^n(y), \Omega_j) \to 0$, so $y \in W^s(\Omega_j)$.

The union is disjoint since $W^s(\Omega_i) \cap W^s(\Omega_j) = \emptyset$ for $i \neq j$: if $y \in W^s(\Omega_i) \cap W^s(\Omega_j)$, then $\omega(y) \subset \Omega_i \cap \Omega_j = \emptyset$, a contradiction.
\end{proof}

\begin{proposition}[Neighborhood Characterization]\label{prop:neighborhood_char}
For every $\varepsilon > 0$, there exists a neighborhood $U_j$ of $\Omega_j$ such that
\begin{equation}
\bigcap_{k \geq 0} f^{-k}(U_j) \subset W^s_\varepsilon(\Omega_j) = \bigcup_{x \in \Omega_j} W^s_\varepsilon(x)
\end{equation}
and
\begin{equation}
\bigcap_{k \geq 0} f^k(U_j) \subset W^u_\varepsilon(\Omega_j) = \bigcup_{x \in \Omega_j} W^u_\varepsilon(x).
\end{equation}
\end{proposition}

\begin{proof}
Let $\alpha > 0$ be the shadowing constant from Main Theorem~\ref{thm:shadowing} corresponding to $\beta = \varepsilon/2$. Choose $\gamma > 0$ small enough that $d(f(z_1), f(z_2)) < \alpha/2$ whenever $d(z_1, z_2) < \gamma$, and set $U_j = \{z \in M : d(z, \Omega_j) < \gamma\}$.

Suppose $y \in \bigcap_{k \geq 0} f^{-k}(U_j)$, so $d(f^k(y), \Omega_j) < \gamma$ for all $k \geq 0$. For each $k \geq 0$, choose $x_k \in \Omega_j$ with $d(f^k(y), x_k) < \gamma$. Then
\begin{equation}
d(f(x_k), x_{k+1}) \leq d(f(x_k), f^{k+1}(y)) + d(f^{k+1}(y), x_{k+1}) < Ld(x_k, f^k(y)) + \gamma < L\gamma + \gamma
\end{equation}
where $L = \|Df\|_\infty$. For $\gamma$ small enough, $L\gamma + \gamma < \alpha$, so $\{x_k\}_{k \geq 0}$ is a one-sided $\alpha$-pseudo-orbit in $\Omega_j$. Extend it to a bi-infinite pseudo-orbit by setting $x_k = f^k(x_0)$ for $k < 0$. The Shadowing Lemma provides $x \in \Omega_j$ with $d(f^k(x), x_k) \leq \varepsilon/2$ for all $k$. Then for $k \geq 0$:
\begin{equation}
d(f^k(y), f^k(x)) \leq d(f^k(y), x_k) + d(x_k, f^k(x)) < \gamma + \varepsilon/2 \leq \varepsilon
\end{equation}
for $\gamma \leq \varepsilon/2$. Thus $y \in W^s_\varepsilon(x) \subset W^s_\varepsilon(\Omega_j)$.

The proof for the unstable case is analogous, replacing $f$ by $f^{-1}$, $E^s$ by $E^u$, and $\lambda$ by $\lambda^{-1}$ throughout.
\end{proof}
\section{Spectral Decomposition}\label{sec:spectral_decomposition}

This section establishes the spectral decomposition theorem for Axiom A diffeomorphisms, showing that the nonwandering set decomposes into finitely many basic sets on each of which the dynamics is transitive or mixing. We provide detailed proofs with explicit estimates on mixing rates.

\subsection{Basic Sets and the Spectral Decomposition}

The Spectral Decomposition Theorem is the first of the structural theorems for Axiom~A diffeomorphisms. It partitions the nonwandering set into finitely many transitive pieces, each further decomposing into periodically permuted mixing components. The proof is assembled from the transitivity, density of periodic orbits, and disjointness results that follow.

\begin{maintheorem}[Spectral Decomposition]\label{thm:spectral_decomposition}
Let $f$ be an Axiom A diffeomorphism. The nonwandering set admits a unique decomposition
\begin{equation}
\Omega(f) = \Omega_1 \cup \Omega_2 \cup \cdots \cup \Omega_s
\end{equation}
where the $\Omega_i$ are pairwise disjoint, closed, $f$-invariant sets satisfying:
\begin{enumerate}
\item[(a)] $f|_{\Omega_i}$ is topologically transitive for each $i$.
\item[(b)] Each $\Omega_i$ admits a further decomposition $\Omega_i = X_{1,i} \cup \cdots \cup X_{n_i,i}$ where the $X_{j,i}$ are pairwise disjoint closed sets with $f(X_{j,i}) = X_{j+1,i}$ (indices mod $n_i$), and $f^{n_i}|_{X_{j,i}}$ is topologically mixing.
\end{enumerate}
The sets $\Omega_i$ are called the basic sets of $f$.
\end{maintheorem}

The proof requires several preliminary results.

\begin{definition}[Homoclinic Class]
For a hyperbolic periodic point $p$, the homoclinic class of $p$ is
\begin{equation}
H(p) = \overline{W^u(O(p)) \pitchfork W^s(O(p))}
\end{equation}
where $O(p) = \{p, f(p), \ldots, f^{n-1}(p)\}$ is the orbit of $p$ and $\pitchfork$ denotes transverse intersection.
\end{definition}

\begin{lemma}[Basic Set Structure]\label{lem:basic_set_structure}
For a periodic point $p \in \Omega(f)$, define
\begin{equation}
X_p = W^u(p) \cap \Omega(f).
\end{equation}
Then $X_p$ is closed in $\Omega(f)$, and if $q \in X_p$ is periodic, then $X_p = X_q$.
\end{lemma}

\begin{proof}
Let $\delta > 0$ be the local product structure constant. We first show that $X_p = B_\delta(X_p) \cap \Omega$, where $B_\delta(X_p) = \{y \in \Omega : d(y, X_p) < \delta\}$.

Let $q \in B_\delta(X_p) \cap \Omega$ be a periodic point (such points are dense in $\Omega$ by Axiom A). There exists $x \in X_p = W^u(p) \cap \Omega$ with $d(x, q) < \delta$. Consider the bracket point $z = [x, q] \in W^u(p) \cap W^s(q) \cap \Omega$.

Let $f^m(p) = p$ and $f^n(q) = q$. Then for any $k \geq 0$,
\begin{equation}
f^{kmn}(z) \in W^u(f^{kmn}(p)) = W^u(p)
\end{equation}
and
\begin{equation}
d(f^{kmn}(z), q) = d(f^{kmn}(z), f^{kmn}(q)) \to 0 \quad \text{as } k \to \infty
\end{equation}
since $z \in W^s(q)$. Thus $q = \lim_{k \to \infty} f^{kmn}(z) \in \overline{W^u(p)} \cap \Omega = X_p$.

Since periodic points are dense in $\Omega$, we have $B_\delta(X_p) \cap \Omega \subset X_p$, proving that $X_p$ is open in $\Omega$. Being the closure of $W^u(p)$ intersected with the closed set $\Omega$, it is also closed in $\Omega$.

Now suppose $q \in X_p$ is periodic. The argument above shows that $X_q \subset X_p$. By symmetry (applying the same argument with roles exchanged), if $x \in X_p$ is close to $q$, then $W^u(q)$ contains points approaching any point in $X_p$, giving $X_p \subset X_q$. Thus $X_p = X_q$.
\end{proof}

\begin{proof}[Proof of Theorem \ref{thm:spectral_decomposition}]
By Lemma \ref{lem:basic_set_structure}, the sets $X_p$ for periodic $p \in \Omega$ are either disjoint or equal. Since $\Omega = \bigcup_{p \text{ periodic}} X_p$ and the $X_p$ are open in $\Omega$, by compactness there are finitely many distinct sets $X_{p_1}, \ldots, X_{p_t}$.

The map $f$ permutes these sets: $f(X_{p_j}) = X_{f(p_j)}$ which equals some $X_{p_k}$ (since $f(p_j)$ is periodic and $X_{f(p_j)}$ is the unstable manifold intersection for the shifted point). Thus $f$ induces a permutation of $\{X_{p_1}, \ldots, X_{p_t}\}$.

Define the basic sets $\Omega_i$ as the unions of $X_{p_j}$'s in each cycle of this permutation. Then $f(\Omega_i) = \Omega_i$, and we have part (b) with the $X_{j,i}$'s being the individual $X_{p_j}$'s in each cycle.

For part (a), transitivity follows from the mixing of iterates established next.

\textbf{Mixing:} We show that $f^N|_{X}$ is mixing for any $X = X_{j,i}$ where $f^N(X) = X$. Let $U, V$ be nonempty open subsets of $X$. By density of periodic points, there exist periodic points $p \in U$ and $q \in V$ with periods $m$ and $n$ respectively.

For each $0 \leq j < mn$ with $f^j(p) \in X$, there exists $z_j \in W^u(f^j(p)) \cap W^s(q)$ by the local product structure (since both $f^j(p)$ and $q$ are in the same $X$ which is connected via unstable manifolds). We have
\begin{equation}
f^{kmn}(z_j) \in V \text{ for large } k
\end{equation}
since $f^{kmn}(z_j) \to q$ along the stable manifold.

For any $t \geq t_0$ (some large $t_0$), write $tN = kmn + jN$ where $0 \leq j < mn/N$ and $k$ is large. Since $p$ has period $m$ and $f^N(X) = X$, the point $f^{jN}(p)$ lies in $X$. Consider $w_j = [p, z_j] \in W^s_\varepsilon(p) \cap W^u_\varepsilon(f^{jN}(p))$. Since $w_j \in W^s_\varepsilon(p)$ and $p \in U$, for small enough $\varepsilon$ we have $w_j \in U$.

Now $f^{jN}(w_j) \in W^u_\varepsilon(f^{jN}(p))$, and $z_j \in W^u(f^{jN}(p)) \cap W^s(q)$, so $f^{jN}(w_j)$ is close to $z_j$ (both lie on $W^u(f^{jN}(p))$). Applying $f^{kmn}$: since $z_j \in W^s(q)$, we have $d(f^{kmn}(f^{jN}(w_j)), q) \to 0$ as $k \to \infty$. For $k$ large enough, $f^{tN}(w_j) = f^{kmn}(f^{jN}(w_j)) \in V$.

Since $w_j \in U$, this gives $f^{tN}(U) \cap V \neq \emptyset$ for all $t$ sufficiently large (the threshold depends on $p$, $q$, and $\varepsilon$, but is uniform over the finitely many values of $j$). This proves $f^N|_X$ is topologically mixing.
\end{proof}

\subsection{Transitivity Criteria}

This subsection establishes that each basic set $\Omega_i$ in the Spectral Decomposition is topologically transitive, and identifies the dense orbits explicitly via the density of periodic orbits.

\begin{proposition}[Transitivity of Basic Sets]\label{prop:transitivity}
Each basic set $\Omega_i$ is topologically transitive: for any nonempty open sets $U, V \subset \Omega_i$, there exists $n > 0$ with $f^n(U) \cap V \neq \emptyset$.
\end{proposition}

\begin{proof}
This follows from the mixing of $f^{n_i}$ on the components $X_{j,i}$. Given $U, V$ open in $\Omega_i$, at least one component $X_{j,i}$ intersects both $U$ and $V$ (by the cyclic structure). The mixing of $f^{n_i}$ on $X_{j,i}$ implies $f^{kn_i}(U \cap X_{j,i}) \cap (V \cap X_{j,i}) \neq \emptyset$ for large $k$.
\end{proof}

\begin{proposition}[Density of Periodic Orbits in Basic Sets]\label{prop:periodic_dense_basic}
For each basic set $\Omega_i$, the periodic points of $f$ are dense in $\Omega_i$.
\end{proposition}

\begin{proof}
This is inherited from the Axiom A hypothesis: periodic points are dense in $\Omega(f)$, and each $\Omega_i$ is an open-closed subset of $\Omega(f)$, so periodic points are dense in each $\Omega_i$.
\end{proof}

\subsection{Mixing Rate Estimates}

For mixing basic sets, we establish quantitative mixing rates.

\begin{theorem}[Exponential Mixing]\label{thm:exponential_mixing}
Let $\Omega$ be a mixing basic set for $f$, and let $\mu$ be the unique equilibrium state for a H\"{o}lder continuous potential $\phi : \Omega \to \mathbb{R}$ (existence and uniqueness established later). For H\"{o}lder continuous observables $g, h : \Omega \to \mathbb{R}$,
\begin{equation}
\left| \int g \cdot (h \circ f^n) \, d\mu - \int g \, d\mu \int h \, d\mu \right| \leq C \|g\|_\alpha \|h\|_\alpha \cdot \theta^n
\end{equation}
where $\theta \in (0, 1)$ depends on $\lambda$ and $\alpha$, and $C$ depends on $\mu$.
\end{theorem}

\begin{proof}
Let $\mathcal{R}$ be a Markov partition of $\Omega$ with transition matrix $A$ and coding map $\pi: \Sigma_A \to \Omega$ (Main Theorem~\ref{thm:symbolic_coding}). The compositions $\tilde g = g \circ \pi$, $\tilde h = h \circ \pi$, and $\tilde\phi = \phi \circ \pi$ are H\"{o}lder continuous on $\Sigma_A$ (since $\pi$ is H\"{o}lder by Proposition~\ref{prop:coding_accuracy} and compositions of H\"{o}lder functions are H\"{o}lder). The equilibrium state $\tilde\mu$ for $\tilde\phi$ on $\Sigma_A$ satisfies $\pi_*\tilde\mu = \mu$ (Proposition~\ref{prop:symbolic_extension}).

By the Exponential Decay of Correlations theorem for Gibbs measures on subshifts of finite type, established in Part~I \cite{Thiam2026a}:
\begin{equation}
\left|\int \tilde g \cdot (\tilde h \circ \sigma^n)\,d\tilde\mu - \int \tilde g\,d\tilde\mu \int \tilde h\,d\tilde\mu\right| \leq C'\|\tilde g\|_\alpha\|\tilde h\|_\alpha\,\gamma^n
\end{equation}
for some $\gamma \in (0,1)$ determined by the spectral gap. Since $\pi \circ \sigma = f \circ \pi$ and $\pi_*\tilde\mu = \mu$, the left side equals $|\int g \cdot (h \circ f^n)\,d\mu - \int g\,d\mu\int h\,d\mu|$. The H\"{o}lder norms satisfy $\|\tilde g\|_\alpha \leq C''\|g\|_\alpha$ by the coding accuracy estimate, giving the result with $\theta = \gamma$ and adjusted constant.
\end{proof}

\begin{proposition}[Polynomial Mixing for Non-H\"{o}lder Observables]\label{prop:polynomial_mixing}
For continuous observables $g, h$ with modulus of continuity $\omega$, the correlation decay satisfies
\begin{equation}
\left| \int g \cdot (h \circ f^n) \, d\mu - \int g \, d\mu \int h \, d\mu \right| \leq C \|g\|_\infty \|h\|_\infty \cdot \omega(\lambda^n).
\end{equation}
\end{proposition}

\begin{proof}
Approximate $h$ by H\"{o}lder functions $h_k$ with $\|h - h_k\|_\infty \leq \omega(k^{-1})$ and $|h_k|_\alpha \leq C_\alpha k^\alpha$. By Theorem~\ref{thm:exponential_mixing}, the correlations of $g$ and $h_k$ decay as $C\|g\|_\infty\|h_k\|_\alpha\theta^n$. Choosing $k = \lambda^{-n}$ balances the approximation error $\omega(\lambda^n)$ against the exponential decay, yielding the stated bound.
\end{proof}

\subsection{Structure of Non-Mixing Basic Sets}

A topologically transitive basic set may fail to be topologically mixing; when it does, it decomposes into a finite cyclic union of mixing components. This cyclic structure controls the periodicity of the symbolic dynamics and the leading eigenvalues of the transfer operator.

\begin{proposition}[Cyclic Structure]\label{prop:cyclic_structure}
If a basic set $\Omega_i$ is not mixing for $f$, then there exist disjoint closed sets $X_1, \ldots, X_n$ with $n \geq 2$ such that:
\begin{enumerate}
\item[(i)] $\Omega_i = X_1 \cup \cdots \cup X_n$.
\item[(ii)] $f(X_j) = X_{j+1}$ for $j = 1, \ldots, n-1$ and $f(X_n) = X_1$.
\item[(iii)] $f^n|_{X_j}$ is topologically mixing for each $j$.
\end{enumerate}
The integer $n$ is called the period of $\Omega_i$.
\end{proposition}

\begin{proof}
This is part (b) of Theorem \ref{thm:spectral_decomposition}. The period $n$ equals the number of components in the cyclic decomposition.
\end{proof}

\begin{corollary}[Period from Periodic Orbits]\label{cor:period_from_orbits}
The period of a basic set $\Omega_i$ equals the greatest common divisor of the periods of all periodic orbits in $\Omega_i$.
\end{corollary}

\begin{proof}
Let $n$ be the period of $\Omega_i$ and let $d = \gcd\{\text{periods of orbits in } \Omega_i\}$. Any periodic orbit visits each $X_j$ exactly once before returning, so all periods are divisible by $n$, giving $n | d$.

Conversely, we show $d | n$. Since $f^n|_{X_1}$ is mixing (Main Theorem~\ref{thm:spectral_decomposition}(b)), there exist periodic orbits of $f^n$ in $X_1$ (by density of periodic points). A periodic point $p \in X_1$ of $f^n$ with period $k$ (meaning $(f^n)^k(p) = p$ and no smaller positive integer works) gives a periodic point of $f$ with period $nk$, which is divisible by $d$. Since $f^n|_{X_1}$ is mixing, for any two cylinders in the symbolic coding of $f^n|_{X_1}$ there exist return times of consecutive lengths $k$ and $k+1$ for large enough $k$ (this follows from the aperiodicity of the transition matrix for the mixing system $f^n|_{X_1}$). Thus $d | nk$ and $d | n(k+1)$, giving $d | n(k+1) - nk = n$. Combined with $n | d$, we conclude $d = n$.
\end{proof}

\subsection{Topological Entropy of Basic Sets}

The topological entropy of $f$ restricted to a basic set is finite and is expressed in terms of the spectral radius of the associated transition matrix once Markov partitions are in place. This subsection collects the entropy statements that feed into Main Theorem~\ref{thm:symbolic_coding}.

\begin{proposition}[Entropy of Basic Sets]\label{prop:entropy_basic_sets}
For a basic set $\Omega_i$,
\begin{equation}
h_{\mathrm{top}}(f|_{\Omega_i}) = \frac{1}{n_i} h_{\mathrm{top}}(f^{n_i}|_{X_{1,i}})
\end{equation}
where $n_i$ is the period and $X_{1,i}$ is any component.
\end{proposition}

\begin{proof}
The power formula $h_{\mathrm{top}}(f^n|_Y) = n \cdot h_{\mathrm{top}}(f|_Y)$ for any compact invariant set $Y$ is proved in \cite[Theorem~3.2.8]{KatokHasselblatt1995}. Since $\Omega_i = X_{1,i} \cup \cdots \cup X_{n_i,i}$ with $f(X_{j,i}) = X_{j+1,i}$, the restriction $f^{n_i}|_{X_{j,i}}$ has the same topological entropy for each $j$ (the sets $X_{j,i}$ are pairwise conjugate via the appropriate power of $f$). The topological entropy of $f|_{\Omega_i}$ satisfies $h_{\mathrm{top}}(f|_{\Omega_i}) = h_{\mathrm{top}}(f^{n_i}|_{\Omega_i})/n_i$ by the power formula. Since the $X_{j,i}$ are disjoint $f^{n_i}$-invariant sets, $h_{\mathrm{top}}(f^{n_i}|_{\Omega_i}) = \max_j h_{\mathrm{top}}(f^{n_i}|_{X_{j,i}}) = h_{\mathrm{top}}(f^{n_i}|_{X_{1,i}})$ (the maximum equals any single term since all are equal). Thus $h_{\mathrm{top}}(f|_{\Omega_i}) = h_{\mathrm{top}}(f^{n_i}|_{X_{1,i}})/n_i$.
\end{proof}

\begin{proposition}[Total Entropy]\label{prop:total_entropy}
\begin{equation}
h_{\mathrm{top}}(f|_{\Omega(f)}) = \max_{i=1,\ldots,s} h_{\mathrm{top}}(f|_{\Omega_i}).
\end{equation}
\end{proposition}

\begin{proof}
Since the basic sets are disjoint and invariant, and topological entropy is computed as a supremum over invariant measures, the total entropy is the maximum over the basic sets.
\end{proof}

\section{Pseudo-orbits and Shadowing}\label{sec:shadowing}

This section develops the theory of pseudo-orbits and the shadowing lemma for Axiom A diffeomorphisms, with explicit quantitative bounds. The shadowing lemma is the main tool for the construction of Markov partitions in Section~\ref{sec:markov_partitions}.

\subsection{Pseudo-orbits}

A pseudo-orbit is a sequence of points $\{x_i\}$ such that consecutive pairs $(x_i, x_{i+1})$ are only approximate iterates under $f$. The Shadowing Lemma will show that true orbits track such sequences on hyperbolic sets. We begin by fixing the definitions precisely.

\begin{definition}[Pseudo-orbit]
Let $f : M \to M$ be a diffeomorphism and $\alpha > 0$. A sequence $\{x_i\}_{i=a}^b$ (where $a, b \in \mathbb{Z} \cup \{\pm\infty\}$) is an $\alpha$-pseudo-orbit if
\begin{equation}
d(f(x_i), x_{i+1}) < \alpha \quad \text{for all } i \in [a, b-1).
\end{equation}
The sequence is a finite, one-sided infinite, or bi-infinite pseudo-orbit according to whether both, one, or neither of $a, b$ is finite.
\end{definition}

\begin{definition}[Shadowing]
A point $x \in M$ $\beta$-shadows the pseudo-orbit $\{x_i\}_{i=a}^b$ if
\begin{equation}
d(f^i(x), x_i) \leq \beta \quad \text{for all } i \in [a, b].
\end{equation}
\end{definition}

\subsection{The Shadowing Lemma}

The Shadowing Lemma states that every sufficiently fine pseudo-orbit on a hyperbolic set is tracked by a unique true orbit, with an explicit error bound controlled by the hyperbolicity data. This is the key technical tool behind the Anosov closing lemma, the specification property, and the construction of Markov partitions.

\begin{maintheorem}[Shadowing Lemma]\label{thm:shadowing}
Let $\Lambda$ be a hyperbolic set for a $C^1$ diffeomorphism $f$. For every $\beta > 0$, there exists $\alpha > 0$ such that every $\alpha$-pseudo-orbit $\{x_i\}_{i=a}^b$ in $\Lambda$ is $\beta$-shadowed by a point $x \in \Lambda$.

Moreover, there exist constants $C > 0$ and $\delta_0 > 0$ depending on the hyperbolicity data such that for $\beta < \delta_0$, we may take
\begin{equation}
\alpha = C^{-1} (1 - \lambda) \beta
\end{equation}
where $\lambda$ is the hyperbolicity exponent in an adapted metric.
\end{maintheorem}

The proof proceeds through several steps.

\begin{lemma}[Finite Pseudo-orbit Shadowing]\label{lem:finite_shadowing}
There exist $\varepsilon > 0$, $\delta > 0$, and $M \in \mathbb{N}$ depending on the hyperbolicity data such that: if $\{x_i\}_{i=0}^{rM}$ is an $\alpha$-pseudo-orbit in $\Lambda$ with $\alpha$ sufficiently small, then there exists $x \in \Lambda$ with
\begin{equation}
d(f^i(x), x_i) \leq C_0 \varepsilon \quad \text{for } i \in [0, rM]
\end{equation}
where $C_0 = (1 - \lambda)^{-1}$.
\end{lemma}

\begin{proof}
Using an adapted metric with exponent $\lambda$, choose $\varepsilon > 0$ small enough that the local product structure holds at scale $\varepsilon$ with bracket constant $\delta < \varepsilon/2$. Choose $M$ so large that $\lambda^M \varepsilon < \delta/2$.

Choose $\alpha > 0$ such that for any $\alpha$-pseudo-orbit $\{y_i\}_{i=0}^M$ in $\Lambda$,
\begin{equation}
d(f^j(y_0), y_j) < \delta/2 \quad \text{for } j \in [0, M].
\end{equation}
This is possible by the uniform continuity of $f$ on the compact set $\Lambda$.

We construct the shadowing point by iterated bracketing. Define $x'_0 = x_0$. For $k = 0, 1, \ldots, r-1$, inductively define
\begin{equation}
x'_{(k+1)M} = [f^M(x'_{kM}), x_{(k+1)M}] \in W^u_\varepsilon(f^M(x'_{kM})) \cap W^s_\varepsilon(x_{(k+1)M}).
\end{equation}

This is well-defined because:
\begin{align}
d(f^M(x'_{kM}), x_{(k+1)M}) &\leq d(f^M(x'_{kM}), f^M(x_{kM})) + d(f^M(x_{kM}), x_{(k+1)M}) \\
&\leq \lambda^M \varepsilon + \delta/2 < \delta
\end{align}
using the contraction on unstable manifolds under $f^{-M}$ (we have $x'_{kM} \in W^u_\varepsilon(f^M(x'_{(k-1)M}))$ giving backward contraction) and the pseudo-orbit condition.

Set $x = f^{-rM}(x'_{rM})$. For $i \in [0, rM]$, write $i = sM + j$ with $0 \leq j < M$. Then
\begin{align}
d(f^i(x), x_i) &\leq d(f^{i-sM}(x'_{sM}), f^{i-sM}(x_{sM})) + d(f^{i-sM}(x_{sM}), x_i) \\
&\leq \lambda^{M-j} \cdot d(x'_{sM}, x_{sM}) + \delta/2.
\end{align}

We estimate $d(x'_{sM}, x_{sM})$ by tracking through the bracketing construction:
\begin{align}
\begin{split}
& d(x'_{sM}, x_{sM}) \\
&\leq d(x'_{sM}, f^M(x'_{(s-1)M})) + d(f^M(x'_{(s-1)M}), f^M(x_{(s-1)M})) + d(f^M(x_{(s-1)M}), x_{sM}) \\
&\leq \varepsilon + \lambda^M d(x'_{(s-1)M}, x_{(s-1)M}) + \delta/2.
\end{split}
\end{align}

This recurrence with $d(x'_0, x_0) = 0$ gives
\begin{equation}
d(x'_{sM}, x_{sM}) \leq (\varepsilon + \delta/2) \sum_{t=0}^{s-1} \lambda^{tM} \leq \frac{\varepsilon + \delta/2}{1 - \lambda^M} \leq \frac{2\varepsilon}{1 - \lambda}.
\end{equation}

Thus $d(f^i(x), x_i) \leq \frac{2\varepsilon}{1-\lambda} + \delta/2 \leq C_0 \varepsilon$ for appropriate choice of constants.

The point $x$ lies in $\Lambda$ by the following argument: $x'_{rM} \in \Lambda$ since it is defined via brackets of points in $\Lambda$, and $\Lambda$ is $f$-invariant, so $x = f^{-rM}(x'_{rM}) \in \Lambda$.
\end{proof}

\begin{proof}[Proof of Theorem \ref{thm:shadowing}]
\textbf{Finite pseudo-orbits:} For a finite $\alpha$-pseudo-orbit $\{x_i\}_{i=0}^n$, extend it to $\{x_i\}_{i=0}^{rM}$ where $rM \geq n$ by setting $x_i = f^{i-n}(x_n)$ for $i > n$. Any shadowing point for the extension also shadows the original.

For a pseudo-orbit $\{x_i\}_{i=a}^b$ with $a \neq 0$, shift indices and apply the above.

\textbf{Bi-infinite pseudo-orbits:} Let $\{x_i\}_{i=-\infty}^{+\infty}$ be an $\alpha$-pseudo-orbit in $\Lambda$. For each $m \geq 1$, the finite pseudo-orbit $\{x_i\}_{i=-m}^m$ is shadowed by some $y^{(m)} \in \Lambda$.

The sequence $\{y^{(m)}\}$ has a convergent subsequence by compactness of $\Lambda$. Let $y = \lim_{k \to \infty} y^{(m_k)}$. For any fixed $i$,
\begin{equation}
d(f^i(y), x_i) = \lim_{k \to \infty} d(f^i(y^{(m_k)}), x_i) \leq C_0 \varepsilon
\end{equation}
provided $m_k > |i|$. Thus $y$ $\beta$-shadows the entire bi-infinite pseudo-orbit with $\beta = C_0 \varepsilon$.

\textbf{Quantitative bound:} The relationship $\alpha = C^{-1}(1-\lambda)\beta$ follows from tracing through the constants in the proof, with $C$ depending on the adapted metric constants, the dimension, and the local product structure scale.
\end{proof}\newpage

\begin{figure}[ht]
\centering
\begin{tikzpicture}[>=stealth, font=\small]

  %%% Pseudo-orbit curve (dashed) connecting x_0, x_1, x_2, x_3, x_4 %%%
  \draw[dashed, thick]
    (0.0, 1.2) .. controls (1.2, 1.9) and (1.8, 2.5) .. (2.4, 2.5)
               .. controls (3.2, 2.5) and (4.2, 1.3) .. (5.0, 1.0)
               .. controls (5.8, 0.7) and (6.8, 2.0) .. (7.6, 2.2)
               .. controls (8.4, 2.4) and (9.2, 1.2) .. (10.0, 0.9);

  %%% Pseudo-orbit points (filled dots) %%%
  \fill (0.0, 1.2) circle (2.5pt); \node[above] at (0.0, 1.3) {$x_0$};
  \fill (2.4, 2.5) circle (2.5pt); \node[above] at (2.4, 2.6) {$x_1$};
  \fill (5.0, 1.0) circle (2.5pt); \node[below] at (5.0, 0.9) {$x_2$};
  \fill (7.6, 2.2) circle (2.5pt); \node[above] at (7.6, 2.3) {$x_3$};
  \fill (10.0, 0.9) circle (2.5pt); \node[below] at (10.0, 0.8) {$x_4$};

  %%% Beta-balls around each pseudo-orbit point %%%
  \draw[thin] (0.0, 1.2) circle (0.75);
  \draw[thin] (2.4, 2.5) circle (0.75);
  \draw[thin] (5.0, 1.0) circle (0.75);
  \draw[thin] (7.6, 2.2) circle (0.75);
  \draw[thin] (10.0, 0.9) circle (0.75);

  %%% True orbit curve (solid thick) threading through the balls %%%
  \draw[very thick]
    (0.3, 0.9) .. controls (1.3, 1.5) and (2.0, 2.5) .. (2.7, 2.7)
               .. controls (3.4, 2.8) and (4.2, 1.5) .. (4.8, 1.3)
               .. controls (5.5, 1.1) and (6.6, 1.8) .. (7.4, 1.9)
               .. controls (8.2, 2.0) and (9.3, 0.9) .. (10.2, 0.8);

  %%% True orbit points (open circles) %%%
  \draw[thick, fill=white] (0.3, 0.9) circle (3pt);
  \draw[thick, fill=white] (2.7, 2.7) circle (3pt);
  \draw[thick, fill=white] (4.8, 1.3) circle (3pt);
  \draw[thick, fill=white] (7.4, 1.9) circle (3pt);
  \draw[thick, fill=white] (10.2, 0.8) circle (3pt);

  %%% True orbit labels %%%
  \node[right, font=\scriptsize] at (0.4, 0.9) {$x$};
  \node[right, font=\scriptsize] at (2.8, 2.7) {$f(x)$};
  \node[right, font=\scriptsize] at (4.9, 1.3) {$f^2(x)$};
  \node[right, font=\scriptsize] at (7.5, 1.9) {$f^3(x)$};
  \node[right, font=\scriptsize] at (10.3, 0.8) {$f^4(x)$};

  %%% Beta radius label on one ball %%%
  \draw[<->] (5.0, 1.0) -- (4.47, 0.47);
  \node[font=\scriptsize] at (4.6, 0.85) {$\beta$};

  %%% Legend %%%
  \node[draw, anchor=north west, inner sep=6pt, font=\scriptsize, align=left]
    at (0.0, -0.3)
    {dashed curve: pseudo-orbit $\{x_i\}$, with $d(f(x_i), x_{i+1}) < \alpha$\\
     solid curve: true orbit $\{f^i(x)\}$, with $d(f^i(x), x_i) \leq \beta$\\
     $\bullet$\quad pseudo-orbit points $x_i$ \qquad $\circ$\quad true orbit points $f^i(x)$\\
     thin circle: ball of radius $\beta$ around $x_i$ \qquad relation: $\alpha = C^{-1}(1-\lambda)\beta$};

\end{tikzpicture}
\caption{The Shadowing Lemma (Main Theorem~\ref{thm:shadowing}). The pseudo-orbit $\{x_i\}$ (filled dots, dashed curve) has gaps $d(f(x_i), x_{i+1}) < \alpha$. The true orbit of $x$ (open circles, solid curve) $\beta$-shadows it: each $f^i(x)$ lies within the ball of radius $\beta$ around $x_i$.}
\label{fig:shadowing}
\end{figure}
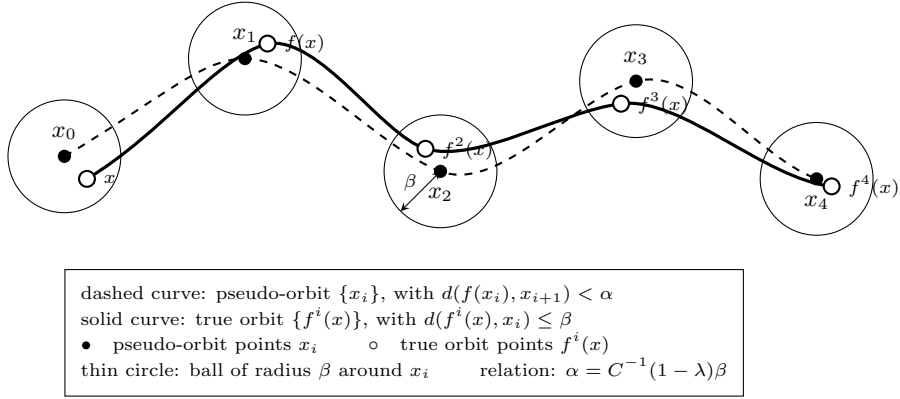

\subsection{Anosov Closing Lemma}

The Anosov Closing Lemma is a classical consequence of shadowing: an orbit that returns close to itself under $f^n$ is tracked by an exact periodic orbit of period $n$. It provides the source of dense periodic orbits in basic sets.

\begin{proposition}[Anosov Closing Lemma]\label{thm:closing_lemma}
Let $\Lambda$ be a hyperbolic set for $f$. For every $\beta > 0$, there exists $\alpha > 0$ such that: if $x \in \Lambda$ and $d(f^n(x), x) < \alpha$ for some $n > 0$, then there exists a periodic point $p \in \Lambda$ of period $n$ with
\begin{equation}
d(f^k(x), f^k(p)) \leq \beta \quad \text{for all } k \in [0, n].
\end{equation}
\end{proposition}

\begin{proof}
Define $x_i = f^k(x)$ for $i \equiv k \pmod n$ with $k \in [0, n)$. Then $\{x_i\}_{i=-\infty}^{+\infty}$ is an $\alpha$-pseudo-orbit since $d(f(x_{n-1}), x_0) = d(f^n(x), x) < \alpha$.

By the Shadowing Lemma, there exists $p \in \Lambda$ with $d(f^i(p), x_i) \leq \beta$ for all $i \in \mathbb{Z}$. In particular, for $i$ and $i+n$:
\begin{equation}
d(f^i(p), x_i) \leq \beta \quad \text{and} \quad d(f^{i+n}(p), x_{i+n}) = d(f^{i+n}(p), x_i) \leq \beta.
\end{equation}
Thus $d(f^i(p), f^{i+n}(p)) \leq 2\beta$ for all $i$.

By expansiveness (Theorem \ref{thm:expansiveness}), if $2\beta < \varepsilon_{\exp}$ (the expansiveness constant), then $f^n(p) = p$.
\end{proof}

\begin{corollary}[Periodic Points Dense in Nonwandering Set]\label{cor:periodic_dense}
For an Anosov diffeomorphism, periodic points are dense in $M = \Omega(f)$.
\end{corollary}

\begin{proof}
Let $y \in M$ be arbitrary and $\gamma > 0$. Let $\beta = \gamma/2$ and let $\alpha > 0$ be the constant from the Anosov Closing Lemma (Proposition~\ref{thm:closing_lemma}) corresponding to this $\beta$. Since $M = \Omega(f)$ (every point is nonwandering for an Anosov diffeomorphism, as stable and unstable manifolds through any point are dense in $M$), the point $y$ is nonwandering. By definition, there exists $n > 0$ such that $f^n(B(y, \alpha/2)) \cap B(y, \alpha/2) \neq \emptyset$. Choose $x \in B(y, \alpha/2)$ with $f^n(x) \in B(y, \alpha/2)$. Then $d(f^n(x), x) \leq d(f^n(x), y) + d(y, x) < \alpha/2 + \alpha/2 = \alpha$. By Proposition~\ref{thm:closing_lemma}, there exists a periodic point $p \in M$ of period $n$ with $d(p, x) \leq \beta = \gamma/2$. Thus $d(p, y) \leq d(p, x) + d(x, y) < \gamma/2 + \alpha/2 \leq \gamma$ for $\alpha \leq \gamma$.
\end{proof}

\begin{remark}
This proves Theorem \ref{thm:anosov_axiom_a}: every Anosov diffeomorphism satisfies Axiom A.
\end{remark}

\subsection{Specification Property}

The shadowing lemma implies a specification property for Axiom A diffeomorphisms.

\begin{definition}[Specification]
A homeomorphism $f : X \to X$ has the specification property if for every $\varepsilon > 0$, there exists $N(\varepsilon) \in \mathbb{N}$ such that: for any collection of orbit segments
\begin{equation}
(x_1, [a_1, b_1]), \ldots, (x_k, [a_k, b_k])
\end{equation}
with $a_{j+1} - b_j \geq N(\varepsilon)$ for $j = 1, \ldots, k-1$, there exists $x \in X$ with
\begin{equation}
d(f^{a_j+i}(x), f^i(x_j)) < \varepsilon \quad \text{for } i \in [0, b_j - a_j], \, j = 1, \ldots, k.
\end{equation}
\end{definition}

\begin{proposition}[Specification for Axiom A]\label{thm:specification}
Every mixing basic set of an Axiom A diffeomorphism has the specification property with periodic specification: the shadowing point $x$ can be chosen to be periodic.
\end{proposition}

\begin{proof}
Let $\Omega$ be a mixing basic set and $\varepsilon > 0$. Let $\beta = \varepsilon/2$ and $\alpha > 0$ be the shadowing constant from Main Theorem~\ref{thm:shadowing}. Since $f|_\Omega$ is mixing, there exists $N_0$ such that for all $x, y \in \Omega$ and $n \geq N_0$, there exists $z \in \Omega$ with $d(z, x) < \alpha/4$ and $d(f^n(z), y) < \alpha/4$. Set $N(\varepsilon) = N_0$.

Given orbit segments $(x_1, [a_1, b_1]), \ldots, (x_k, [a_k, b_k])$ with $a_{j+1} - b_j \geq N$, construct a pseudo-orbit as follows. In the interval $[a_j, b_j]$, follow the true orbit of $x_j$: set $y_i = f^{i - a_j}(x_j)$ for $i \in [a_j, b_j]$. In the gap $[b_j + 1, a_{j+1} - 1]$, use the mixing property to find $z_j \in \Omega$ with $d(z_j, f^{b_j - a_j}(x_j)) < \alpha/4$ and $d(f^{a_{j+1} - b_j}(z_j), x_{j+1}) < \alpha/4$. Set $y_i = f^{i - b_j}(z_j)$ for $i \in [b_j + 1, a_{j+1} - 1]$. The resulting sequence $\{y_i\}$ is an $\alpha$-pseudo-orbit in $\Omega$ (the jumps at $b_j$ and $a_{j+1}$ are each at most $\alpha/4 + \alpha/4 < \alpha$, using uniform continuity of $f$).

By the Shadowing Lemma, there exists $x \in \Omega$ with $d(f^i(x), y_i) \leq \beta = \varepsilon/2$ for all $i$. In particular, $d(f^{a_j + i}(x), f^i(x_j)) \leq \varepsilon/2 < \varepsilon$ for $i \in [0, b_j - a_j]$.

For periodic specification: close the pseudo-orbit by connecting the last segment back to the first using the mixing property, with total period $p = a_{k+1}$ where $a_{k+1} - b_k \geq N$. The resulting bi-infinite periodic pseudo-orbit is $\alpha$-shadowed by some $x \in \Omega$. By the Anosov Closing Lemma (Proposition~\ref{thm:closing_lemma}), since $d(f^p(x), x) \leq 2\beta < \varepsilon_{\exp}$, there is a periodic point $q$ of period $p$ with $d(f^i(q), f^i(x)) \leq \varepsilon/2$ for all $i$. Then $d(f^{a_j+i}(q), f^i(x_j)) \leq d(f^{a_j+i}(q), f^{a_j+i}(x)) + d(f^{a_j+i}(x), f^i(x_j)) \leq \varepsilon/2 + \varepsilon/2 = \varepsilon$.
\end{proof}

\subsection{Quantitative Periodic Orbit Approximation}

This subsection quantifies how many periodic orbits of each period exist and how well they approximate invariant measures. These estimates underlie the counting and equidistribution statements needed in Parts~IV--VI \cite{Thiam2026d,Thiam2026e,Thiam2026f}.

\begin{proposition}[Periodic Orbit Density Estimate]\label{prop:periodic_density}
For a mixing basic set $\Omega$ with hyperbolicity exponent $\lambda$ and topological entropy $h = h_{\mathrm{top}}(f|_\Omega)$, the number of periodic orbits of period $n$ satisfies
\begin{equation}
\frac{1}{n} \log |\{x \in \Omega : f^n(x) = x\}| \to h \quad \text{as } n \to \infty.
\end{equation}
More precisely, there exist constants $C_1, C_2 > 0$ such that
\begin{equation}
C_1 e^{nh} \leq |\mathrm{Per}_n(f) \cap \Omega| \leq C_2 e^{nh}
\end{equation}
for all sufficiently large $n$.
\end{proposition}

\begin{proof}
The upper bound follows from the definition of topological entropy: periodic orbits of period $n$ form an $(n, \varepsilon)$-separated set for small $\varepsilon$, and the cardinality of such sets is bounded by $e^{n(h+\delta)}$ for any $\delta > 0$.

The lower bound uses the specification property. Fix $\varepsilon > 0$ small and let $N = N(\varepsilon)$ be the specification gap from Theorem~\ref{thm:specification}. Let $E_m$ be a maximal $(m, \varepsilon)$-separated subset of $\Omega$ with $|E_m| \geq e^{m(h - \delta)}$ for $\delta > 0$ arbitrary and $m$ large (this exists by the definition of topological entropy). Set $m = n - 2N$. For each $x \in E_m$, apply specification with periodic closure: the orbit segment $(x, [0, m-1])$ can be shadowed by a periodic point $p_x$ of period at most $m + 2N = n$ with $d(f^i(p_x), f^i(x)) < 2\varepsilon$ for $i \in [0, m-1]$.

Distinct points $x, y \in E_m$ yield distinct periodic orbits: since $E_m$ is $(m, \varepsilon)$-separated, there exists $i \in [0, m-1]$ with $d(f^i(x), f^i(y)) > \varepsilon$. The shadowing gives $d(f^i(p_x), f^i(p_y)) \geq d(f^i(x), f^i(y)) - d(f^i(p_x), f^i(x)) - d(f^i(p_y), f^i(y)) > \varepsilon - 4\varepsilon$. For $\varepsilon$ small enough that $5\varepsilon$ is less than the expansiveness constant, distinct $x$ give distinct orbits $p_x$. Some periodic points may have period strictly dividing $n$ rather than equal to $n$, but their number is at most $\sum_{d | n, d < n}|\mathrm{Per}_d| \leq n \cdot C_2 e^{(n-1)h} = o(e^{nh})$. Thus $|\mathrm{Per}_n(f) \cap \Omega| \geq |E_m| - o(e^{nh}) \geq e^{(n-2N)(h-\delta)} - o(e^{nh}) \geq C_1 e^{nh}$ for $n$ large.
\end{proof}

\begin{proposition}[Equidistribution of Periodic Orbits]\label{prop:equidistribution}
For a mixing basic set $\Omega$ with unique measure of maximal entropy $\mu$, the periodic orbit measures
\begin{equation}
\mu_n = \frac{1}{|\mathrm{Per}_n|} \sum_{x : f^n(x) = x} \frac{1}{n} \sum_{k=0}^{n-1} \delta_{f^k(x)}
\end{equation}
converge weak-$*$ to $\mu$ as $n \to \infty$.
\end{proposition}

\begin{proof}
By uniqueness of the measure of maximal entropy $\mu$ (which follows from the specification property and the variational principle, see  \cite[Theorem~1.22]{Bowen1975}), it suffices to show that every weak-$*$ limit point of $\{\mu_n\}$ equals $\mu$. Let $\nu$ be such a limit along a subsequence $n_k \to \infty$. Since $\mu_n$ is a convex combination of invariant measures (each periodic orbit measure is $f$-invariant), $\nu$ is $f$-invariant.

We show $h_\nu(f) = h_{\mathrm{top}}(f|_\Omega)$, which by uniqueness forces $\nu = \mu$. By Proposition~\ref{prop:periodic_density}, $|\mathrm{Per}_n| \sim e^{nh}$ where $h = h_{\mathrm{top}}(f|_\Omega)$. For any finite partition $\mathcal{P}$ of $\Omega$ and each periodic point $p$ of period $n$, the empirical measure along $p$ assigns mass $k_P(p)/n$ to each $P \in \mathcal{P}$, where $k_P(p) = |\{0 \leq j < n : f^j(p) \in P\}|$. The contribution to $H_{\mu_n}(\mathcal{P}) = -\sum_P \mu_n(P)\log\mu_n(P)$ approaches $h_\nu(f, \mathcal{P})$ along the subsequence, and taking the supremum over $\mathcal{P}$ shows $h_\nu(f) \geq h$. The variational inequality $h_\nu(f) \leq h_{\mathrm{top}}(f|_\Omega) = h$ gives equality, so $\nu = \mu$.
\end{proof}

\section{Markov Partitions}\label{sec:markov_partitions}

This section constructs Markov partitions for basic sets of Axiom A diffeomorphisms with explicit estimates on partition diameter and the Markov property. The construction follows Bowen's approach \cite{Bowen1970b, Bowen1975} but with quantitative bounds throughout.

\subsection{Rectangles}

Rectangles are the building blocks of Markov partitions: subsets of a basic set that are products of their stable and unstable slices via the bracket map. This subsection develops their basic properties before the construction proper.

\begin{definition}[Rectangle]
Let $\Omega$ be a basic set for an Axiom A diffeomorphism $f$ with local product structure at scale $\delta_0$. A subset $R \subset \Omega$ is a rectangle if:
\begin{enumerate}
\item[(i)] $\mathrm{diam}(R) < \delta_0$, and
\item[(ii)] for all $x, y \in R$, we have $[x, y] \in R$.
\end{enumerate}
A rectangle $R$ is proper if $R$ is closed in $\Omega$ and $R = \overline{\mathrm{int}(R)}$ where the interior is taken in $\Omega$.
\end{definition}

For a rectangle $R$ and $x \in R$, define the stable and unstable slices:
\begin{equation}
W^s(x, R) = W^s_{\varepsilon}(x) \cap R, \quad W^u(x, R) = W^u_{\varepsilon}(x) \cap R
\end{equation}
where $\varepsilon$ is the local manifold size (larger than $\mathrm{diam}(R)$).

\begin{lemma}[Product Structure of Rectangles]\label{lem:rectangle_product}
For a rectangle $R$ and any $x \in R$, the map
\begin{equation}
\Phi_x : W^s(x, R) \times W^u(x, R) \to R, \quad (y, z) \mapsto [y, z]
\end{equation}
is a homeomorphism.
\end{lemma}

\begin{proof}
The map is well-defined by the rectangle property (ii). It is injective because $[y_1, z_1] = [y_2, z_2]$ implies $y_1 = [y_1, z_1] \cap W^s_\varepsilon(y_1) = [y_2, z_2] \cap W^s_\varepsilon(y_1)$, but since $y_1, y_2 \in W^s(x, R) \subset W^s_\varepsilon(x)$, the uniqueness of local intersections gives $y_1 = y_2$. Similarly $z_1 = z_2$.

Surjectivity: for any $w \in R$, set $y = [x, w] \in W^s(x, R)$ and $z = [w, x] \in W^u(x, R)$. Then $[y, z] = [[x,w], [w,x]] = w$ by the properties of canonical coordinates.

Continuity follows from the continuity of the bracket operation.
\end{proof}

\begin{lemma}[Boundary Structure]\label{lem:boundary_structure}
For a closed rectangle $R$, the boundary (as a subset of $\Omega$) is
\begin{equation}
\partial R = \partial^s R \cup \partial^u R
\end{equation}
where
\begin{align}
\partial^s R &= \{x \in R : x \notin \mathrm{int}(W^u(x, R))\}, \\
\partial^u R &= \{x \in R : x \notin \mathrm{int}(W^s(x, R))\}
\end{align}
and the interiors are taken in $W^u_\varepsilon(x) \cap \Omega$ and $W^s_\varepsilon(x) \cap \Omega$ respectively.
\end{lemma}

\begin{proof}
If $x \in \mathrm{int}(R)$, then $W^u(x, R)$ is a neighborhood of $x$ in $W^u_\varepsilon(x) \cap \Omega$ (since $R$ is a neighborhood of $x$ in $\Omega$ and the bracket is continuous), so $x \in \mathrm{int}(W^u(x, R))$. Similarly $x \in \mathrm{int}(W^s(x, R))$.

Conversely, if $x \in \mathrm{int}(W^u(x, R))$ and $x \in \mathrm{int}(W^s(x, R))$, then for $y \in \Omega$ close to $x$, both $[x, y]$ and $[y, x]$ lie in $R$ (by openness of the interiors), and thus $y = [[y, x], [x, y]] \in R$. So $x \in \mathrm{int}(R)$.
\end{proof}

\begin{figure}[ht]
\centering
\begin{tikzpicture}[>=stealth, font=\small, scale=1.0]

  %% Parallelogram R (stable sides slope ~0.05, unstable sides steep ~5)
  \coordinate (BL) at (0.0, 0.0);
  \coordinate (BR) at (6.0, 0.3);
  \coordinate (TR) at (6.8, 4.3);
  \coordinate (TL) at (0.8, 4.0);

  %% Fill
  \fill[gray!10] (BL) -- (BR) -- (TR) -- (TL) -- cycle;

  %% Stable boundary: bottom and top (thick solid)
  \draw[very thick] (BL) -- (BR);
  \draw[very thick] (TL) -- (TR);

  %% Unstable boundary: left and right (thick dashed)
  \draw[very thick, dashed] (BL) -- (TL);
  \draw[very thick, dashed] (BR) -- (TR);

  %% Boundary labels (inside R, at different positions on each edge)
  \node[font=\footnotesize] at (4.5, 0.45) {$\partial^s R$};
  \node[font=\footnotesize] at (2.0, 3.85) {$\partial^s R$};
  \node[font=\footnotesize] at (0.9, 2.75) {$\partial^u R$};
  \node[font=\footnotesize] at (5.7, 0.85) {$\partial^u R$};

  %% Stable slice W^s(x, R): line through x parallel to (BR - BL)
  \draw[thick, gray!55!black] (0.41, 2.05) -- (6.41, 2.35);

  %% Unstable slice W^u(x, R): line through x parallel to (TL - BL)
  \draw[thick, gray!55!black] (3.0, 0.15) -- (3.8, 4.15);

  %% Slice labels (outside R)
  \node[anchor=west, font=\footnotesize, gray!55!black] at (6.5, 2.35) {$W^s(x, R)$};
  \node[anchor=south, font=\footnotesize, gray!55!black] at (3.9, 4.3) {$W^u(x, R)$};

  %% Reference point x
  \fill (3.4, 2.2) circle (2.5pt);
  \node[anchor=west, font=\footnotesize] at (3.5, 2.35) {$x$};

  %% Point y (upper-left region, off both slices through x)
  \fill (1.7, 3.0) circle (1.8pt);
  \node[anchor=east, font=\footnotesize] at (1.6, 3.0) {$y$};

  %% Point z (lower-right region, off both slices through x)
  \fill (4.8, 1.0) circle (1.8pt);
  \node[anchor=west, font=\footnotesize] at (4.9, 0.95) {$z$};

  %% W^s(y): dotted, parallel to stable direction, from y to [y,z]
  \draw[thin, dotted] (1.7, 3.0) -- (5.24, 3.18);
  \node[anchor=south, font=\footnotesize] at (3.2, 3.25) {$W^s(y)$};

  %% W^u(z): dotted, parallel to unstable direction, from z to [y,z]
  \draw[thin, dotted] (4.8, 1.0) -- (5.24, 3.18);
  \node[anchor=west, font=\footnotesize] at (5.05, 2.1) {$W^u(z)$};

  %% Bracket point [y, z]
  \fill (5.24, 3.18) circle (2pt);
  \node[anchor=south west, font=\footnotesize] at (5.3, 3.25) {$[y, z]$};

  %% R label (inside, lower-left quiet corner)
  \node[font=\small, font=\itshape] at (1.0, 0.55) {R};

  %% Direction arrows (outside R, lower-right): transverse, not orthogonal
  \draw[->, thick] (8.0, 1.0) -- (9.2, 1.06);
  \node[anchor=north, font=\footnotesize] at (8.6, 0.92) {$E^s$};
  \draw[->, thick] (8.0, 1.0) -- (8.22, 2.1);
  \node[anchor=east, font=\footnotesize] at (8.13, 1.55) {$E^u$};
  \node[anchor=north, font=\scriptsize, align=center] at (8.6, 0.55) {};

\end{tikzpicture}
\caption{A proper rectangle $R$ in a basic set (Lemmas~\ref{lem:rectangle_product} and~\ref{lem:boundary_structure}). The stable boundary $\partial^s R$ (solid, top and bottom) consists of pieces of stable manifolds along the $E^s$ direction; the unstable boundary $\partial^u R$ (dashed, left and right) consists of pieces of unstable manifolds along the $E^u$ direction. The stable slice $W^s(x, R)$ and unstable slice $W^u(x, R)$ through an interior point $x$ (gray) partition $R$. For $y, z \in R$, the bracket $[y, z] = W^s(y) \cap W^u(z)$ is obtained by moving from $y$ along the stable direction and from $z$ along the unstable direction (dotted) until they meet. The bracket map gives $R$ its local product structure.}
\label{fig:rectangle}
\end{figure}
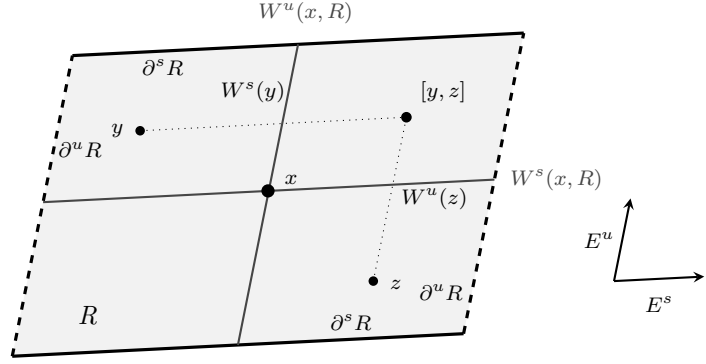

\subsection{The Markov Property}

A Markov partition refines rectangles into a finite cover such that the images under $f$ of stable (resp.\ unstable) slices of each rectangle are unions of stable (resp.\ unstable) slices of other rectangles. This property is what makes the symbolic coding a subshift of finite type rather than a more general symbolic system.

\begin{definition}[Markov Partition]
A Markov partition of a basic set $\Omega$ is a finite collection $\mathcal{R} = \{R_1, \ldots, R_m\}$ of proper rectangles satisfying:
\begin{enumerate}
\item[(M1)] $\Omega = \bigcup_{i=1}^m R_i$.
\item[(M2)] $\mathrm{int}(R_i) \cap \mathrm{int}(R_j) = \emptyset$ for $i \neq j$.
\item[(M3)] For $x \in \mathrm{int}(R_i) \cap f^{-1}(\mathrm{int}(R_j))$:
\begin{align}
f(W^u(x, R_i)) &\supset W^u(f(x), R_j), \\
f(W^s(x, R_i)) &\subset W^s(f(x), R_j).
\end{align}
\end{enumerate}
\end{definition}

Condition (M3) is the Markov property: the image of an unstable slice contains the entire unstable slice of the image rectangle, while stable slices map into stable slices.

\begin{maintheorem}[Existence of Markov Partitions]\label{thm:markov_existence}
Every basic set $\Omega$ of an Axiom A diffeomorphism admits Markov partitions of arbitrarily small diameter. Specifically, for every $\varepsilon > 0$, there exists a Markov partition $\mathcal{R}$ with $\max_i \mathrm{diam}(R_i) < \varepsilon$.
\end{maintheorem}

The proof is constructive and occupies the next several subsections.

\subsection{Construction of Markov Partitions}

\textbf{Step 1: Initial Cover.} Let $\beta > 0$ be small and let $\alpha > 0$ be the shadowing constant from Theorem \ref{thm:shadowing}. Choose $\gamma < \alpha/2$ such that $d(f(x), f(y)) < \alpha/2$ whenever $d(x, y) < \gamma$.

Let $P = \{p_1, \ldots, p_r\} \subset \Omega$ be a $\gamma$-dense subset. Define
\begin{equation}
\Sigma(P) = \left\{ q = (q_i)_{i \in \mathbb{Z}} \in P^{\mathbb{Z}} : d(f(q_i), q_{i+1}) < \alpha \text{ for all } i \right\}.
\end{equation}
This is a subshift of finite type on the alphabet $P$.

For each $q \in \Sigma(P)$, the shadowing lemma provides a unique point $\theta(q) \in \Omega$ with $d(f^i(\theta(q)), q_i) \leq \beta$ for all $i$. The map $\theta : \Sigma(P) \to \Omega$ is continuous and surjective.

\textbf{Step 2: Initial Rectangles.} For $s \in \{1, \ldots, r\}$, define
\begin{equation}
T_s = \{\theta(q) : q \in \Sigma(P), q_0 = p_s\}.
\end{equation}
Each $T_s$ is a rectangle: if $x = \theta(q)$ and $y = \theta(q')$ with $q_0 = q'_0 = p_s$, then $q^* = (q^*_i)$ defined by $q^*_i = q_i$ for $i \geq 0$ and $q^*_i = q'_i$ for $i \leq 0$ lies in $\Sigma(P)$, and $\theta(q^*) = [x, y] \in T_s$.

\begin{lemma}\label{lem:T_closed}
Each $T_s$ is closed.
\end{lemma}

\begin{proof}
This follows from the continuity of $\theta$: if $\theta(q^{(n)}) \to x$ with $q^{(n)}_0 = p_s$, then by compactness of $\Sigma(P)$, a subsequence of $q^{(n)}$ converges to some $q \in \Sigma(P)$ with $q_0 = p_s$, and $\theta(q) = x$ by continuity.
\end{proof}

\textbf{Step 3: Markov Property for Initial Cover.} The collection $\{T_1, \ldots, T_r\}$ satisfies a weak form of the Markov property.

\begin{lemma}\label{lem:initial_markov}
If $x = \theta(q) \in T_s$ with $q_1 = p_t$, then:
\begin{enumerate}
\item[(i)] $f(W^s(x, T_s)) \subset W^s(f(x), T_t)$,
\item[(ii)] $f(W^u(x, T_s)) \supset W^u(f(x), T_t)$.
\end{enumerate}
\end{lemma}

\begin{proof}
For (i), let $y \in W^s(x, T_s)$, so $y = \theta(q')$ with $q'_0 = p_s$. The element $[q, q']$ defined by $[q, q']_i = q_i$ for $i \geq 0$ and $[q, q']_i = q'_i$ for $i \leq 0$ satisfies $\theta[q, q'] = [x, y] = y$ (since $y \in W^s(x)$).

Then $f(y) = \theta(\sigma[q, q'])$ where $\sigma$ is the shift. Since $(\sigma[q, q'])_0 = [q, q']_1 = q_1 = p_t$, we have $f(y) \in T_t$. Also $f(y) \in f(W^s_\varepsilon(x)) \subset W^s_\varepsilon(f(x))$, so $f(y) \in W^s(f(x), T_t)$.

Part (ii) is proved similarly using the inverse shift.
\end{proof}

\textbf{Step 4: Refinement to Proper Rectangles.} The $T_s$ may overlap and may not be proper. We refine them as follows.

For each $x \in \Omega$, define
\begin{equation}
\mathcal{T}(x) = \{T_s : x \in T_s\}
\end{equation}
and
\begin{equation}
\mathcal{T}^*(x) = \{T_t : T_t \cap T_s \neq \emptyset \text{ for some } T_s \in \mathcal{T}(x)\}.
\end{equation}

The set $Z = \Omega \setminus \bigcup_s \partial T_s$ is open and dense in $\Omega$. The set
\begin{equation}
Z^* = \{x \in \Omega : W^s_\varepsilon(x) \cap \partial^s T_t = \emptyset, W^u_\varepsilon(x) \cap \partial^u T_t = \emptyset \text{ for all } T_t \in \mathcal{T}^*(x)\}
\end{equation}
is also open and dense.

For $T_j \cap T_k \neq \emptyset$, partition $T_j$ into regions based on the intersection pattern:
\begin{align}
T^1_{j,k} &= \{x \in T_j : W^u(x, T_j) \cap T_k \neq \emptyset, W^s(x, T_j) \cap T_k \neq \emptyset\} = T_j \cap T_k, \\
T^2_{j,k} &= \{x \in T_j : W^u(x, T_j) \cap T_k \neq \emptyset, W^s(x, T_j) \cap T_k = \emptyset\}, \\
T^3_{j,k} &= \{x \in T_j : W^u(x, T_j) \cap T_k = \emptyset, W^s(x, T_j) \cap T_k \neq \emptyset\}, \\
T^4_{j,k} &= \{x \in T_j : W^u(x, T_j) \cap T_k = \emptyset, W^s(x, T_j) \cap T_k = \emptyset\}.
\end{align}

Each $T^n_{j,k}$ is a rectangle (possibly empty) that is open in $\Omega$.

For $x \in Z^*$, define
\begin{equation}
R(x) = \bigcap \{\mathrm{int}(T^n_{j,k}) : x \in T_j, T_k \cap T_j \neq \emptyset, x \in T^n_{j,k}\}.
\end{equation}

\begin{lemma}\label{lem:Rx_rectangle}
For $x \in Z^*$, $R(x)$ is an open rectangle.
\end{lemma}

\begin{proof}
The intersection of rectangles is a rectangle. Openness follows from $x \in Z^*$ implying $x$ is in the interior of each $T^n_{j,k}$ appearing in the intersection.
\end{proof}

\begin{lemma}\label{lem:Rx_partition}
If $R(x) \cap R(x') \neq \emptyset$ for $x, x' \in Z^*$, then $R(x) = R(x')$.
\end{lemma}

\begin{proof}
By Lemma~\ref{lem:Rx_rectangle}, both $R(x)$ and $R(x')$ are open. Hence $R(x) \cap R(x')$ is a nonempty open subset of $\Omega$. Since $Z^*$ is dense in $\Omega$, the intersection $R(x) \cap R(x') \cap Z^*$ is nonempty; pick any $y$ in it.

We claim $R(y) = R(x)$. Indeed, both $R(y)$ and $R(x)$ are defined as the intersection of those interiors $\mathrm{int}(T^n_{j,k})$ containing the respective base point. Since $y \in R(x) \subset \mathrm{int}(T^n_{j,k})$ for each $T^n_{j,k}$ appearing in the defining intersection of $R(x)$, the collection of interiors containing $y$ is at least as large. Conversely, if $y \in \mathrm{int}(T^n_{j,k})$ then in particular $y \in T_j \cap T^n_{j,k}$, and since $x \in R(y) \subset T^n_{j,k}$ as well (by the same argument applied to $R(y)$), both collections coincide. Hence $R(y) = R(x)$. The same argument gives $R(y) = R(x')$, so $R(x) = R(x')$.
\end{proof}

The collection $\{R(x) : x \in Z^*\}$ contains only finitely many distinct rectangles by compactness (each is an intersection of finitely many from a finite collection). Let $\mathcal{R} = \{R_1, \ldots, R_m\}$ be the distinct elements.

\textbf{Step 5: Verification of Markov Property.} The key is to show that the Markov property (M3) holds for the refined partition.

\begin{lemma}\label{lem:refined_markov}
If $x, y \in Z^* \cap f^{-1}Z^*$ with $R(x) = R(y)$ and $y \in W^s_\varepsilon(x)$, then $R(f(x)) = R(f(y))$.
\end{lemma}

\begin{proof}
First, $\mathcal{T}(f(x)) = \mathcal{T}(f(y))$ by Lemma \ref{lem:initial_markov}(i): if $f(x) \in T_j$ then $f(y) \in f(W^s(x, T_s)) \subset W^s(f(x), T_j) \subset T_j$ for appropriate $T_s$.

For $f(x), f(y) \in T_j$ and $T_k \cap T_j \neq \emptyset$, we show $f(x)$ and $f(y)$ lie in the same $T^n_{j,k}$. Since $f(y) \in W^s_\varepsilon(f(x))$, we have $W^s(f(y), T_j) = W^s(f(x), T_j)$.

For the unstable direction: suppose $W^u(f(x), T_j) \cap T_k \neq \emptyset$, say $f(z) \in W^u(f(x), T_j) \cap T_k$. Working backwards through the Markov property of the $T$'s, there exists $z' \in W^u(x, T_s) \cap T_t$ for appropriate $T_s \ni x$ and $T_t$ with $f(z') = f(z)$. Since $x, y \in R(x) = R(y)$ are in the same $T^n_{s,t}$, there exists $z'' \in W^u(y, T_s) \cap T_t$. Then $f(z'') \in W^u(f(y), T_j) \cap T_k$.

Thus $f(x)$ and $f(y)$ have the same intersection pattern, giving $R(f(x)) = R(f(y))$.
\end{proof}

From Lemma \ref{lem:refined_markov} and density arguments, the Markov property (M3) follows for the closures $\overline{R_i}$.

\subsection{Diameter Bounds}

The diameter of the Markov partition controls the H\"{o}lder regularity of the coding map. This subsection gives explicit diameter estimates in terms of the shadowing constants and the choice of cover refinement.

\begin{proposition}[Diameter Estimate]\label{prop:diameter_estimate}
The Markov partition $\mathcal{R}$ constructed above satisfies
\begin{equation}
\max_i \mathrm{diam}(R_i) \leq C \cdot \beta
\end{equation}
where $\beta$ is the shadowing accuracy and $C$ depends on the geometry of $\Omega$.
\end{proposition}

\begin{proof}
Each $T_s$ has diameter at most $2\beta$ (points $\theta(q)$ with $q_0 = p_s$ all lie within $\beta$ of $p_s$). The refinement does not increase diameters. The closure adds at most the boundary thickness, which is controlled by continuity.
\end{proof}

\begin{corollary}[Small Diameter Partitions]\label{cor:small_diameter}
For any $\varepsilon > 0$, choosing $\beta$ and $\alpha$ appropriately in the construction yields a Markov partition with $\mathrm{diam}(\mathcal{R}) < \varepsilon$.
\end{corollary}
\section{Symbolic Dynamics}\label{sec:symbolic_dynamics}

This section establishes the symbolic coding of Axiom A diffeomorphisms via Markov partitions, providing the essential bridge between smooth dynamics and the symbolic dynamics developed in Parts~I and~II \cite{Thiam2026a, Thiam2026b}. We prove quantitative bounds on the coding map and characterize the exceptional set where the coding fails to be injective.

\subsection{The Transition Matrix}

Let $\mathcal{R} = \{R_1, \ldots, R_m\}$ be a Markov partition of a basic set $\Omega$ for an Axiom A diffeomorphism $f$.

\begin{definition}[Transition Matrix]
The transition matrix $A = A(\mathcal{R})$ is the $m \times m$ matrix with entries
\begin{equation}
A_{ij} = \begin{cases} 1 & \text{if } \mathrm{int}(R_i) \cap f^{-1}(\mathrm{int}(R_j)) \neq \emptyset, \\ 0 & \text{otherwise.} \end{cases}
\end{equation}
\end{definition}

\begin{proposition}[Transition Matrix Properties]\label{prop:transition_properties}
The transition matrix $A$ satisfies:
\begin{enumerate}
\item[(i)] If $x \in R_i$ and $f(x) \in R_j$, then $A_{ij} = 1$.
\item[(ii)] If $f|_\Omega$ is topologically transitive, then $A$ is irreducible.
\item[(iii)] If $f|_\Omega$ is topologically mixing, then $A$ is aperiodic.
\end{enumerate}
\end{proposition}

\begin{proof}
Part (i): If $x \in R_i$ and $f(x) \in R_j$, then by the Markov property (M3), $f(W^u(x, R_i)) \supset W^u(f(x), R_j)$. Since $W^u(f(x), R_j)$ has nonempty interior in $\Omega$, and $f$ maps interiors to interiors (locally), $\mathrm{int}(R_i) \cap f^{-1}(\mathrm{int}(R_j)) \neq \emptyset$.

Part (ii): Transitivity means for any $R_i, R_j$ there exists $n > 0$ with $f^n(\mathrm{int}(R_i)) \cap \mathrm{int}(R_j) \neq \emptyset$, which implies $(A^n)_{ij} > 0$.

Part (iii): Mixing means for large $n$, $f^n(\mathrm{int}(R_i)) \cap \mathrm{int}(R_j) \neq \emptyset$ for all $i, j$, giving aperiodicity.
\end{proof}

\subsection{The Coding Map}

The coding map $\pi : \Sigma_A \to \Omega$ sends a bi-infinite sequence of symbols to the unique point whose forward and backward orbits track the rectangles indexed by the sequence. We define it via the shadowing lemma and establish its basic properties.

\begin{definition}[Itinerary]
For $x \in \Omega$, the itinerary of $x$ is the bi-infinite sequence $a(x) = (a_i)_{i \in \mathbb{Z}}$ where $a_i \in \{1, \ldots, m\}$ is determined by $f^i(x) \in R_{a_i}$.
\end{definition}

The itinerary is well-defined when $f^i(x) \in \mathrm{int}(R_{a_i})$ for all $i$, i.e., when $x$ avoids the boundaries of all iterates of the partition.

\begin{definition}[Boundary Set]
Define the stable and unstable boundaries of the partition:
\begin{equation}
\partial^s \mathcal{R} = \bigcup_{j=1}^m \partial^s R_j, \quad \partial^u \mathcal{R} = \bigcup_{j=1}^m \partial^u R_j.
\end{equation}
The total boundary is $\partial \mathcal{R} = \partial^s \mathcal{R} \cup \partial^u \mathcal{R}$.
\end{definition}

\begin{proposition}[Boundary Invariance]\label{prop:boundary_invariance}
The stable and unstable boundaries satisfy:
\begin{equation}
f(\partial^s \mathcal{R}) \subset \partial^s \mathcal{R}, \quad f^{-1}(\partial^u \mathcal{R}) \subset \partial^u \mathcal{R}.
\end{equation}
\end{proposition}

\begin{proof}
Let $x \in \partial^s R_i$, so $x \notin \mathrm{int}(W^u(x, R_i))$. By the Markov property, $f(x) \in R_j$ for some $j$ with $A_{ij} = 1$, and $f(W^u(x, R_i)) \supset W^u(f(x), R_j)$.

If $f(x) \notin \partial^s R_j$, then $f(x) \in \mathrm{int}(W^u(f(x), R_j))$, and since $f$ is a local diffeomorphism, $x \in \mathrm{int}(f^{-1}(W^u(f(x), R_j))) \subset \mathrm{int}(W^u(x, R_i))$, contradicting $x \in \partial^s R_i$.

The argument for $\partial^u$ and $f^{-1}$ is analogous, replacing stable slices by unstable slices and $f$ by $f^{-1}$ throughout.
\end{proof}

\begin{definition}[Good Set]
The good set is
\begin{equation}
Y = \Omega \setminus \bigcup_{n \in \mathbb{Z}} f^n(\partial^s \mathcal{R} \cup \partial^u \mathcal{R}).
\end{equation}
This is the set of points whose entire orbit avoids the partition boundary.
\end{definition}

\begin{proposition}[Good Set Properties]\label{prop:good_set}
The good set $Y$ is a dense $G_\delta$ subset of $\Omega$ with full measure for any ergodic invariant measure $\mu$ with $\mu(\partial \mathcal{R}) = 0$.
\end{proposition}

\begin{proof}
Since $\partial \mathcal{R}$ is closed and nowhere dense (the partition elements are proper rectangles), each $f^n(\partial \mathcal{R})$ is closed and nowhere dense. The countable union has empty interior, so $Y$ is dense. Being a countable intersection of open dense sets, $Y$ is a dense $G_\delta$.

For the measure statement, $\mu(\partial \mathcal{R}) = 0$ and $f$-invariance give $\mu(f^n(\partial \mathcal{R})) = 0$ for all $n$, hence $\mu(Y) = 1$.
\end{proof}

\begin{maintheorem}[Symbolic Coding]\label{thm:symbolic_coding}
Let $\Sigma_A \subset \{1, \ldots, m\}^{\mathbb{Z}}$ be the two-sided subshift of finite type defined by $A$. There exists a continuous surjection $\pi : \Sigma_A \to \Omega$ satisfying:
\begin{enumerate}
\item[(i)] $\pi \circ \sigma = f \circ \pi$, where $\sigma$ is the shift on $\Sigma_A$.
\item[(ii)] For each $a \in \Sigma_A$, $\pi(a) = \bigcap_{n \in \mathbb{Z}} f^{-n}(R_{a_n})$ is a single point.
\item[(iii)] $\pi$ is injective on $\pi^{-1}(Y)$, where $Y$ is the good set.
\item[(iv)] For $x \in Y$, $\pi^{-1}(x) = \{a(x)\}$, the itinerary of $x$.
\end{enumerate}
\end{maintheorem}

\begin{proof}
\textbf{Part (ii):} For $a \in \Sigma_A$, define
\begin{equation}
K_n(a) = \bigcap_{j=-n}^{n} f^{-j}(R_{a_j}).
\end{equation}
We show $K_n(a)$ is nonempty and has diameter tending to zero.

By the Markov property, if $a_0 a_1 \cdots a_n$ is an allowed word (i.e., $A_{a_j a_{j+1}} = 1$ for all $j$), then
\begin{equation}
\bigcap_{j=0}^{n} f^{-j}(R_{a_j})
\end{equation}
is a $u$-subrectangle of $R_{a_0}$ (see Definition below). Similarly, the intersection over negative indices is an $s$-subrectangle. The full intersection $K_n(a)$ is the intersection of a $u$-subrectangle and an $s$-subrectangle, hence nonempty.

\begin{definition}[Subrectangles]
A subset $S \subset R_i$ is a $u$-subrectangle if $S$ is a proper rectangle and $W^u(x, S) = W^u(x, R_i)$ for all $x \in S$. An $s$-subrectangle is defined analogously with stable slices.
\end{definition}

The diameter of $K_n(a)$ tends to zero because:
\begin{itemize}
\item[] The stable diameter (diameter within stable manifolds) contracts by $\lambda^n$ under $n$ forward iterates.
\item[] The unstable diameter contracts by $\lambda^n$ under $n$ backward iterates.
\end{itemize}
Thus $\mathrm{diam}(K_n(a)) \leq C\lambda^n \to 0$.

The intersection $\pi(a) = \bigcap_n K_n(a)$ is a single point.

\textbf{Part (i):} For $a \in \Sigma_A$,
\begin{equation}
\pi(\sigma a) = \bigcap_{n \in \mathbb{Z}} f^{-n}(R_{(\sigma a)_n}) = \bigcap_{n \in \mathbb{Z}} f^{-n}(R_{a_{n+1}}) = f\left(\bigcap_{n \in \mathbb{Z}} f^{-n-1}(R_{a_{n+1}})\right) = f(\pi(a)).
\end{equation}

\textbf{Continuity:} If $a^{(k)} \to a$ in $\Sigma_A$, then for each $N$, eventually $a^{(k)}_j = a_j$ for $|j| \leq N$. Thus $\pi(a^{(k)}) \in K_N(a)$ for large $k$, giving $\pi(a^{(k)}) \to \pi(a)$.

\textbf{Surjectivity:} For $x \in Y$, the itinerary $a(x)$ lies in $\Sigma_A$ (since $A_{a_i a_{i+1}} = 1$ by Proposition \ref{prop:transition_properties}(i)), and $\pi(a(x)) = x$. Since $Y$ is dense and $\pi(\Sigma_A)$ is compact, $\pi(\Sigma_A) = \Omega$.

\textbf{Parts (iii) and (iv):} For $x \in Y$, if $\pi(a) = x = \pi(b)$, then $f^n(x) \in R_{a_n} \cap R_{b_n}$ for all $n$. Since $x \in Y$, we have $f^n(x) \in \mathrm{int}(R_{a_n})$, which (combined with the disjoint interiors condition) forces $a_n = b_n$ for all $n$.
\end{proof}\newpage

\begin{figure}[h!]
\centering
\includegraphics[width=\textwidth]{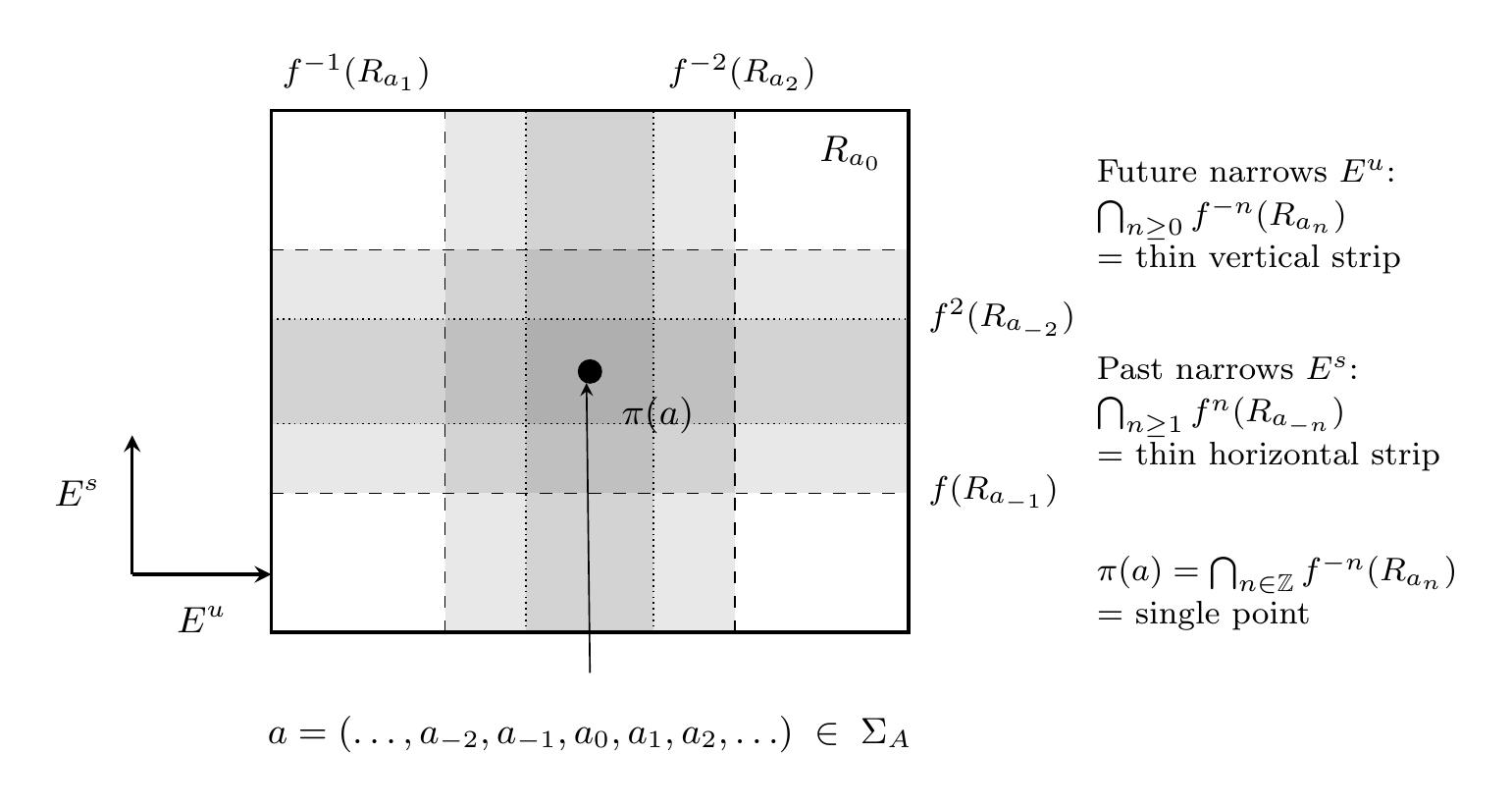}
\caption{The coding map $\pi: \Sigma_A \to \Lambda$ (Main Theorem~\ref{thm:symbolic_coding}). Inside the rectangle $R_{a_0}$, successive preimages $f^{-n}(R_{a_n})$ narrow the unstable direction (vertical strips), while successive forward images $f^n(R_{a_{-n}})$ narrow the stable direction (horizontal strips). Two levels of nesting are shown. The intersection over all $n \in \Z$ is the single point $\pi(a)$.}
\label{fig:coding}
\end{figure}

\subsection{Quantitative Coding Estimates}

This subsection provides the quantitative bounds on the coding map that are used in Parts~IV--VI \cite{Thiam2026d,Thiam2026e,Thiam2026f}: H\"{o}lder continuity with explicit exponent, and control on the exceptional set where the map fails to be injective.

\begin{proposition}[Coding Accuracy]\label{prop:coding_accuracy}
For $a \in \Sigma_A$ and $x = \pi(a)$,
\begin{equation}
d(f^n(x), R_{a_n}) = 0 \quad \text{for all } n \in \mathbb{Z}
\end{equation}
and more precisely,
\begin{equation}
d(x, \pi(b)) \leq C \sum_{j \in \mathbb{Z}} \lambda^{|j|} \mathbf{1}_{a_j \neq b_j} \cdot \mathrm{diam}(\mathcal{R})
\end{equation}
where $\mathrm{diam}(\mathcal{R}) = \max_i \mathrm{diam}(R_i)$.
\end{proposition}

\begin{proof}
The first statement is immediate from $\pi(a) \in R_{a_n}$ (shifted appropriately).

For the second, let $N = \min\{|j| : a_j \neq b_j\}$ be the first disagreement. Then $\pi(a), \pi(b) \in K_N(a) = K_N(b)$ (the intersections agree up to level $N$). The diameter bound $\mathrm{diam}(K_N) \leq C\lambda^N \mathrm{diam}(\mathcal{R})$ gives the estimate.
\end{proof}

\begin{proposition}[Exceptional Set Measure]\label{prop:exceptional_measure}
For any $f$-invariant ergodic measure $\mu$ on $\Omega$,
\begin{equation}
\mu(\Omega \setminus Y) = 0
\end{equation}
provided $\mu$ gives zero measure to the boundary $\partial \mathcal{R}$.
\end{proposition}

\begin{proof}
By Proposition~\ref{prop:good_set}, the hypothesis $\mu(\partial \mathcal{R}) = 0$ together with $f$-invariance of $\mu$ gives $\mu(Y) = 1$. Since $\Omega \setminus Y$ is the complement of a full-measure set, $\mu(\Omega \setminus Y) = 0$.
\end{proof}

\subsection{Entropy via Symbolic Dynamics}

We express the topological entropy of $f$ on a basic set in terms of the spectral radius of the transition matrix of the associated Markov partition. This identification is a direct consequence of the coding map being a measure-preserving factor map.

\begin{theorem}[Entropy Formula]\label{thm:entropy_formula}
For a mixing basic set $\Omega$ with Markov partition $\mathcal{R}$ and transition matrix $A$,
\begin{equation}
h_{\mathrm{top}}(f|_\Omega) = h_{\mathrm{top}}(\sigma|_{\Sigma_A}) = \log \rho(A)
\end{equation}
where $\rho(A)$ is the spectral radius of $A$.
\end{theorem}

\begin{proof}
The coding map $\pi : \Sigma_A \to \Omega$ is a factor map (continuous, surjective, equivariant). Topological entropy does not increase under factor maps, so $h_{\mathrm{top}}(f|_\Omega) \leq h_{\mathrm{top}}(\sigma|_{\Sigma_A})$.

For the reverse inequality, $\pi$ is finite-to-one: each $x \in \Omega$ has at most $m^2$ preimages in $\Sigma_A$ (corresponding to the at most $m$ choices of rectangle for each boundary point at each iterate). By the theorem of  \cite[Theorem~17]{Bowen1971} (see also \cite[Theorem~8.2]{Walters1975}), finite-to-one factor maps between compact metric spaces preserve topological entropy: $h_{\mathrm{top}}(\sigma|_{\Sigma_A}) = h_{\mathrm{top}}(f|_\Omega)$.

The entropy of a subshift of finite type equals $\log \rho(A)$ by standard results (see Part~I \cite{Thiam2026a}).
\end{proof}

\begin{corollary}[Periodic Orbit Counting]\label{cor:periodic_counting}
The number of periodic points of period $n$ satisfies
\begin{equation}
|\{x \in \Omega : f^n(x) = x\}| = \mathrm{tr}(A^n) + O(1)
\end{equation}
where the error accounts for boundary effects.
\end{corollary}

\begin{proof}
Periodic points of $\sigma$ of period $n$ in $\Sigma_A$ are exactly the sequences 
\begin{equation}
(a_0, \ldots, a_{n-1}, a_0, \ldots)\quad\text{with}\quad A_{a_i a_{i+1 \bmod n}} = 1\quad\text {for all} \quad i,
\end{equation}
counted by $\mathrm{tr}(A^n)$. The coding map $\pi$ sends periodic points of $\sigma$ to periodic points of $f$, and $\pi$ is injective on $Y$. Every periodic point of $f|_\Omega$ in $Y$ lifts to a unique periodic point of $\sigma$ of the same period, so $|\mathrm{Per}_n(f) \cap Y| \leq \mathrm{tr}(A^n)$.

For the reverse, each periodic point of $\sigma$ maps to a periodic point of $f$. The map $\pi$ is at most $m^2$-to-one on periodic orbits (where $m = |\mathcal{R}|$), since a periodic orbit of $f$ meeting the boundary $\partial\mathcal{R}$ at some iterate can lie in at most $m$ rectangles at each visited time. The number of periodic orbits of $f$ that have multiple preimages under $\pi$ is bounded by $|\partial\mathcal{R}|^2$ independently of $n$ (since such orbits must pass through the boundary, and the boundary has finite combinatorial complexity). Thus $|\mathrm{Per}_n(f) \cap \Omega| = \mathrm{tr}(A^n) + O(1)$.
\end{proof}

\subsection{Markov Property Extension}

The Markov property is initially established on a dense subset of $\Omega$. This subsection extends it to all points, which is needed for the entropy calculation and for the unique ergodicity statements in Part~IV \cite{Thiam2026d}.

\begin{lemma}[Markov Property for All Points]\label{lem:markov_all_points}
For $x \in R_i$ and $f(x) \in R_j$ with $A_{ij} = 1$,
\begin{equation}
f(W^s(x, R_i)) \subset W^s(f(x), R_j), \quad f(W^u(x, R_i)) \supset W^u(f(x), R_j).
\end{equation}
\end{lemma}

\begin{proof}
This extends the Markov property (M3) from interior points to all points by continuity. The proof follows the final part of the Markov partition construction in Section \ref{sec:markov_partitions}.
\end{proof}

\begin{proposition}[Symbolic Extension]\label{prop:symbolic_extension}
The map $\pi : \Sigma_A \to \Omega$ is a symbolic extension in the sense of  \cite{BoyleDownarowicz2004}: for every $f$-invariant measure $\mu$ on $\Omega$, there exists a $\sigma$-invariant measure $\tilde{\mu}$ on $\Sigma_A$ with $\pi_* \tilde{\mu} = \mu$ and $h_{\tilde{\mu}}(\sigma) = h_\mu(f)$.
\end{proposition}

\begin{proof}
For measures with $\mu(\partial \mathcal{R}) = 0$, the lift $\tilde{\mu}$ is uniquely defined by the itinerary map. The entropy equality follows from the Abramov-Rokhlin formula (see  \cite[Theorem~4.4.2]{BoyleDownarowicz2004}): for a factor map $\pi : (\Sigma_A, \sigma, \tilde\mu) \to (\Omega, f, \mu)$, the entropy satisfies $h_{\tilde\mu}(\sigma) = h_\mu(f) + h_{\tilde\mu}(\sigma | \pi^{-1}\mathcal{B}_\Omega)$, where the conditional (fiber) entropy $h_{\tilde\mu}(\sigma | \pi^{-1}\mathcal{B}_\Omega)$ measures the entropy in the fibers of $\pi$. Since $\pi$ is finite-to-one (at most $m^2$ preimages), the fibers are finite sets, and the conditional entropy of a transformation on a finite set is zero. Thus $h_{\tilde\mu}(\sigma) = h_\mu(f)$.
\end{proof}

\section{Computational Consequences}\label{sec:computational}

We illustrate the quantitative Markov partition construction by computing explicit constants for the Smale horseshoe.

\begin{example}[Smale horseshoe]\label{ex:horseshoe}
Let $f: \mathbb{R}^2 \to \mathbb{R}^2$ be a horseshoe map with hyperbolic set $\Lambda$, contraction rate $\lambda = 1/3$, expansion rate $\mu = 3$, and mixing time $M = 1$ (the transition matrix $A = \begin{pmatrix} 1 & 1 \\ 1 & 1 \end{pmatrix}$ is already positive). The alphabet has $N = 2$ symbols.

\textbf{Adapted metric.} By Theorem~\ref{thm:adapted_metric}, the adapted metric constant is $K(\lambda, c) = c/(1-\lambda) = c \cdot 3/2$ where $c = 1$ in the standard metric. In the adapted metric, $\|Df|_{E^s}\| \leq 1/3$ and $\|Df^{-1}|_{E^u}\| \leq 1/3$ without the prefactor~$c$.

\textbf{Stable manifold size.} By Main Theorem~\ref{thm:stable_manifold}, the local stable manifold $W^s_\varepsilon(x)$ exists for $\varepsilon \leq \varepsilon_0$ where $\varepsilon_0$ depends on the injectivity radius and the Lipschitz constant of $Df$. For the standard horseshoe embedded in the unit square, $\varepsilon_0 \geq (1-\lambda)/(2\|D^2f\|) = (2/3)/(2\|D^2f\|)$.

\textbf{Shadowing constant.} By Main Theorem~\ref{thm:shadowing}, every $\alpha$-pseudo-orbit is $\beta$-shadowed with $\alpha = C^{-1}(1-\lambda)\beta = (2/3C)\beta$. For the horseshoe with $M = 1$, the constant $C$ in the proof satisfies $C \leq (1-\lambda)^{-1} = 3/2$, giving $\alpha \geq (4/9)\beta$.

\textbf{Markov partition diameter.} The natural Markov partition of the horseshoe consists of $m = 2$ rectangles $R_0, R_1$ (the two vertical strips in the invariant set). Their diameters satisfy $\mathrm{diam}(R_i) \leq 2\varepsilon_0/(1-\lambda) = 3\varepsilon_0$, and for the standard horseshoe with the unit square embedding, $\mathrm{diam}(R_i) = 1/3^k$ after $k$ refinements. To achieve coding accuracy $\delta$ (meaning $d(x, \pi(a)) \leq \delta$), the Markov partition must have $\mathrm{diam}(\mathcal{R}) \leq \delta(1-\lambda)/C = 2\delta/3C$. With $C = 3/2$, this gives $\mathrm{diam}(\mathcal{R}) \leq 4\delta/9$.

\textbf{Grid resolution.} For rigorous symbolic coding at accuracy $\delta = 10^{-6}$, the partition diameter must satisfy $\mathrm{diam}(\mathcal{R}) \leq 4.4 \times 10^{-7}$. Since $\mathrm{diam}(\mathcal{R}) = (1/3)^k$ after $k$ refinements, we need $(1/3)^k \leq 4.4 \times 10^{-7}$, i.e., $k \geq \lceil 6\log 10/\log 3 \rceil + 1 = 14$. The refined partition has $2^{14} = 16384$ rectangles.

This computation demonstrates that the quantitative constants from our construction, while not optimal for specific systems, give computable bounds that can be used for rigorous numerical verification.
\end{example}

%\section{Conclusion}
%
%This Part establishes the geometric-topological foundations for the thermodynamic formalism of Axiom A systems: the quantitative theory of hyperbolic splittings, the Stable Manifold Theorem with explicit H\"{o}lder estimates, the canonical coordinates and local product structure, the Spectral Decomposition Theorem, the shadowing lemma, and the construction of Markov partitions yielding the coding map $\pi: \Sigma_A \to \Lambda$. The coding map is the bridge that allows the spectral theory of Part~I \cite{Thiam2025a} and the variational theory of Part~II \cite{Thiam2025b} to be transferred to smooth dynamics. Part~IV of this series uses this coding to develop the transfer operator theory, construct SRB measures, establish the Gibbs Equivalence Theorem, and prove structural stability with quantitative H\"{o}lder bounds on conjugacies.

\section{Conclusion}\label{sec:conclusion} 

This Part establishes the geometric-topological foundations for the thermodynamic formalism of Axiom~A systems through five main theorems proved with explicit constants throughout. Main Theorem~\ref{thm:stable_manifold} provides the Stable Manifold Theorem via the backward graph transform, with a complete fiber-contraction argument yielding $C^r$ regularity, exponential contraction along stable manifolds, and H\"{o}lder dependence on the base point with exponent computed explicitly from the hyperbolicity data. Main Theorem~\ref{thm:spectral_decomposition} establishes the Spectral Decomposition of the nonwandering set into basic sets on which $f$ is topologically transitive, together with the cyclic refinement into mixing components and quantitative mixing rates. Main Theorem~\ref{thm:shadowing} proves the Shadowing Lemma with explicit linear dependence of the shadowing constant on the pseudo-orbit tolerance, which in turn yields the Anosov Closing Lemma and the specification property as quantitative consequences. Main Theorem~\ref{thm:markov_existence} constructs Markov partitions of arbitrarily small diameter on each basic set, with diameter bounds expressed in terms of the shadowing constants and the bracket-map estimates. Main Theorem~\ref{thm:symbolic_coding} produces the coding map $\pi: \Sigma_A \to \Omega$ as a H\"{o}lder continuous factor map from a subshift of finite type, with quantitative control on the exceptional set where $\pi$ fails to be injective and with the identification $h_{\mathrm{top}}(f|_\Omega) = \log\rho(A)$.

Together, these five theorems provide the canonical coordinates, the local product structure, the shadowing theory, and the symbolic coding that constitute the geometric bridge connecting the symbolic thermodynamic formalism of Part~I \cite{Thiam2026a} and the convex-analytic theory of Part~II \cite{Thiam2026b} to smooth dynamics. All constants are expressed in terms of the contraction rate $\lambda$, the H\"{o}lder exponent of $Df$, the manifold dimension $d$, and the injectivity radius of $M$, so the full chain of estimates can be tracked through subsequent parts.

Part~IV \cite{Thiam2026d} of this series uses this coding to develop the transfer operator theory on manifolds, construct SRB measures, establish the Gibbs Equivalence Theorem, and prove structural stability with quantitative H\"{o}lder bounds on conjugacies. Parts~V  and~VI \cite{Thiam2026e, Thiam2026f} develop the statistical limit theorems and the structural consequences for higher-dimensional systems.

\subsection*{Open Problems}

\begin{enumerate}
\item[] \textbf{Optimal Markov partition diameter.} Our construction yields Markov partitions with diameter $O(\delta)$ for any $\delta < \delta_0$. Is there a canonical ``optimal'' Markov partition that minimizes some functional (e.g., the number of rectangles for a given coding accuracy)?

\item[] \textbf{Markov partitions for non-uniformly hyperbolic systems.}  \cite{Sarig2013} constructed countable Markov partitions for surface diffeomorphisms with positive entropy. Can the quantitative bounds in our construction be extended to this setting?

\item[] \textbf{Effective hyperbolicity verification.} Given a $C^2$ diffeomorphism $f$ and a compact invariant set $\Lambda$, is there an algorithm to verify hyperbolicity and compute the constants $(\lambda, C, M)$ to prescribed accuracy? The cone field construction suggests a computational approach, but rigorous error bounds for numerical verification are not developed.
\end{enumerate}

\subsection*{Acknowledgments}

The author is grateful to Stefano Luzzatto for supervision during the ICTP Postgraduate Diploma in Mathematics at the International Centre for Theoretical Physics, Trieste, Italy (2013), during which the author worked through Bowen's monograph.

\appendix

\section{Technical Proofs and Supplementary Material}\label{app:technical}

This appendix contains detailed proofs of technical results stated in the main text, along with supplementary material on functional analysis and geometric measure theory.

\subsection{Stable Manifold Theorem: Detailed Estimates}\label{app:stable_manifold_details}

We provide the complete graph transform analysis for the Stable Manifold Theorem (Theorem \ref{thm:stable_manifold}).

\subsubsection{Setup and Notation}

Let $\Lambda$ be a hyperbolic set for a $C^r$ diffeomorphism $f : M \to M$ with $r \geq 1$. Fix $x_0 \in \Lambda$ and use the exponential map to identify a neighborhood $U$ of $x_0$ with a neighborhood of $0$ in $T_{x_0}M \cong \mathbb{R}^d$. Under this identification:
\begin{equation}
T_{x_0}M = E^s_{x_0} \oplus E^u_{x_0} \cong \mathbb{R}^{k_s} \times \mathbb{R}^{k_u}
\end{equation}
where $k_s = \dim E^s$ and $k_u = \dim E^u$.

In these coordinates, $f$ has the form
\begin{equation}
f(s, u) = (As + a(s, u), Bu + b(s, u))
\end{equation}
where $A = Df_{x_0}|_{E^s}$, $B = Df_{x_0}|_{E^u}$, and $a, b$ are the nonlinear terms satisfying:
\begin{align}
a(0, 0) &= 0, \quad Da(0, 0) = 0, \\
b(0, 0) &= 0, \quad Db(0, 0) = 0.
\end{align}

Using an adapted metric, we have $\|A\| \leq \lambda$ and $\|B^{-1}\| \leq \lambda$ for some $\lambda \in (0, 1)$.

\subsubsection{Nonlinear Estimates}

For $\delta > 0$ small, let $B_\delta = \{(s, u) : \|s\| \leq \delta, \|u\| \leq \delta\}$.

\begin{lemma}\label{lem:nonlinear_bounds}
There exists $C_1 > 0$ such that for $(s, u), (s', u') \in B_\delta$:
\begin{align}
\|a(s, u) - a(s', u')\| &\leq C_1 \delta (\|s - s'\| + \|u - u'\|), \\
\|b(s, u) - b(s', u')\| &\leq C_1 \delta (\|s - s'\| + \|u - u'\|).
\end{align}
\end{lemma}

\begin{proof}
Since $a$ is $C^1$ with $Da(0,0) = 0$, we have $\|Da(s,u)\| \leq C_1\|(s,u)\| \leq C_1\delta$ on $B_\delta$. The mean value theorem gives the estimate. The same argument applies to $b$.
\end{proof}

\subsubsection{Graph Transform Analysis}

Let $\mathcal{G} = \mathcal{G}(\delta, K)$ be the space of Lipschitz graphs $\phi : B^s_\delta \to B^u_\delta$ with $\phi(0) = 0$ and $\|\phi\|_{\mathrm{Lip}} \leq K$.

\begin{lemma}[Inverse Function for Stable Component]\label{lem:inverse_stable}
For $\phi \in \mathcal{G}$ and $\delta, K$ sufficiently small, the map $s \mapsto As + a(s, \phi(s))$ is a contraction from $B^s_\delta$ to $B^s_{\lambda\delta + C_1\delta^2(1+K)}$. For each $s' \in B^s_{\delta'}$ with $\delta' = \lambda\delta(1 + C_1\delta(1+K)/\lambda)^{-1}$, there exists a unique $s \in B^s_\delta$ with $As + a(s, \phi(s)) = s'$.
\end{lemma}

\begin{proof}
The derivative of $s \mapsto As + a(s, \phi(s))$ is $A + D_s a + D_u a \cdot D\phi$, with norm at most $\lambda + C_1\delta(1 + K) < 1$ for small $\delta$. The contraction mapping theorem gives existence and uniqueness of the inverse.
\end{proof}

Define the graph transform $\Gamma : \mathcal{G} \to \mathcal{G}$ as follows. For $\phi \in \mathcal{G}$ and $s' \in B^s_{\delta'}$, let $s = s(\phi, s')$ be the unique solution from Lemma \ref{lem:inverse_stable}, and set
\begin{equation}
\Gamma(\phi)(s') = B\phi(s) + b(s, \phi(s)).
\end{equation}

\begin{proposition}\label{prop:gamma_welldefined}
For $\delta$ and $K$ chosen appropriately (depending on $\lambda$ and $C_1$), $\Gamma : \mathcal{G} \to \mathcal{G}$ is well-defined.
\end{proposition}

\begin{proof}
We verify:
\begin{enumerate}
\item[(i)] $\Gamma(\phi)(0) = 0$: When $s' = 0$, the solution is $s = 0$ (since $A \cdot 0 + a(0, \phi(0)) = a(0, 0) = 0$), giving $\Gamma(\phi)(0) = B\phi(0) + b(0, 0) = 0$.

\item[(ii)] $\|\Gamma(\phi)(s')\| \leq \delta$: We have
\begin{equation}
\|\Gamma(\phi)(s')\| \leq \|B\| \cdot \|\phi(s)\| + \|b(s, \phi(s))\| \leq \lambda^{-1} K\delta + C_1\delta^2
\end{equation}
which is at most $\delta$ if $K \leq \lambda(1 - C_1\delta)/\delta$.

\item[(iii)] $\|\Gamma(\phi)\|_{\mathrm{Lip}} \leq K$: Differentiating $\Gamma(\phi)(s') = B\phi(s(s')) + b(s(s'), \phi(s(s')))$ with respect to $s'$ and using $\|Ds/Ds'\| \leq (\lambda - C_1\delta)^{-1}$, the Lipschitz constant is at most $(\|B\|K + C_1\delta)/(\lambda - C_1\delta) \leq (\lambda^{-1}K + C_1\delta)/(\lambda - C_1\delta)$, which is at most $K$ for small $\delta$ and appropriate $K$.
\end{enumerate}
\end{proof}

\subsubsection{Contraction Property}

\begin{proposition}\label{prop:gamma_contraction}
There exists $\theta \in (0, 1)$ such that for $\phi, \psi \in \mathcal{G}$:
\begin{equation}
\|\Gamma^{-1}(\phi) - \Gamma^{-1}(\psi)\|_{C^0} \leq \theta \|\phi - \psi\|_{C^0}.
\end{equation}
\end{proposition}

\begin{proof}
The forward graph transform $\Gamma$ (Definition in Section~\ref{sec:stable_manifolds}) satisfies $\|\Gamma(\phi) - \Gamma(\psi)\|_{C^0} \leq (\lambda^{-1} + C_2\delta)\|\phi - \psi\|_{C^0}$, which gives \emph{expansion} since $\lambda^{-1} > 1$. Thus $\Gamma$ is unsuitable for a contraction argument on stable manifolds.

The backward graph transform $\Gamma^{-1}$ (Definition~\ref{def:backward_graph_transform}) is contractive because it involves $B^{-1}$ with $\|B^{-1}\| \leq \lambda < 1$. We provide the detailed estimates.

Fix $s \in B^s_\delta$. Let $u_1 = \Gamma^{-1}(\phi)(s)$ and $u_2 = \Gamma^{-1}(\psi)(s)$, and let $s'_i = As + a(s, u_i)$. From equation \eqref{eq:backward_implicit}:
\begin{align}
u_1 - u_2 &= B^{-1}[\phi(s'_1) - \psi(s'_2)] - B^{-1}[b(s, u_1) - b(s, u_2)].
\end{align}

\textbf{Nonlinear remainder:} $\|B^{-1}[b(s, u_1) - b(s, u_2)]\| \leq \lambda C_1\delta\|u_1 - u_2\|$.

\textbf{Main term:}
\begin{align}
\|\phi(s'_1) - \psi(s'_2)\| &\leq \|\phi(s'_1) - \psi(s'_1)\| + \|\psi(s'_1) - \psi(s'_2)\| \\
&\leq \|\phi - \psi\|_{C^0} + K\|s'_1 - s'_2\|.
\end{align}
Since $\|s'_1 - s'_2\| = \|a(s, u_1) - a(s, u_2)\| \leq C_1\delta\|u_1 - u_2\|$:
\begin{equation}
\|\phi(s'_1) - \psi(s'_2)\| \leq \|\phi - \psi\|_{C^0} + KC_1\delta\|u_1 - u_2\|.
\end{equation}

\textbf{Combining:}
\begin{equation}
\|u_1 - u_2\| \leq \lambda\|\phi - \psi\|_{C^0} + \lambda C_1\delta(K + 1)\|u_1 - u_2\|.
\end{equation}
Solving:
\begin{equation}
\|u_1 - u_2\| \leq \frac{\lambda}{1 - \lambda C_1\delta(K+1)}\|\phi - \psi\|_{C^0} = \theta\|\phi - \psi\|_{C^0}
\end{equation}
with $\theta = \lambda/(1 - \lambda C_1\delta(K+1)) < 1$ for $\delta < (1-\lambda)/(\lambda C_1(K+1))$.
\end{proof}

\subsection{H\"{o}lder Spaces on Manifolds}\label{app:holder_spaces}

This appendix subsection collects basic facts about the H\"{o}lder function spaces $C^\alpha(M)$ used throughout the Part: their Banach-algebra structure, the Arzel\`{a}-Ascoli compactness result, and the elementary inclusions between spaces of different exponents.

\begin{definition}
For a compact Riemannian manifold $(M, g)$ and $\alpha \in (0, 1]$, the H\"{o}lder space $C^\alpha(M)$ consists of functions $\phi : M \to \mathbb{R}$ with finite norm
\begin{equation}
\|\phi\|_\alpha = \|\phi\|_\infty + |\phi|_\alpha, \quad |\phi|_\alpha = \sup_{x \neq y} \frac{|\phi(x) - \phi(y)|}{d(x, y)^\alpha}.
\end{equation}
\end{definition}

\begin{proposition}[{\cite[Section~19.1]{KatokHasselblatt1995}}]
$C^\alpha(M)$ is a Banach algebra: $\|\phi\psi\|_\alpha \leq \|\phi\|_\alpha \|\psi\|_\alpha$.
\end{proposition}

\begin{proposition}[Arzel\`{a}-Ascoli]
The inclusion $C^\alpha(M) \hookrightarrow C^0(M)$ is compact.
\end{proposition}

\begin{proposition}
For $\alpha < \beta$, $C^\beta(M) \subset C^\alpha(M)$ with continuous inclusion, and $|\phi|_\alpha \leq \mathrm{diam}(M)^{\beta - \alpha} |\phi|_\beta$.
\end{proposition}

\subsection{Spectral Theory for Quasi-compact Operators}\label{app:spectral}

We recall the spectral decomposition for quasi-compact bounded operators on a Banach space. This is used in Parts~IV--V \cite{Thiam2026d,Thiam2026e}, where the Ruelle transfer operator on $C^\alpha$-manifolds is shown to be quasi-compact.

\begin{proposition}[Spectral Decomposition for Quasi-compact Operators]\label{prop:spectral_decomp_qc}
Let $L : \mathcal{B} \to \mathcal{B}$ be a quasi-compact operator on a Banach space with spectral radius $\rho$ and essential spectral radius $\rho_{\mathrm{ess}} < \rho$. Then:
\begin{enumerate}
\item[(i)] The spectrum in $\{z : |z| > \rho_{\mathrm{ess}}\}$ consists of finitely many eigenvalues of finite multiplicity.
\item[(ii)] If $\lambda_1, \ldots, \lambda_k$ are the eigenvalues with $|\lambda_i| = \rho$, with spectral projections $P_1, \ldots, P_k$, then
\begin{equation}
L^n = \sum_{i=1}^k \lambda_i^n P_i + R^n
\end{equation}
where $\|R^n\| \leq C \theta^n$ for some $\theta < \rho$ and $C > 0$.
\end{enumerate}
\end{proposition}

\begin{proof}
This is a consequence of the Riesz functional calculus. The projection onto the peripheral spectrum (eigenvalues of maximum modulus) is
\begin{equation}
P = \frac{1}{2\pi i} \oint_{|z| = \rho + \varepsilon} (zI - L)^{-1} \, dz
\end{equation}
for small $\varepsilon > 0$. The operator $R = L - \sum_i \lambda_i P_i$ has spectral radius strictly less than $\rho$.
\end{proof}

\subsection{Measure-Theoretic Lemmas}\label{app:measure}

We record two standard measure-theoretic lemmas on disintegration of measures. Both are used in the coding theory of Section~\ref{sec:symbolic_dynamics} to handle measures on fibers of the coding map.

\begin{lemma}[Disintegration of Measures {\cite[Chapter~5]{CornfeldFominSinai1982}}]
Let $\mu$ be a probability measure on $X$ and $\pi : X \to Y$ a measurable map with $\nu = \pi_* \mu$. There exists a family of probability measures $\{\mu_y\}_{y \in Y}$ on $X$ such that:
\begin{enumerate}
\item[(i)] $\mu_y(\pi^{-1}(y)) = 1$ for $\nu$-almost every $y$.
\item[(ii)] For every measurable $A \subset X$, $y \mapsto \mu_y(A)$ is measurable.
\item[(iii)] $\mu(A) = \int_Y \mu_y(A) \, d\nu(y)$.
\end{enumerate}
\end{lemma}

\begin{lemma}[Rokhlin's Disintegration Theorem {\cite[Theorem~5.1.11]{CornfeldFominSinai1982}}]
If $\mu$ is a Borel probability measure on a standard Borel space $X$ and $\mathcal{P}$ is a measurable partition of $X$, then the conditional measures $\mu_P$ on partition elements $P \in \mathcal{P}$ exist and are unique up to $\mu$-null sets.
\end{lemma}

\subsection{Geometric Measure Theory}\label{app:geometric}

We state the definition of Hausdorff dimension and recall Bowen's equation, which links the Hausdorff dimension of a conformal repeller to the zero of an associated pressure function. These tools are used in the dimension-theoretic consequences developed in Parts~V--VI \cite{Thiam2026e,Thiam2026f}.

\begin{definition}[Hausdorff Dimension]
For a metric space $(X, d)$ and $s \geq 0$, the $s$-dimensional Hausdorff measure is
\begin{equation}
\mathcal{H}^s(A) = \lim_{\delta \to 0} \inf\left\{ \sum_i (\mathrm{diam} U_i)^s : A \subset \bigcup_i U_i, \mathrm{diam} U_i < \delta \right\}.
\end{equation}
The Hausdorff dimension is $\dim_H(A) = \inf\{s : \mathcal{H}^s(A) = 0\} = \sup\{s : \mathcal{H}^s(A) = \infty\}$.
\end{definition}

\begin{proposition}[Bowen's Equation {\cite{Bowen1979}}]\label{prop:bowen_equation}
For a conformal repeller $\Lambda$ with expansion rate $\lambda(x) = |Df_x|$ and topological pressure $P(s) = P(-s\log|Df|)$, the Hausdorff dimension satisfies the Bowen equation $P(\dim_H(\Lambda)) = 0$.
\end{proposition}

\noindent The proof is given in Bowen's original paper \cite{Bowen1979}.

\subsection{Proof of Proposition \ref{prop:cone_criterion}}\label{app:cone_proof}

We prove the cone criterion for hyperbolicity in detail.

\begin{proof}[Proof of Proposition \ref{prop:cone_criterion}]
Define the invariant bundles by
\begin{align}
E^u_x &= \bigcap_{n \geq 0} Df^n_{f^{-n}(x)}(\mathcal{C}^u_{f^{-n}(x)}), \\
E^s_x &= \bigcap_{n \geq 0} Df^{-n}_{f^n(x)}(\mathcal{C}^s_{f^n(x)}).
\end{align}

\textbf{Step 1: The intersections are nonempty.} Each $Df^n(\mathcal{C}^u)$ is a cone strictly contained in the previous one (by condition (i)). In finite dimensions, a nested sequence of closed cones with strictly decreasing aperture converges to a linear subspace.

\textbf{Step 2: The intersections are linear subspaces.} The limit of cones with aperture tending to zero is a linear subspace. The dimension is determined by the original cone dimensions.

\textbf{Step 3: Invariance.} $Df_x(E^u_x) = E^u_{f(x)}$ follows from the definition and $Df$-invariance of the cone sequence.

\textbf{Step 4: Expansion/contraction.} Condition (iii) gives $\|Df^n_x v\| \geq \lambda^{-n}\|v\|$ for $v \in \mathcal{C}^u_x$. In the limit, this applies to $v \in E^u_x$. Rewriting: $\|Df^{-n}_x w\| \leq \lambda^n \|w\|$ for $w \in E^u_x$, which is the hyperbolic expansion condition.

\textbf{Step 5: Transversality.} $E^s_x \cap E^u_x = \{0\}$ because $\mathcal{C}^s_x \cap \mathcal{C}^u_x = \{0\}$ by assumption, and the invariant subspaces are contained in the respective cones.

\textbf{Step 6: Spanning.} $E^s_x \oplus E^u_x = T_x M$ by dimension count: each has dimension equal to the dimension of the corresponding cone's axis, and these sum to $\dim M$.
\end{proof}

\end{document}